\documentclass[12pt]{amsart}
\usepackage{amsmath}
\usepackage{amssymb}
\usepackage{xcolor}
\usepackage{tikz}
\usetikzlibrary{cd}
\usetikzlibrary{knots}
\usetikzlibrary{calc}

\tikzset{string/.style={ultra thick}}
\tikzset{smallstring/.style={thick,scale=0.75,every node/.style={transform shape}}}
\tikzset{
    triple/.style args={[#1] in [#2] in [#3]}{
        #1,preaction={preaction={draw,#3},draw,#2}
    }
}
\tikzset{
    quadruple/.style args={[#1] in [#2] in [#3] in [#4]}{
        #1,preaction={preaction={preaction={draw,#4},draw,#3}, draw,#2}
    }
} 
\tikzset{
	super thick/.style={line width=3pt},
	more thick/.style={line width=1pt},
}
\newcommand{\iV}{}

\usepackage{graphicx}
\usepackage{fullpage}
\usepackage[T1]{fontenc}
\usepackage[final]{microtype}
\usepackage{libertine}
\usepackage[libertine]{newtxmath}
\usepackage{ifthen}
\usepackage[all]{xy}

\usepackage{xcolor}
\definecolor{dark-red}{rgb}{0.7,0.25,0.25}
\definecolor{dark-blue}{rgb}{0.15,0.15,0.55}
\definecolor{medium-blue}{rgb}{0,0,0.65}
\definecolor{DarkGreen}{RGB}{0,150,0}

\usepackage[pdftex,plainpages=false,hypertexnames=false,pdfpagelabels,breaklinks]{hyperref}
\hypersetup{
   colorlinks, linkcolor={purple},
   citecolor={medium-blue}, urlcolor={medium-blue}
}
\usepackage[backend=biber,style=alphabetic,doi=false,isbn=false,url=false,minnames=6,maxnames=6]{biblatex}
\renewbibmacro{in:}{%
  \ifentrytype{article}{}{\printtext{\bibstring{in}\intitlepunct}}}
\renewbibmacro*{volume+number+eid}{%
  \printtext{vol.}
  \printfield{volume}%
  \setunit*{\addnbspace}
  \printfield{number}%
  \setunit{\addcomma\space}%
  \printfield{eid}}
\DeclareFieldFormat[article]{number}{\mkbibparens{#1}}
\usepackage{silence}
\newcommand{\doi}[1]{\href{https://dx.doi.org/#1}{{\small DOI:#1}}}

\newcommand{\googlebooks}[1]{(preview at \href{https://books.google.com/books?id=#1}{google books})}

\newcommand{\numdam}[1]{}

\theoremstyle{plain}
\newtheorem{prop}{Proposition}[section]

\newtheorem{thm}[prop]{Theorem}
\newtheorem{lem}[prop]{Lemma}
\newtheorem{cor}[prop]{Corollary}
\numberwithin{equation}{section}

\theoremstyle{remark}
\newtheorem{example}[prop]{Example}

\newtheorem{remark}[prop]{Remark}

\theoremstyle{definition}
\newtheorem{defn}[prop]{Definition}         

\newcommand{\sslash}{\mathbin{/\mkern-6mu/}}

\DeclareMathOperator{\Ad}{Ad}
\DeclareMathOperator{\ev}{ev}
\DeclareMathOperator{\coev}{coev}
\DeclareMathOperator{\Hom}{Hom}
\DeclareMathOperator{\Nat}{Nat}
\DeclareMathOperator{\Irr}{Irr}
\DeclareMathOperator{\mate}{mate}

\DeclareMathOperator{\Spin}{Spin}
\DeclareMathOperator{\Tr}{Tr}

\newcommand{\dslash}[1]{_{{\sslash}#1}}
\newcommand{\id}{\mathbf{1}}
\newcommand{\sVec}{{\mathsf {sVec}}}
\renewcommand{\Vec}{{\mathsf {Vec}}}
\newcommand{\Rep}{{\mathsf {Rep}}}

\def\semicolon{;}
\def\applytolist#1{
    \expandafter\def\csname multi#1\endcsname##1{
        \def\multiack{##1}\ifx\multiack\semicolon
            \def\next{\relax}
        \else
            \csname #1\endcsname{##1}
            \def\next{\csname multi#1\endcsname}
        \fi
        \next}
    \csname multi#1\endcsname}

\def\calc#1{\expandafter\def\csname c#1\endcsname{{\mathcal #1}}}
\applytolist{calc}QWERTYUIOPLKJHGFDSAZXCVBNM;
\def\bbc#1{\expandafter\def\csname bb#1\endcsname{{\mathbb #1}}}
\applytolist{bbc}QWERTYUIOPLKJHGFDSAZXCVBNM;
\def\bfc#1{\expandafter\def\csname bf#1\endcsname{{\mathbf #1}}}
\applytolist{bfc}QWERTYUIOPLKJHGFDSAZXCVBNM;

\makeatletter
\newlength{\L@UnitsRaiseDisplaystyle}
\newlength{\L@UnitsRaiseTextstyle}
\newlength{\L@UnitsRaiseScriptstyle}
\DeclareRobustCommand*{\@UnitsNiceFrac}[3][]{%
  \ifthenelse{\boolean{mmode}}{%
    \settoheight{\L@UnitsRaiseDisplaystyle}{%
      \ensuremath{\displaystyle#1{M}}%
    }%
    \settoheight{\L@UnitsRaiseTextstyle}{%
      \ensuremath{\textstyle#1{M}}%
    }%
    \settoheight{\L@UnitsRaiseScriptstyle}{%
      \ensuremath{\scriptstyle#1{M}}%
    }%
    \settoheight{\@tempdima}{%
      \ensuremath{\scriptscriptstyle#1{M}}%
    }%
    \addtolength{\L@UnitsRaiseDisplaystyle}{%
      -\L@UnitsRaiseScriptstyle%
    }%
    \addtolength{\L@UnitsRaiseTextstyle}{%
      -\L@UnitsRaiseScriptstyle%
    }%
    \addtolength{\L@UnitsRaiseScriptstyle}{-\@tempdima}%
    \mathchoice
      {%
        \raisebox{\L@UnitsRaiseDisplaystyle}{%
          \ensuremath{\scriptstyle#1{#2}}%
        }%
      }%
      {%
        \raisebox{\L@UnitsRaiseTextstyle}{%
          \ensuremath{\scriptstyle#1{#2}}%
        }%
      }%
      {%
        \raisebox{\L@UnitsRaiseScriptstyle}{%
          \ensuremath{\scriptscriptstyle#1{#2}}%
        }%
      }%
      {%
        \raisebox{\L@UnitsRaiseScriptstyle}{%
          \ensuremath{\scriptscriptstyle#1{#2}}%
        }%
      }%
    \mkern-2mu{\sslash}\mkern-1mu%
    \bgroup
      \mathchoice
        {\scriptstyle}%
        {\scriptstyle}%
        {\scriptscriptstyle}%
        {\scriptscriptstyle}%
      #1{#3}%
    \egroup
  }%
  {%
    \settoheight{\L@UnitsRaiseTextstyle}{#1{M}}%
    \settoheight{\@tempdima}{%
      \ensuremath{%
        \mbox{\fontsize\sf@size\z@\selectfont#1{M}}%
      }%
    }%
    \addtolength{\L@UnitsRaiseTextstyle}{-\@tempdima}%
    \raisebox{\L@UnitsRaiseTextstyle}{%
      \ensuremath{%
        \mbox{\fontsize\sf@size\z@\selectfont#1{#2}}%
      }%
    }%
    \ensuremath{\mkern-2mu}{\sslash}\ensuremath{\mkern-1mu}%
    \ensuremath{%
      \mbox{\fontsize\sf@size\z@\selectfont#1{#3}}%
    }%
  }%
}

\DeclareRobustCommand*{\nicefrac}{\@UnitsNiceFrac}%
\makeatother

\begin{document}

\title{Monoidal categories enriched in braided monoidal categories}
\author{Scott Morrison and David Penneys }
\date{}
\begin{abstract}
We introduce the notion of a monoidal category enriched in a braided monoidal category $\cV$. 
We set up the basic theory, and prove a classification result in terms of braided oplax monoidal functors to the Drinfeld center of some monoidal category $\cT$.

Even the basic theory is interesting; it shares many characteristics with the
theory of monoidal categories enriched in a symmetric monoidal
category, but lacks some features. 
Of particular note, there is no cartesian product of braided-enriched categories, and the natural transformations do not form a 2-category, but rather satisfy a braided interchange relation.

Strikingly, our classification is slightly more general than what one might have anticipated in terms of strong monoidal functors $\cV\to Z(\cT)$. 
We would like to understand this further; in a future paper we show that the functor is strong
if and only if the enriched category is `complete' in a certain sense.
Nevertheless it remains to understand what non-complete enriched categories
may look like.

One should think of our construction as a generalization of de-equivariantization, which takes a strong monoidal functor $\Rep(G) \to Z(\cT)$ for some finite group $G$ and a monoidal category $\cT$, and produces a new monoidal category $\cT \dslash G$.
In our setting, given any braided oplax monoidal functor $\cV \to Z(\cT)$, for any braided $\cV$, we produce $\cT \dslash \cV$: this is not usually an `honest' monoidal category, but is instead $\cV$-enriched. 
If $\cV$ has a braided lax monoidal functor to $\Vec$, we can use this to reduce the enrichment to $\Vec$, and this recovers de-equivariantization as a special case.

\end{abstract} 
\maketitle

\section{Introduction}

While the symmetries of classical mathematical objects form groups, the symmetries of  `quantum' mathematical objects such as subfactors and quantum groups form more general objects which are best axiomatized as tensor categories.
In turn, tensor categories have connections to many branches of mathematics, including representation theory, topological and conformal field theory, and quantum information.

Early in the study of monoidal categories, Eilenberg and Kelly defined the notion of a category enriched in a given monoidal category $\cV$ \cite{MR0225841} (see also \cite{MR2177301}).
An ordinary category has objects and hom sets, while a $\cV$-enriched category $\cC$ has objects, and for every $a,b\in \cC$,  an associated hom object $\cC(a\to b)\in \cV$.
The $\cV$-enriched category also comes with distinguished identity elements $j_c\in \cV(1_\cV \to \cC(c\to c))$ for every $c\in \cC$, and a composition morphism $-\circ_\cC - : \cC(a\to b)\cC(b\to c) \to \cC(a\to c)$ for every $a,b,c\in \cC$ which must satisfy certain compatibility and associativity axioms.
From this perspective, we may think of an ordinary category as enriched in the monoidal category $\sf Set$, and a linear category as enriched in the monoidal category $\sf Vec$.

Braided monoidal categories were introduced in \cite{MR1250465}. 
They play an essential role as algebraic ingredients in 3-dimensional quantum topology.
This article introduces the notion of a \emph{monoidal category enriched in a 
braided monoidal category}.
Linear monoidal categories are of course the case when the enriching category $\cV={\sf Vec}$.
The special case when the enriching category $\cV = \sVec$ has received some recent attention \cite{1603.05928,1606.03466,1603.09294}.
We will expand more on related work in \S \ref{sec:RelatedWork} below.

We believe the notion of a monoidal category enriched in a symmetric closed
monoidal category is well-known to experts in the field.
However, the fact that the enriching category need only be braided, not necessarily symmetric, has not been pursued.
The main difficulty in defining the monoidal structure of a $\cV$-monoidal category is the exchange relation from an ordinary monoidal category.
If $f_1\in \cC(a\to b)$, $f_2\in \cC(b\to c)$, $g_1\in \cC(d\to e)$, and $g_2\in \cC(e\to f)$, we have
$$
(f_1\otimes g_1)\circ (f_2\otimes g_2) = (f_1\circ f_2)\otimes (g_1\circ g_2).
$$
Throughout this article, we choose to read composition left to right so that we do not need to change the order of objects: that is, composition is a map $\cC(a \to b)\cC(b \to c) \to \cC(a \to c)$.
Indeed, we see the two morphisms $g_1$ and $f_2$ are transposed in the above relation, which tells us that the enriching monoidal category should be braided.
As enriched categories do not have morphisms, but rather hom objects, we replace the ordinary exchange relation with the following \emph{braided interchange relation}, which we express using string diagrams for morphisms in $\cV$:
\begin{equation}
\label{eq:BraidedInterchange}
\begin{tikzpicture}[baseline=50,smallstring]
\node (a-b) at (0,0) {$\cC(a \to b)$};
\node (d-e) at (2,0) {$\cC(d \to e)$};
\node (b-c) at (4,0) {$\cC(b \to c)$};
\node (e-f) at (6,0) {$\cC(e \to f)$};
\node[draw,rectangle] (t1) at (1,2) {$\quad -\otimes-\quad $};
\draw (a-b) to[in=-90,out=90] (t1.-135);
\draw (d-e) to[in=-90,out=90] (t1.-45);
\node[draw,rectangle] (t2) at (5,2) {$\quad -\otimes-\quad $};
\draw (b-c) to[in=-90,out=90] (t2.-135);
\draw (e-f) to[in=-90,out=90] (t2.-45);
\node[draw,rectangle] (c) at (3,4) {$\quad - \circ - \quad $};
\draw (t1) to[in=-90,out=90] node[left=13pt] {$\cC(ad\to be)$} (c.-135);
\draw (t2) to[in=-90,out=90] node[right=13pt] {$\cC(be\to cf)$} (c.-45);
\node (r) at (3,5.5) {$\cC(ad \to cf)$};
\draw (c) -- (r);
\end{tikzpicture}
=
\begin{tikzpicture}[baseline=50,smallstring]
\node (a-b) at (0,0) {$\cC(a \to b)$};
\node (d-e) at (2,0) {$\cC(d \to e)$};
\node (b-c) at (4,0) {$\cC(b \to c)$};
\node (e-f) at (6,0) {$\cC(e \to f)$};
\node[draw,rectangle] (t1) at (1,2) {$\quad -\circ-\quad $};
\node[draw,rectangle] (t2) at (5,2) {$\quad -\circ-\quad $};
\draw (a-b) to[in=-90,out=90] (t1.-135);
\draw[knot] (b-c) to[in=-90,out=90] (t1.-45);
\draw[knot] (d-e) to[in=-90,out=90] (t2.-135);
\draw (e-f) to[in=-90,out=90] (t2.-45);
\node[draw,rectangle] (c) at (3,4) {$\quad - \otimes - \quad $};
\draw (t1) to[in=-90,out=90] node[left=13pt] {$\cC(a\to c)$} (c.-135);
\draw (t2) to[in=-90,out=90] node[right=13pt] {$\cC(d\to f)$} (c.-45);
\node (r) at (3,5.5) {$\cC(ad \to cf)$};
\draw (c) -- (r);
\end{tikzpicture}
.
\end{equation}
We refer the reader to Section \ref{sec:BasicNotions} for the formal definition of a (strict) $\cV$-monoidal category.

In this article, we classify monoidal categories enriched in $\cV$ in terms of braided oplax monoidal functors from $\cV$ to the Drinfeld center $Z(\cT)$ of an ordinary monoidal category $\cT$.
(The Drinfeld center was introduced in \cite{MR1107651}.)
Recall that a functor $\cF$ is \emph{oplax monoidal} if there is a family of morphisms $\mu_{u,v}:\cF(uv) \to \cF(u)\cF
(v)$, which need not be isomorphisms, but must merely satisfy naturality and associativity conditions.
We call $\cF$ braided oplax monoidal if $\mu$ also intertwines the braidings (see Proposition \ref{prop:Braided}).

Our main result is:
\begin{thm}
\label{thm:Main}
Let $\cV$ be a braided monoidal category.
There is a bijective correspondence 
\[
\left\{\, 
\parbox{5cm}{\rm Rigid $\cV$-monoidal categories $\cC$, such that $x\mapsto \cC(1_\cC\to x)$ admits a left adjoint}\,\left\}
\,\,\,\,\cong\,\,
\left\{\,\parbox{8cm}{\rm Pairs $(\cT,\cF^{\scriptscriptstyle Z})$ with $\cT$ a rigid monoidal category and $\cF^{\scriptscriptstyle Z}: \cV\to Z(\cT)$ braided oplax monoidal, such that $\cF:=\cF^{\scriptscriptstyle Z}\circ R$ admits a right adjoint}\,\right\}.
\right.\right.
\]
\end{thm}
Here, $R: Z(\cT)\to \cT$ is the forgetful functor, and we use the superscript on $\cF^{\scriptscriptstyle Z}$ to distinguish it from $\cF : \cV \to \cT$.
The notion of rigidity for $\cV$-monoidal categories is introduced in Section \ref{sec:rigidity}.

One could dress this theorem up as an equivalence of 2-categories, but we do not pursue this here.
We also work with a strict notion of $\cV$-monoidal category for convenience.

Theorem \ref{thm:Main} thus gives us a powerful tool to construct $\cV$-monoidal categories.
Some examples of strong monoidal functors $\cF^{\scriptscriptstyle Z}: \cV \to Z(\cT)$ as above are explored in detail in \cite[\S3.3]{1607.06041}.
Additional examples of strong monoidal functors include the presence of a full copy of $\sf Fib$ inside $Z(\Ad(E_8))$ (by \cite[Cor.~4.9]{MR1815993}, see also \cite{Cain-ADE}) and a full copy of $\Ad(SU(3)_3)$ inside $Z(\Ad(4442))$ (using the modular data from \cite{1501.07679}, and the classification from \cite{MR3548123}). 
These examples seem very interesting, and we look forward to studying them in detail.

For any braided $\cV$, there is an uninteresting braided oplax monoidal functor $\lambda^{\scriptscriptstyle Z}: \cV \to Z(\Vec) = \Vec$, obtained as a left adjoint of the strong monoidal inclusion $\Vec \to \cV$. 
Under the correspondence, this just interprets $\cC$ as `trivially enriched' in $\cV$: that is, the morphism objects of the resulting $\cV$-monoidal category are still just vector spaces, but thought of as multiples of the identity object in $\cV$.

To proceed from left to right in Theorem \ref{thm:Main}, from a $\cV$-monoidal category $\cC$ we first extract an ordinary monoidal category $\cC^{\cV}$ (enriched in ${\sf Vec}$) by replacing each $\Hom$ object $\cC(a \to b)$ with the vector space $\cV(\id \to \cC(a \to b))$. 
(See Section \ref{sec:UnderlyingTensorCategory} for more details on $\cC^\cV$.)\footnote{
	This is a special case of a more general construction: given a braided \emph{lax} monoidal functor $\cF: \cV \to \cW$,
	we can turn a $\cV$-monoidal category into a $\cW$-monoidal category by applying $\cF$ to each of the $\Hom$ objects.
	See Section \ref{sec:TransportingEnrichment} for more details on this construction.
	The construction of $\cC^{\cV}$ uses the braided lax monoidal functor $\cV(1_\cV \to -)$.
}
We chose the notation $\cC^\cV$ to hint at the idea of taking fixed points, akin to \emph{equivariantization}. 
In particular, if $\cC$ is a monoidal category enriched in the symmetric monoidal category $\Rep(G)$, for $G$ a finite group, this just means that there is an action of $G$ on the morphisms of $\cC$, and $\cC^{\Rep(G)}$ is the subcategory of $G$-invariant morphisms.
We contruct the functor $\cF :\cV \to \cT$ in Section \ref{sec:Mates}, and show that it lifts to the centre, giving $\cF^{\scriptscriptstyle Z}: \cV\to Z(\cT)$, in Section \ref{sec:extracting-oplax}.

To pass from right to left, we use the right adjoint of $\cF$ together with rigidity to define the hom objects of the $\cV$-enriched category $\cT \dslash \cF$.
The category $\cT \dslash \cF$ has the same objects as $\cT$, and the hom objects are determined by the natural isomorphisms
\begin{align*}
\cV(v \to \cT \dslash \cF(a \to b)) \cong \cT(a \cF(v) \to b).
\end{align*}
We describe this construction in full detail in Section \ref{sec:EnrichingT}.
In the case that $\cV$ is semisimple, we get a more explicit description by
\begin{align*}
\cT \dslash \cF(a \to b) \cong \bigoplus_{\text{$v_i$ simple}} \cT(a \cF(v_i) \to b) v_i.
\end{align*}

The notation $\cT \dslash \cF$ is meant to evoke the feeling that the $\cV$-monoidal category is some type of \emph{quotient} of $\cT$ by $\cF$, akin to \emph{de-equivariantization}.
The usual process of de-equivariantisation begins with a Tannakian subcategory, that is a copy of $\Rep(G)$, for $G$ a finite group, inside $Z(\cT)$, for $\cT$
some monoidal category. This can be viewed as a fully faithful braided strong monoidal functor $\cF: \Rep(G) \to Z(\cT)$.
We can factor de-equivariantisation into two steps:  first applying our main theorem to obtain the $\Rep(G)$-monoidal category $\cT \dslash{\cF}$, and second applying the the fibre functor (the underlying vector space) to each $\Hom$ space. 
In this sense our construction is a generalization of de-equivariantization, although when we `quotient out' by $\cV$ inside $Z(\cT)$, there is in general no subsequent `underlying vector space' for the $\Hom$ objects in $\cV$.

\subsection{Related work}
\label{sec:RelatedWork}

As mentioned earlier, we have seen recent interest in monoidal categories enriched in $\cV = \sVec$.
Brundan and Ellis defined a super tensor category in \cite{1603.05928} (see also \cite[\S 6]{Muger-Oxford}), and Usher worked out many basic properties in \cite{1606.03466}.
Usher also indicated some interesting examples (his Example 6.9) which were earlier announced by Walker in the language of \emph{spin planar algebras}.
Recently, \cite{1603.09294} defines the notion of a fermionic modular tensor category as a pre-modular tensor category whose M\"{u}ger center is $\sf sVec$.
This latter condition has also been called `slightly degenerate' in \cite{MR3022755}.
The article \cite{1603.09294} defines a procedure similar to de-equivariantization which produces super tensor categories from fermionic modular tensor categories.

We would like to acknowledge explicitly the work of Kevin Walker on enrichment for 2-categories and higher categories; although much of this is unpublished, we and others in the field have learnt a great deal from his ideas, disseminated in notes, conversations, and seminars.

We also point out that recent work of Henriques, Penneys, and Tener \cite{1607.06041} introduces
the notion of an \emph{anchored planar algebra} internal to a braided
pivotal tensor category, and show that these are equivalent to braided pivotal
strong (not merely oplax) monoidal functors $\cF^{\scriptscriptstyle Z}:\cV \to Z(\cT)$ for some
pivotal tensor category $\cT$ such that $\cF= \cF^{\scriptscriptstyle Z} \circ R$ admits a right adjoint. 
The functor $\cF^{\scriptscriptstyle Z}$ endows $\cT$ with the structure of a
\emph{module tensor category for $\cV$} as studied in \cite{1509.02937}.
See also \cite{1602.02662} for the related notion of a para planar algebra.

In a similar vein, if we interpret a pivotal braided tensor category as a disklike 3-category, one can obtain an analogous classification for its disklike modules.
In a concurrent article, Morrison and Walker study super tensor categories from the
point of view of $\Spin$-disklike 2-categories, in the sense of \cite{MR2978449}. 
That article will also include many examples of categories with objects of small dimension.

Also connected to this theorem is the MathOverflow question \cite{MO:51783}
which discusses the construction of $\cV$-enriched monoidal categories when $\cV$ is symmetric and closed from braided strong monoidal functors to the Drinfeld centers of monoidal categories.
The activity there reinforces our belief that monoidal categories enriched in symmetric closed monoidal categories are probably known to experts.
Interestingly, our theorem only requires the braided central functor be oplax monoidal, and not strong monoidal.
We only need an oplax functor to pass from right to left, and all we recover when passing from left to right is the oplax structure.

All the examples we know about at this point, however, are either actually strong monoidal functors, or left adjoints of strong monoidal functors (e.g. the `trivially enriched' examples discussed above). 
It would be very interesting to have `genuinely' oplax examples. 

\subsection{Future research}

In a subsequent paper, we will characterize strong monoidality of $\cF: \cV \to Z(\cT)$ in terms of $\cV$-completeness of $\cC$.
A $\cV$-monoidal category $\cC$ is \emph{$\cV$-complete} if there is a
$\cV$-monoidal functor $\widehat{\cV} \to \cC$, where $\widehat{\cV}$ is the \emph{self-enrichment} of $\cV$ described in \S \ref{sec:VEnrichedInV}.
This is the appropriate generalisation of $\Pi$-completeness introduced in \cite{1606.03466} for super tensor
categories, and is the analog of $\cC$ being tensored over $\cV$ in the sense of \cite{MR2177301}.
We will explore $\cV$-completeness of $\cV$-monoidal categories in a followup article.
In particular, we will prove that under the bijective correspondence given in Theorem \ref{thm:Main},  $\cV$-complete fusion categories correspond to braided strong monoidal functors $\cV\to Z(\cT)$ for some rigid monoidal category $\cT$.

Moreover, we will discuss the $\cV$-completion of a $\cV$-monoidal category $\cC$, which generalises the $\Pi$-envelope introduced in \cite{1603.05928} for super-tensor categories.
We will prove that a $\cV$-monoidal category $\cC$ is $\cV$-complete if and only if it is $\cV$-equivalent to its $\cV$-completion.

It would be interesting to see if one could weaken the rigidity assumption in Theorem \ref{thm:Main} to the assumption that the monoidal categories are merely closed.
(For example, this could hopefully improve the proof of Lemma \ref{lem:MateOfId - aFv} below, which appears in Appendix \ref{sec:ProofsAdjunctionsAndMates}.)
As we use rigidity for various other purposes, and as \cite{1509.02937,1607.06041} uses pivotal categories, we are content to remain in the rigid world for now.

In another direction, it seems that we use the fact that the braiding in $\cV$ has an inverse rather infrequently.
Perhaps it is possible to generalise the setting throughout to monoidal categories enriched in a category $\cV$ equipped with a lax braiding 
$u v \to v u$ as in \cite{MR2342829,MR2381533}. 
For now, however, we have no application of such a generalisation, so we have not pursued it.

\section{Basic notions}
\label{sec:BasicNotions}

Suppose $\cV$ is a monoidal category.
We suppress all unitors and associators in $\cV$ to ease the notation.
Tensor products are indicated by juxtaposition, that is, omitting all $\otimes$-symbols, while all compositions are written explicitly with $\circ$.
We write composition left-to-right throughout.

Recall from \cite{MR2177301} that a $\cV$-enriched category $\cC$ associates to each pair $a,b\in \cC$ a hom object $\cC(a\to b)\in \cV$.
For each $a\in \cC$, there is a distinguished identity element $j_a\in \cV(1_\cV \to \cC(a\to a)$. 
For each $a,b,c\in\cC$, there is a distinguished composition morphism $-\circ_\cC - \in \cV(\cC(a\to b)\cC(b\to c) \to \cC(a\to c))$.

These data must satisfy the following two axioms. (We have two options for
describing such axioms, either as commutative diagrams or as string diagrams \cite{MR1113284}
in $\cV$. Throughout this introduction we use both, to ensure all readers
find something they are comfortable with; later in the paper we use whichever
is most convenient.)
\begin{itemize}
\item
\textbf{
(identity)
}
For all $a,b\in \cC$, $(j_a \id_{\cC(a\to b)})\circ(-\circ_\cC -) = \id_{\cC
(a\to b)}$ and $(\id_{\cC(a\to b)j_b})\circ (-\circ_\cC -) = \id_{\cC(a\to b)}$:
\begin{equation*}
\begin{tikzcd}[column sep=tiny]
	& \cC(a \to a)\cC(a\to b) \ar[dd, "-\circ_\cC -"] & \\
\cC(a\to b) \ar[ur, "j_a \id_{\cC(a\to b)}"]
\ar[dr, "\id_{\cC(a\to b)}"] \\
	& \cC(a\to b) & 
\end{tikzcd}
\text{ and }\quad
\begin{tikzcd}[column sep=tiny]
	& \cC(a \to b)\cC(b\to b) \ar[dd, "-\circ_\cC -"] & \\
\cC(a\to b) \ar[ur, " \id_{\cC(a\to b)} j_b"]
\ar[dr, "\id_{\cC(a\to b)}"] \\
	& \cC(a\to b) & 
\end{tikzcd}
\end{equation*}
In string diagrams the above axiom reads as:
$$
\begin{tikzpicture}[baseline=25,smallstring]
\node[draw,rectangle] (j) at (0,1) {$j_a$};
\node[] (a-b-bottom) at (2,0) {$\cC(a\to b)$};
\node[draw,rectangle] (circ) at (1,2) {$-\circ_\cC-$};
\node[] (a-b-top) at (1,3) {$\cC(a\to b)$};
\draw (j.90) to[in=-90,out=90] (circ.-135);
\draw (a-b-bottom.90) to[in=-90,out=90] (circ.-45);
\draw (circ.90) to[in=-90,out=90] (a-b-top.270);
\end{tikzpicture}
=
\begin{tikzpicture}[baseline=25,smallstring]
\node[] (a-b-bottom) at (0,0) {$\cC(a\to b)$};
\node[] (a-b-top) at (0,3) {$\cC(a\to b)$};
\draw (a-b-bottom.90) to[in=-90,out=90] (a-b-top.270);
\end{tikzpicture}
=
\begin{tikzpicture}[baseline=25,smallstring]
\node[] (a-b-bottom) at (0,0) {$\cC(a\to b)$};
\node[draw,rectangle] (j) at (2,1) {$j_b$};
\node[draw,rectangle] (circ) at (1,2) {$-\circ_\cC-$};
\node[] (a-b-top) at (1,3) {$\cC(a\to b)$};
\draw (j.90) to[in=-90,out=90] (circ.-45);
\draw (a-b-bottom.90) to[in=-90,out=90] (circ.-135);
\draw (circ.90) to[in=-90,out=90] (a-b-top.270);
\end{tikzpicture}
.
$$
\item
\textbf{
(associativity)
}
For all $a,b,c,d\in\cC$, the following diagram commutes:
\begin{equation*}
\begin{tikzcd}
\cC(a \to b)\cC(b\to c)\cC(c\to d)
\ar[rr, "\,\,\,\,\id(-\circ_\cC -)\,\,\,\,"] 
\ar[d, "(-\circ_\cC -)\id"] 
&&
\cC(a\to b)\cC(b\to d)
\ar[d, "(-\circ_\cC -)"] 
\\
\cC(a\to c)\cC(c\to d)
\ar[rr, "(-\circ_\cC -)"] 
&&
\cC(a\to d)
\end{tikzcd}
\end{equation*}
which in string diagrams becomes:
$$
\begin{tikzpicture}[baseline=30,smallstring]
\node[] (a-b) at (0,0) {$\cC(a\to b)$};
\node[] (b-c) at (2,0) {$\cC(b\to c)$};
\node[] (c-d) at (4,0) {$\cC(c\to d)$};
\node[draw,rectangle] (circ1) at (1,1) {$-\circ_\cC-$};
\node[draw,rectangle] (circ2) at (2,2) {$-\circ_\cC-$};
\node[] (a-d) at (2,3) {$\cC(a\to d)$};
\draw (a-b.90) to[in=-90,out=90] (circ1.-135);
\draw (b-c.90) to[in=-90,out=90] (circ1.-45);
\draw (circ1.90) to[in=-90,out=90] (circ2.-135);
\draw (c-d.90) to[in=-90,out=90] (circ2.-45);
\draw (circ2.90) to[in=-90,out=90] (a-d.-90);
\end{tikzpicture}
=
\begin{tikzpicture}[baseline=30,smallstring]
\node[] (a-b) at (0,0) {$\cC(a\to b)$};
\node[] (b-c) at (2,0) {$\cC(b\to c)$};
\node[] (c-d) at (4,0) {$\cC(c\to d)$};
\node[draw,rectangle] (circ1) at (3,1) {$-\circ_\cC-$};
\node[draw,rectangle] (circ2) at (2,2) {$-\circ_\cC-$};
\node[] (a-d) at (2,3) {$\cC(a\to d)$};
\draw (a-b.90) to[in=-90,out=90] (circ2.-135);
\draw (b-c.90) to[in=-90,out=90] (circ1.-135);
\draw (c-d.90) to[in=-90,out=90] (circ1.-45);
\draw (circ1.90) to[in=-90,out=90] (circ2.-45);
\draw (circ2.90) to[in=-90,out=90] (a-d.-90);
\end{tikzpicture}
.
$$
\end{itemize}

From this point onward, we assume $\cV$ is a braided monoidal category where the braiding in $\cV$ is denoted by $\beta_{u,v} : uv \to vu$ for all $u,v\in \cV$.

\begin{defn}
A (strict)\footnote{It would be wonderful for someone to work out the axioms for a non-strict $\cV$-monoidal category!}
$\cV$-monoidal category $\cC$%
\footnote{
	One can motivate this definition by taking the usual
	notion of a $\cV$-enriched category, at first not using the braiding, then giving $\cA \times \cB$, for $\cA$ and $\cB$
	$\cV$-enriched categories, the structure of a
	$\cV$-enriched category by defining the composition using the braiding in $\cV$ in the inevitable way. 
	After this, the definition above is just the usual definition of a (strict) monoidal category.
	Perhaps someone will prove a coherence theorem for not-necessarily-strict monoidal categories enriched in a braided
	monoidal category, but for now we stay in the strict setting.
}
is a $\cV$-enriched category $\cC$ together with the following data: 
\begin{itemize}
\item
a unit object $1_\cC \in \cC$,
\item
for every $a,b\in \cC$, an object $ab\in\cC$, and
\item
for all $a,b,c,d\in \cC$, a tensor product morphism $-\otimes_\cC - \in \cV(
\cC(a\to c)\cC(b\to d) \to \cC(ab\to cd))$ 
\end{itemize}
which satisfy the following axioms:
\begin{itemize}
\item
(strict unitor for objects)
For all $a\in \cC$, $1_\cC a = a1_\cC = a$.
\item
(strict associator for objects)
For all $a,b,c\in\cC$, $(ab)c=a(bc)$.
\item
(unitality)
For all $a,b\in \cC$, $(j_{1_\cC}\id_{\cC(a\to b)})\circ (-\otimes_\cC-) = \id_{\cC(a\to b)}$, $(\id_{\cC(a\to b)}j_{1_\cC})\circ  (-\otimes_\cC-) = \id_{\cC(a\to b)}$, and
$(j_aj_b)\circ (-\otimes_\cC-) = j_{ab}$.
\item
(associativity of $-\otimes_\cC-$) 
For all $a,b,c,d,e,f\in \cC$, the following diagram commutes:
\begin{equation*}
\begin{tikzcd}
\cC(a \to d)\cC(b\to e)\cC(c\to f)
\ar[rr, "\id(-\otimes_\cC -)"] 
\ar[d, "(-\otimes_\cC -)\id"] 
&&
\cC(a\to d)\cC(bc\to ef)
\ar[d, "(-\otimes_\cC -)"] 
\\
\cC(ab\to de)\cC(c\to f)
\ar[rr, "(-\otimes_\cC -)"] 
&&
\cC(abc\to def).
\end{tikzcd}
\end{equation*}
\item
(braided interchange)
For all $a,b,c,d,e,f\in\cC$, the following diagram commutes:
\begin{equation*}
\begin{tikzcd}
\cC(a \to b) \cC(d \to e) \cC(b \to c) \cC(e \to f) 
\ar[rr, "(-\otimes_\cC -)(-\otimes_\cC -)"] 
\ar[dd, "\id \beta_{\cC(d\to e)\cC(b\to c)} \id"] 
&&
\cC(ad\to be) \cC(be\to cf)
\ar[dr, "-\circ_\cC -"] 
\\
&&&
\cC(ad\to cf).
\\
\cC(a \to b) \cC(b \to c) \cC(d \to e) \cC(e \to f) 
\ar[rr, "(-\circ_\cC -)(-\circ_\cC -)"] 
&&
\cC(a \to c) \cC(d \to f)
\ar[ur, "-\circ_\cC -"] 
\end{tikzcd}
\end{equation*}
The corresponding string diagram for the braided interchange relation was already given in \eqref{eq:BraidedInterchange}.
\end{itemize}
\end{defn}

%
%
%
%
%
%

\subsection{\texorpdfstring{$v$}{v}-graded morphisms}

As the objects of the enriching category $\cV$ do not necessarily have underlying sets, we must be careful when talking about
`morphisms in a $\cV$-enriched category'. 
A \emph{$1_\cV$-graded morphism} from $a$ to $b$ in a $\cV$-monoidal category $\cC$ is a morphism $\id_\cV \to \cC(a \to b)$ of $\cV$,
and a
\emph{$v$-graded morphism}, for $v$ an object of $\cV$, is a morphism $v \to \cC(a \to b)$. 
We can compose (or tensor) a
$u$-graded morphism with a $v$-graded morphism to obtain a $uv$-graded morphism.

A $\id_\cV$-graded morphism $f:\id_\cV \to \cC(a\to b)$ is \emph{invertible}
 if there is a morphism $g: 1_\cV \to \cC(b\to a)$ called an \emph{inverse} such that
the maps
\begin{align*}
\id_\cV \xrightarrow{fg} & \cC(a\to b)\cC(b\to a) \xrightarrow{-\circ_\cC-} \cC(a\to a)
\\
\id_\cV \xrightarrow{gf} & \cC(b\to a)\cC(a\to b) \xrightarrow{-\circ_\cC-} \cC(b\to b)
\end{align*}
are identity elements, i.e., $(fg)\circ(-\circ_\cC -) = j_a$ and $(gf)\circ(-\circ_\cC -)=j_b$. 
Notice that if $f$ is invertible, the usual proof shows its inverse is unique and can be denoted $f^{-1}$:
$$
h = (h j_a)\circ(-\circ_\cC -) = (hfg)\circ(-\circ_\cC-\circ_\cC -) = (j_b g)\circ(-\circ_\cC -) = g.
$$
There are obvious notions of monomorphisms and epimorphisms which we leave to the reader.

\subsection{\texorpdfstring{$\cV$}{V}-functors}

A \emph{$\cV$-functor} $\cF: \cC \to \cD$ between $\cV$-enriched categories is just a function between the objects, and
for each
$a,b\in\cC$,
an element $\cF_{a \to b} \in \cV(\cC(a \to b) \to \cD(\cF(a) \to \cF(b))$, such that
\begin{equation}
\label{eq:CompositionAxiomForVFunctors}
\begin{tikzcd}[column sep=huge]
\cC(a \to b) \cC(b \to c)
	\ar[d, "\cF_{a \to b} \cF_{b \to c}"]
	\ar[r, "-\circ_\cC-"] 
&
\cC(a \to c)
	\ar[d, "\cF_{a \to c}"] 
\\
\cD(\cF(a) \to \cF(b)) \cD(\cF(b) \to \cF(c))
	\ar[r, "-\circ_{\cD}-"] 
&
\cD(\cF(a) \to \cF(c))
\end{tikzcd}
\end{equation}
commutes, as does
\begin{equation}
\label{eq:TriangleIdentityAxiomForVFunctors}
\begin{tikzcd}[column sep=tiny]
	& \cC(a \to a) \ar[dd, "\cF_{a \to a}"] & \\
1_\cV \ar[ur, "j^\cC_a"]
\ar[dr, "j^\cD_{\cF(a)}"] \\
	& \cD(\cF(a) \to \cF(a)) & 
\end{tikzcd}
.
\end{equation}
Given $\cV$-functors $\cF:\cC \to \cD$ and $\cG:\cD \to \cE$, we define the $\cV$-functor $\cF\circ \cG : \cC \to \cE$ by $(\cF \circ \cG)_{a\to b} = \cF_{a\to b} \circ \cG_{\cF(a) \to \cF(b)}$ for all $a,b\in \cC$.

A
\emph{$v$-graded natural transformation} $\lambda$ (again, for $v\in \cV$) between $\cV$-functors $\cF, \cG: \cC
\to
\cD$ is
a collection $\lambda_a: v \to
\cD(\cF(a) \to \cG(a))$ so that the following diagram commutes for all objects $a, b$:
$$
\begin{tikzcd}
v \cC(a \to b) \ar[r, "\lambda_a \cG_{a \to b}"] \ar[dd, "\beta"]
	& \cD(\cF(a) \to \cG(a)) \cD(\cG(a) \to \cG(b)) \ar[dr, "-\circ_\cD-"]
	& \\
 & & \cD(\cF(a) \to \cG(b)). \\
\cC(a \to b) v \ar[r, "\cF_{a \to b} \lambda_b"]
	& \cD(\cF(a) \to \cF(b)) \cD(\cF(b) \to \cG(b)) \ar[ur, "-\circ_\cD-"]
	&
\end{tikzcd}
$$
(This is just `the naturality square, viewed from outside'.) We write $\lambda: \cF \Rightarrow \cG$.


Notice that we are only talking about natural transformations for $\cV$-enriched categories rather than $\cV$-monoidal categories, and yet the braiding in $\cV$ is essential to the definition! 

\begin{lem}
\label{lem:VerticalComposition}

Suppose $\lambda: \cF \Rightarrow \cG$ is a $u$-graded natural transformation, and  $\mu: \cG \Rightarrow \cH$ is a $v$-graded natural transformation. Then there is a $uv$-graded natural transformation $\lambda \circ \mu: \cF \Rightarrow \cH$, called the vertical composite, defined by
\begin{align*}
(\lambda \circ \mu)_a :
	uv
	  \xrightarrow{\lambda_a \mu_a}
  & \cD(\cF(a) \to \cG(a)) \cD(\cG(a) \to \cH(a))) \\
  	  \xrightarrow{-\circ_\cD-}
  & \cD(\cF(a) \to \cH(a)).
\end{align*}
Vertical composition is associative.
\end{lem}
The proof is sufficiently straightforward that we leave it as an exercise.

Similarly, horizontal composition follows the usual formula:

\begin{lem}
\label{lem:HorizontalComposition}
Suppose $\kappa: \cF \Rightarrow \cG$ is a $u$-graded natural transformation where $\cF, \cG: \cC \to \cD$, and 
$\lambda: \cH \Rightarrow \cI$ is $v$-graded where $\cH, \cI : \cD \to \cE$.
The following formula defines a $uv$-graded natural transformation called the horizontal composite:
$\kappa \lambda: \cF \circ \cH \Rightarrow \cG \circ \cI$ by
\begin{align*}
(\kappa \lambda)_a : uv
	  \xrightarrow{\kappa_a \lambda_{\cG(a)}}
  & \cD(\cF(a) \to \cG(a)) \cE(\cH(\cG(a)) \to \cI(\cG(a))) \\
      \xrightarrow{\cH_{\cF(a) \to \cG(a)} \id}
  & \cE(\cH(\cF(a)) \to \cH(\cG(a))) \cE(\cH(\cG(a)) \to \cI(\cG(a))) \\
  	  \xrightarrow{-\circ_\cE-}
  & \cE(\cH(\cF(a)) \to \cI(\cG(a))).
\end{align*}
Again, horizontal composition is associative.
\end{lem}

\begin{lem}
\label{lem:BraidedInterchangeForNaturalTransformations}
Natural transformations themselves satisfy a braided interchange.
Given $\cV$-monoidal categories, functors, and natural transformations
$$
\begin{tikzpicture}
\node (C) at (0,0) {$\cC$};
\node (D) at (2,0) {$\cD$};
\node (E) at (4,0) {$\cE$};
\draw[->] (C) to[out=-60,in=-120] node (CD1) {} node[below] (F) {\tiny $\cF$}  (D);
\draw[->] (C) to[out=0,in=-180] node (CD2) {} node[above,xshift=3mm] (G) {\tiny $\cG$}(D);
\draw[->] (C) to[out=60,in=120] node (CD3) {} node[above] (H) {\tiny $\cH$} (D);
\draw[->] (D) to[out=-60,in=-120] node (DE1) {} node[below] (I) {\tiny $\cI$}  (E);
\draw[->] (D) to[out=0,in=-180] node (DE2) {} node[above,xshift=3mm] (J) {\tiny $\cJ$}(E);
\draw[->] (D) to[out=60,in=120] node (DE3) {} node[above] (K) {\tiny $\cK$} (E);
\draw[double,arrows={-stealth}] (CD1) -- node[left] {\tiny $\kappa$} (CD2);
\draw[double,arrows={-stealth}] (CD2) -- node[left] {\tiny $\mu$} (CD3);
\draw[double,arrows={-stealth}] (DE1) -- node[left] {\tiny $\lambda$} (DE2);
\draw[double,arrows={-stealth}] (DE2) -- node[left] {\tiny $\nu$} (DE3);
\end{tikzpicture}
$$
where $\kappa, \lambda,\mu,\nu$ are respectively $u,v,w,x$-graded natural transformations,
we have 
$$ ((\kappa \lambda) \circ (\mu \nu))_a = (\id_u \beta_{v,w} \id_x) \circ ((\kappa \circ \mu) (\lambda \circ \nu))_a : uvwx \to \cE(\cI(\cF(a)) \to \cK(\cH(a))).
$$
\end{lem}

We defer the proofs of Lemma \ref{lem:HorizontalComposition} and Lemma \ref{lem:BraidedInterchangeForNaturalTransformations} to Appendix \ref{sec:NaturalityOfNaturalTransformationComposition}.

\begin{remark}
It would be interesting to show that $\cV$-categories, $\cV$-functors, and $\cV$-natural transformations form a $\cV$-enriched 2-category.
In doing so, one would define a $\Hom$-object $\Nat(\cF\Rightarrow \cG)\in \cV$ for $\cV$-functors $\cF,\cG : \cC\to \cD$.
We could then express vertical composition as a morphism $- \circ_{\Nat} - :  \Nat(\cF \Rightarrow \cG)\Nat(\cG \Rightarrow \cE)
\to \Nat(\cF \Rightarrow \cE)$ and horizontal composition as a morphism $- \otimes_{\Nat} -
: \Nat(\cF \Rightarrow \cG)\Nat(\cH \Rightarrow \cI) \to \Nat(\cF \cI \Rightarrow \cG \cI)$.
One would then prove that these morphims satisfied a braided interchange.
\end{remark}

\subsection{Self-enriched categories}
\label{sec:VEnrichedInV}
Given a braided rigid category $\cV$, we can construct a $\cV$-monoidal category $\widehat{\cV}$ with the same
objects as $\cV$, and $\widehat{\cV}(u \to v) \overset{\text{def}}{=} u^*v$. The composition and tensor product
maps are given by 
\begin{align*}
- \circ_{\widehat{\cV}} - :
{\widehat{\cV}}(u \to v) {\widehat{\cV}}(v \to w) = u^*vv^*w 
	& \xrightarrow{u^* \ev_v w} u^*w = {\widehat{\cV}}(u \to w) \\
\begin{tikzpicture}[baseline,smallstring]
\node[draw,rectangle] (circ) at (0,0) {$ - \circ_{\widehat{\cV}} - $};
\draw (circ) -- +(0,0.6);
\draw (circ.-135) to[in=90,out=-90] +(-0.4,-0.3);
\draw (circ.-45)  to[in=90,out=-90] +( 0.4,-0.3);
\end{tikzpicture} & =
\begin{tikzpicture}[baseline,smallstring]
\draw (-0.6,-0.6) to[in=-90,out=90] (-0.2,0.6);
\draw (0.6,-0.6) to[in=-90,out=90] (0.2,0.6);
\draw (-0.2,-0.6) to[in=90,out=90] (0.2,-0.6);
\end{tikzpicture}
\\
- \otimes_{\widehat{\cV}} - :
{\widehat{\cV}}(u \to v) {\widehat{\cV}}(w \to x) = u^*vw^*x
	& \xrightarrow{\beta^{-1}_{u^*v,w^*}x} w^*u^*vx = {\widehat{\cV}}(uw \to vx) \\
\begin{tikzpicture}[baseline,smallstring]
\node[draw,rectangle] (circ) at (0,0) {$ - \otimes_{\widehat{\cV}} - $};
\draw (circ) -- +(0,0.6);
\draw (circ.-135) to[in=90,out=-90] +(-0.4,-0.3);
\draw (circ.-45)  to[in=90,out=-90] +( 0.4,-0.3);
\end{tikzpicture} & =
\begin{tikzpicture}[baseline,smallstring]
\draw (-0.6,-0.6) -- (-0.2,0.6);
\draw (-0.2,-0.6) -- (0.2,0.6);
\draw[knot] ( 0.2,-0.6) -- (-0.6,0.6);
\draw ( 0.6,-0.6) -- (0.6,0.6);
\end{tikzpicture}
\end{align*}
It is an enjoyable exercise to discover that these satisfy the braided exchange law:
\begin{align*}
\begin{tikzpicture}[baseline=22,x=0.3cm,y=1.25cm,smallstring]
\draw       (0,0) to[in=-90,out=90] (1,1);
\draw       (1,0) to[in=-90,out=90] (2,1);
\draw[knot] (2,0) to[in=-90,out=90] (0,1);
\draw       (3,0) to[in=-90,out=90] (3,1);
\draw       (4,0) to[in=-90,out=90] (5,1);
\draw       (5,0) to[in=-90,out=90] (6,1);
\draw[knot] (6,0) to[in=-90,out=90] (4,1);
\draw       (7,0) to[in=-90,out=90] (7,1);
\draw (0,1) to[in=-90,out=90]  (2,2);
\draw (1,1) to[in=-90,out=90]  (3,2);
\draw (2,1) to[in=90,out=90] (5,1);
\draw (3,1) to[in=90,out=90] (4,1);
\draw (6,1) to[in=-90,out=90]  (4,2);
\draw (7,1) to[in=-90,out=90]  (5,2);
\end{tikzpicture}
& = 
\begin{tikzpicture}[baseline=22,x=0.3cm,y=1cm,smallstring]
\draw       (0,0) --                  (0,1);
\draw       (1,0) --                  (1,1);
\draw       (4,0) to[in=-90,out=90]   (2,1);
\draw       (5,0) to[in=-90,out=90]   (3,1);
\draw[knot] (2,0) to[in=-90,out=90]   (4,1);
\draw[knot] (3,0) to[in=-90,out=90]   (5,1);
\draw       (6,0) --                  (6,1);
\draw       (7,0) --                  (7,1);
\draw       (0,1) to[in=-90,out=90]   (0.5,1.5);
\draw       (3,1) to[in=-90,out=90]   (2.5,1.5);
\draw       (4,1) to[in=-90,out=90]   (4.5,1.5);
\draw       (7,1) to[in=-90,out=90]   (6.5,1.5);
\draw       (1,1) to[in=90,out=90]  (2,1);
\draw       (5,1) to[in=90,out=90]  (6,1);
\draw       (0.5,1.5) to[in=-90,out=90] (2.5,2.5);
\draw       (2.5,1.5) to[in=-90,out=90] (4.5,2.5);
\draw[knot] (4.5,1.5) to[in=-90,out=90] (0.5,2.5);
\draw       (6.5,1.5) to[in=-90,out=90] (6.5,2.5);
\end{tikzpicture}
\,.
\end{align*}
This example is related to the canonical functor $\cV \to Z(\cV)$ when $\cV$ is braided, via our main theorem. 
It is the analogue in the braided rigid setting of the example in \S 1.6 of \cite{MR2177301}.

The category $\sf Tangle$ of (unoriented, framed) tangles is a braided rigid monoidal
category. 
It has a faithful functor from $\sf Braid$, the free braided monoidal category on one object. 
The objects of $\sf Tangle$ and of $\sf Braid$ are just the natural numbers.
We denote the standard generators of the braid group by $\sigma_i$.

We can form $\widehat{\sf Tangle}$, the category of tangles enriched in
itself. This allows us to prove the following useful result.

\begin{lem}
\label{lem:equationsInBraids}
Suppose that an equation of the form
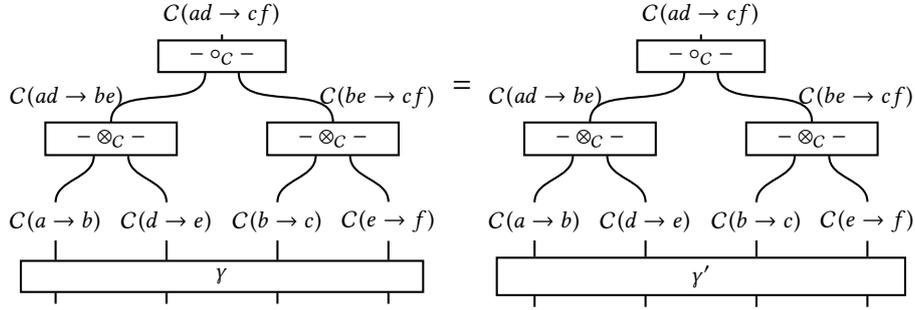
\begin{align*}
\begin{tikzpicture}[baseline=50,smallstring]
\node (a-b) at (0,0) {$\cC(a \to b)$};
\node (d-e) at (2,0) {$\cC(d \to e)$};
\node (b-c) at (4,0) {$\cC(b \to c)$};
\node (e-f) at (6,0) {$\cC(e \to f)$};
\node[draw,rectangle] (t1) at (1,1.5) {$\quad -\otimes_\cC-\quad $};
\draw (a-b) to[in=-90,out=90] (t1.-135);
\draw (d-e) to[in=-90,out=90] (t1.-45);
\node[draw,rectangle] (t2) at (5,1.5) {$\quad -\otimes_\cC-\quad $};
\draw (b-c) to[in=-90,out=90] (t2.-135);
\draw (e-f) to[in=-90,out=90] (t2.-45);
\node[draw,rectangle] (c) at (3,3) {$\quad - \circ_\cC - \quad $};
\draw (t1) to[in=-90,out=90] node[left=13pt] {$\cC(ad \to be)$} (c.-135);
\draw (t2) to[in=-90,out=90] node[right=13pt] {$\cC(be \to cf)$} (c.-45);
\node (r) at (3,3.75) {$\cC(ad \to cf)$};
\draw (c) -- (r);
\node[draw,rectangle] (beta) at (3,-1) {$\qquad\qquad\qquad\qquad \gamma
\qquad\qquad\qquad\qquad$};
\draw (a-b) -- (a-b|-beta.north);
\draw (d-e) -- (d-e|-beta.north);
\draw (b-c) -- (b-c|-beta.north);
\draw (e-f) -- (e-f|-beta.north);
\draw (a-b|-beta.south) -- +(0,-0.2);
\draw (d-e|-beta.south) -- +(0,-0.2);
\draw (b-c|-beta.south) -- +(0,-0.2);
\draw (e-f|-beta.south) -- +(0,-0.2);
\end{tikzpicture}
=
\begin{tikzpicture}[baseline=50,smallstring]
\node (a-b) at (0,0) {$\cC(a \to b)$};
\node (d-e) at (2,0) {$\cC(d \to e)$};
\node (b-c) at (4,0) {$\cC(b \to c)$};
\node (e-f) at (6,0) {$\cC(e \to f)$};
\node[draw,rectangle] (t1) at (1,1.5) {$\quad -\otimes_\cC-\quad $};
\draw (a-b) to[in=-90,out=90] (t1.-135);
\draw (d-e) to[in=-90,out=90] (t1.-45);
\node[draw,rectangle] (t2) at (5,1.5) {$\quad -\otimes_\cC-\quad $};
\draw (b-c) to[in=-90,out=90] (t2.-135);
\draw (e-f) to[in=-90,out=90] (t2.-45);
\node[draw,rectangle] (c) at (3,3) {$\quad - \circ_\cC - \quad $};
\draw (t1) to[in=-90,out=90] node[left=13pt] {$\cC(ad \to be)$} (c.-135);
\draw (t2) to[in=-90,out=90] node[right=13pt] {$\cC(be \to cf)$} (c.-45);
\node (r) at (3,3.75) {$\cC(ad \to cf)$};
\draw (c) -- (r);
\node[draw,rectangle] (beta) at (3,-1) {$\qquad\qquad\qquad\qquad \gamma'
\qquad\qquad\qquad\qquad$};
\draw (a-b) -- (a-b|-beta.north);
\draw (d-e) -- (d-e|-beta.north);
\draw (b-c) -- (b-c|-beta.north);
\draw (e-f) -- (e-f|-beta.north);
\draw (a-b|-beta.south) -- +(0,-0.2);
\draw (d-e|-beta.south) -- +(0,-0.2);
\draw (b-c|-beta.south) -- +(0,-0.2);
\draw (e-f|-beta.south) -- +(0,-0.2);
\end{tikzpicture}
\end{align*}
where $\gamma$ and $\gamma'$ are 4-strand braids with the same underlying
permutation,
holds for all $\cV$-enriched categories, and all choices of objects
$a,b,c,d,e,f$. Then $\gamma = \gamma'$.
\end{lem}
\begin{proof}
We pick $a = c = d = f = 1$, and $b = e = 0$ in $\widehat{\sf Tangle}$.
Then $\widehat{\sf Tangle}(1 \to 0) = \widehat{\sf Tangle}(0 \to 1) = 1$,
and (abbreviating $\sf Tangle$ to $\sf T$)
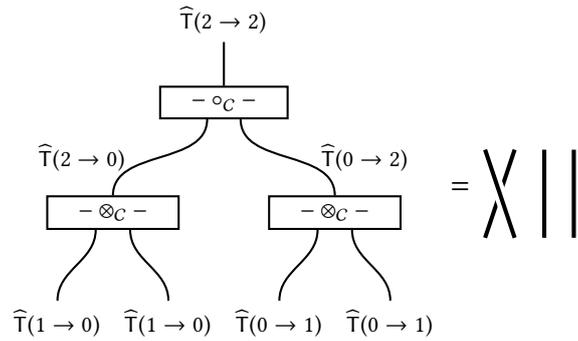
\begin{align*}
\begin{tikzpicture}[baseline=50,smallstring]
\node (a-b) at (0,0) {$\widehat{\sf T}(1 \to 0)$};
\node (d-e) at (2,0) {$\widehat{\sf T}(1 \to 0)$};
\node (b-c) at (4,0) {$\widehat{\sf T}(0 \to 1)$};
\node (e-f) at (6,0) {$\widehat{\sf T}(0 \to 1)$};
\node[draw,rectangle] (t1) at (1,2) {$\quad -\otimes_\cC-\quad $};
\draw (a-b) to[in=-90,out=90] (t1.-135);
\draw (d-e) to[in=-90,out=90] (t1.-45);
\node[draw,rectangle] (t2) at (5,2) {$\quad -\otimes_\cC-\quad $};
\draw (b-c) to[in=-90,out=90] (t2.-135);
\draw (e-f) to[in=-90,out=90] (t2.-45);
\node[draw,rectangle] (c) at (3,4) {$\quad - \circ_\cC - \quad $};
\draw (t1) to[in=-90,out=90] node[left=13pt] {$\widehat{\sf T}(2 \to 0)$} (c.-135);
\draw (t2) to[in=-90,out=90] node[right=13pt] {$\widehat{\sf T}(0 \to 2)$} (c.-45);
\node (r) at (3,5.5) {$\widehat{\sf T}(2 \to 2)$};
\draw (c) -- (r);
\end{tikzpicture}
& = 
\begin{tikzpicture}[string,baseline]
\draw (-0.6,-0.6) -- (-0.2,0.6);
\draw[knot] (-0.2,-0.6) -- (-0.6,0.6);
\draw ( 0.2,-0.6) -- (0.2,0.6);
\draw ( 0.6,-0.6) -- (0.6,0.6);
\end{tikzpicture}
\end{align*}
Thus the equation reduces to $\gamma \sigma_1^{-1} = \gamma' \sigma_1^{-1}$ in $\sf Tangle$, which is equivalent to $\gamma=\gamma'$.
Since braids map faithfully into tangles, we have the conclusion.
\end{proof}

\subsection{The rotation of a \texorpdfstring{$\cV$}{V}-monoidal category}
\label{sec:RotatingC}

Given a monoidal category $\cC$ enriched in a \emph{symmetric} monoidal category, we can take the opposite composition or the opposite tensor product, obtaining a new enriched monoidal category.
When the enrichment is merely braided, we find that this is not the case.
Nevertheless there is something which we call the $\pi$-rotation of $\cC$, which is formed by simultaneously modifying the composition and the tensor product.
This is another point of departure from the theory for symmetric enrichments.

\begin{lem}
Suppose that $\cC$ is a
$\cV$-monoidal category. 
Consider a new composition and tensor product on $\cC$ given by
$$
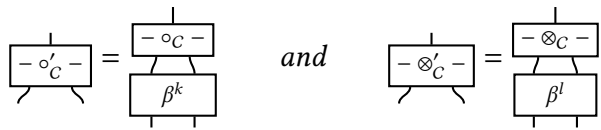
\begin{tikzpicture}[smallstring,baseline]
\node[draw,rectangle] (circ) at (0,0) {$ - \circ_\cC' - $};
\draw (circ) -- +(0,0.6);
\draw (circ.-135) to[in=90,out=-90] +(-0.2,-0.3);
\draw (circ.-45)  to[in=90,out=-90] +( 0.2,-0.3);
\end{tikzpicture} 
=
\begin{tikzpicture}[smallstring,baseline=-10]
\node[draw,rectangle] (circ) at (0,0) {$ - \circ_\cC - $};
\draw (circ) -- +(0,0.6);
\node[draw,rectangle] (tau) at (0,-1) {$\quad \beta^k \quad$};
\draw (circ.-135) to[in=90,out=-90] (tau.135);
\draw (circ.-45)  to[in=90,out=-90] (tau.45);
\draw (tau.-135) -- +(0,-0.2);
\draw (tau.-45) -- +(0,-0.2);
\end{tikzpicture} 
\qquad \text{and} \qquad
\begin{tikzpicture}[smallstring,baseline]
\node[draw,rectangle] (circ) at (0,0) {$ - \otimes_\cC' - $};
\draw (circ) -- +(0,0.6);
\draw (circ.-135) to[in=90,out=-90] +(-0.2,-0.3);
\draw (circ.-45)  to[in=90,out=-90] +( 0.2,-0.3);
\end{tikzpicture}
 =
\begin{tikzpicture}[smallstring,baseline=-10]
\node[draw,rectangle] (circ) at (0,0) {$ - \otimes_\cC - $};
\draw (circ) -- +(0,0.6);
\node[draw,rectangle] (tau) at (0,-1) {$\quad \beta^l \quad$};
\draw (circ.-135) to[in=90,out=-90] (tau.135);
\draw (circ.-45)  to[in=90,out=-90] (tau.45);
\draw (tau.-135) -- +(0,-0.2);
\draw (tau.-45) -- +(0,-0.2);
\end{tikzpicture}
$$
for some $k,l\in \bbZ$.
These always satisfy associativity, but satisfy the braided interchange
axiom if and only if $k=l$.
\end{lem}
\begin{proof}
The braided interchange axiom becomes (by pulling composition or tensor
product morphisms through the twists)
$$
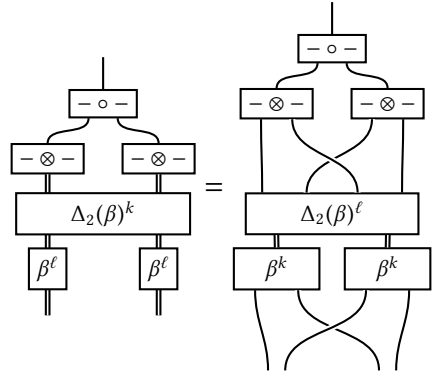
\begin{tikzpicture}[smallstring,baseline=50]
\node (bottom-L) at (-1,0) {};
\node (bottom-R) at (1,0) {};
\node[draw,rectangle] (otimes-L) at (-1,3) {$ - \otimes - $};
\node[draw,rectangle] (otimes-R) at (1,3) {$ - \otimes - $};
\draw[double] (bottom-L.90)  to[in=-90,out=90] (otimes-L.-90);
\draw[double] (bottom-R.90)  to[in=-90,out=90] (otimes-R.-90);
\node[draw,rectangle, fill=white] (beta-L) at (-1,1) {$\beta^\ell$};
\node[draw,rectangle, fill=white] (beta-R) at (1,1) {$\beta^\ell$};
\node[draw,rectangle, fill=white] (Delta) at (0,2) {$\qquad\Delta_2(\beta)^k \qquad$};
\node[draw,rectangle] (circ) at (0,4) {$ - \circ - $};
\node (top) at (0,5) {};
\draw (otimes-L.90)  to[in=-90,out=90] (circ.-135);
\draw (otimes-R.90)  to[in=-90,out=90] (circ.-45);
\draw (circ.90)  to[in=-90,out=90] (top.-90);
\end{tikzpicture}
=
\begin{tikzpicture}[smallstring,baseline=50]
\node (bottom-L) at (-1,-1) {};
\node (bottom-R) at (1,-1) {};
\node[draw,rectangle, fill=white] (beta-L) at (-1,1) {$\quad\beta^k\quad$};
\node[draw,rectangle, fill=white] (beta-R) at (1,1) {$\quad\beta^k\quad$};
\node[draw,rectangle, fill=white] (Delta) at (0,2) {$\qquad\Delta_2(\beta)^\ell \qquad$};
\node[draw,rectangle] (otimes-L) at (-1,4) {$ - \otimes - $};
\node[draw,rectangle] (otimes-R) at (1,4) {$ - \otimes - $};
\node[draw,rectangle] (circ) at (0,5) {$ - \circ - $};
\node (top) at (0,6) {};
\draw (bottom-L.135)  to[in=-90,out=90] (beta-L.-135);
\draw (bottom-R.45)  to[in=-90,out=90] (beta-R.-45);
\draw[knot] (bottom-R.135)  to[in=-90,out=90] (beta-L.-45);
\draw[knot] (bottom-L.45)  to[in=-90,out=90] (beta-R.-135);
\draw[double] (beta-L.90)  to[in=-90,out=90] (Delta.-159);
\draw[double] (beta-R.90)  to[in=-90,out=90] (Delta.-21);
\draw[knot] (Delta.140)  to[in=-90,out=90] (otimes-R.-135);
\draw[knot] (Delta.40)  to[in=-90,out=90] (otimes-L.-45);
\draw (Delta.163)  to[in=-90,out=90] (otimes-L.-135);
\draw (Delta.17)  to[in=-90,out=90] (otimes-R.-45);
\draw (otimes-L.90)  to[in=-90,out=90] (circ.-135);
\draw (otimes-R.90)  to[in=-90,out=90] (circ.-45);
\draw (circ.90)  to[in=-90,out=90] (top.-90);
\end{tikzpicture}
$$
where 
$
\Delta_2(\beta)
=
\begin{tikzpicture}[baseline=15, smallstring]
\node (bottom-L) at (0,0) {};
\node (bottom-R) at (1.5,0) {};
\node (top-L) at (0,1.5) {};
\node (top-R) at (1.5,1.5) {};
\draw[double] (bottom-R.90)  to[in=-90,out=90] (top-L.-90);
\draw[super thick, white] (bottom-L.90)  to[in=-90,out=90] (top-R.-90);
\draw[double] (bottom-L.90)  to[in=-90,out=90] (top-R.-90);
\end{tikzpicture}
=
\sigma_2 \sigma_1 \sigma_3 \sigma_2
$ 
denotes the two strand cabling of $\beta$.
By Lemma \ref{lem:equationsInBraids}, we have that
$$
\sigma_1^l \sigma_3^l (\sigma_2 \sigma_1 \sigma_3 \sigma_2)^k 
 =
\sigma_2 \sigma_1^k \sigma_3^k (\sigma_2 \sigma_1 \sigma_3 \sigma_2)^l\sigma_2^{-1}
$$
(recall that the $\sigma_i$ are the standard generators of the braid group).
We then apply the Burau representation to this identity. Looking at the 
first column of the last row of the resulting 4-by-4 matrices, we obtain
$$
\left((-1)^k+(-1)^{l+1}\right) t^{k+l+2}+\left((-1)^k+(-1)^{l+1}\right) t^
{k+l+4}-2 (-1)^k t^{2 k+3}+2 (-1)^l t^{2 l+3} = 0. $$
It is relatively straightforward to see that this polynomial in $t$ only
vanishes identically when $k = l$ or when $k = l \pm 1$.
We now take the second column of the second row, and after setting $k
= l \pm 1$ and clearing a
denominator, obtain
$$
2 (-1)^l (t-1)^2 (t+1) \left(t^2+1\right) t^{2 l \pm 1} = 0
$$
which is impossible. Thus we must have $k=l$, and an isotopy verifies
that the braided interchange axiom indeed holds.
\end{proof}

Note in the above that if the enriching category were symmetric, the integers
$k$ and $l$ would only have appeared modulo 2, and indeed all four choices
would have given new enriched monoidal categories.

\begin{defn}
We define $\cC^\text{rot}$ to be the $\cV$-monoidal category obtained taking
$k=l=1$ in the above lemma.
\end{defn}

Again, note that in the symmetrically enriched case, $(\cC^\text{rot})^{\text
{rot}} = \cC$.
Generally, this is not the case, so we obtain an integer family of rotations of the original category. 
When we discuss rigidity below, we will see that a choice of duality functor is a choice of an isomorphism $\cC \cong \cC^{\text{rot}}$.

\subsection{Products of \texorpdfstring{$\cV$}{V}-monoidal categories (only?) exist when \texorpdfstring{$\cV$}{V} is symmetric}
\label{sec:NoProductsOfVMonoidalCategories}

We now point out a significant difference between the theory of monoidal categories with a braided enrichment, and the
theory of monoidal categories with a symmetric enrichment.

If $\cC$ and $\cD$ are $\cV$-monoidal categories enriched in a symmetric monoidal category $\cV$, we can
define their Cartesian product $\cC \times_\cV \cD$, which is also a $\cV$-monoidal category.
First, we produce the $(\cV\times \cV)$-enriched monoidal category $\cC\times \cD$ by $(\cC\times \cD)((a,c)\to (b, d)) = \cC(a\to b)\cD(c\to d)$, and composition is given by
$$
\begin{tikzpicture}[baseline=30,smallstring]
\node (bottom-L) at (0,0) {$\cC(a\to b) \cD(d\to e)$};
\node (bottom-R) at (4,0) {$\cC(b\to c) \cD(e\to f)$};
\node[draw,rectangle] (circ) at (2,2) {$-\circ_{\cC\times \cD} -$};
\node (top) at (2,3) {$\cC(a\to c) \cD(d\to f)$};
\draw (bottom-L) to[in=-90,out=90] (circ.-135);
\draw (bottom-R) to[in=-90,out=90] (circ.-45);
\draw (circ) to[in=-90,out=90] (top);
\end{tikzpicture}
=
\begin{tikzpicture}[baseline=30,smallstring]
\node (a-b) at (0,0) {$\cC(a\to b)$};
\node (d-e) at (2,0) {$\cD(d\to e)$};
\node (b-c) at (4,0) {$\cC(b\to c)$};
\node (e-f) at (6,0) {$\cD(e\to f)$};
\node[draw,rectangle] (circ-C) at (2,2) {$-\circ_{\cC} -$};
\node[draw,rectangle] (circ-D) at (4,2) {$-\circ_{\cD} -$};
\node (a-c) at (2,3) {$\cC(a\to c)$};
\node (d-f) at (4,3) {$\cD(d\to f)$};
\draw (a-b) to[in=-90,out=90] (circ-C.-135);
\draw (b-c) to[in=-90,out=90] (circ-C.-45);
\draw (d-e) to[in=-90,out=90] (circ-D.-135);
\draw (e-f) to[in=-90,out=90] (circ-D.-45);
\draw (circ-C) to[in=-90,out=90] (a-c);
\draw (circ-D) to[in=-90,out=90] (d-f);
\end{tikzpicture}
,
$$
with a similar formula for tensor product.
We then apply the braided lax monoidal functor $\cM:\cV \times \cV \to \cV$ given by $(u,v) \mapsto uv$ to transport the $\cV\times \cV$ enrichment to $\cV$ as in \ref{sec:TransportingEnrichment}, obtaining the $\cV$-enriched monoidal category $\cC\times_\cV \cD :=\cM_*(\cC\times \cD)$.

We now observe that it is not possible to follow this construction in the setting where $\cV$ is merely braided.
We can still form the $\cV\times \cV$-enriched monoidal category $\cC \times \cD$, but there is no braided lax monoidal functor $\cV\times \cV \to \cV$.

Nevertheless, we can attempt to define a composition and tensor product on $\cC \times \cD$ by the formulas
\begin{align*}
\begin{tikzpicture}[baseline=30,smallstring]
\node (a-b) at (0,0) {$\cC(a\to b)$};
\node (d-e) at (2,0) {$\cD(d\to e)$};
\node (b-c) at (4,0) {$\cC(b\to c)$};
\node (e-f) at (6,0) {$\cD(e\to f)$};
\node[draw,rectangle] (circ) at (3,2) {$\quad-\circ_{p} -\quad$};
\node (a-c) at (2,3) {$\cC(a\to c)$};
\node (d-f) at (4,3) {$\cD(d\to f)$};
\draw (a-b) to[in=-90,out=90] (circ.-150);
\draw (d-e) to[in=-90,out=90] (circ.-120);
\draw (b-c) to[in=-90,out=90] (circ.-60);
\draw (e-f) to[in=-90,out=90] (circ.-30);
\draw (circ.150) to[in=-90,out=90] (a-c);
\draw (circ.30) to[in=-90,out=90] (d-f);
\end{tikzpicture}
&=
\begin{tikzpicture}[baseline=30,smallstring]
\node (a-b) at (0,0) {$\cC(a\to b)$};
\node (d-e) at (2,0) {$\cD(d\to e)$};
\node (b-c) at (4,0) {$\cC(b\to c)$};
\node (e-f) at (6,0) {$\cD(e\to f)$};
\node[draw,rectangle] (p) at (3,2) {$\qquad\qquad p\qquad\qquad$};
\node[draw,rectangle] (circ-C) at (2,3) {$-\circ_{\cC} -$};
\node[draw,rectangle] (circ-D) at (4,3) {$-\circ_{\cD} -$};
\node (a-c) at (2,4) {$\cC(a\to c)$};
\node (d-f) at (4,4) {$\cD(d\to f)$};
\draw (a-b) to[in=-90,out=90] (p.-168);
\draw[knot] (b-c) to[in=-90,out=90] (p.-150);
\draw[knot] (d-e) to[in=-90,out=90] (p.-30);
\draw (e-f) to[in=-90,out=90] (p.-12);
\draw (p.168) to[in=-90,out=90] (circ-C.-135);
\draw (p.150) to[in=-90,out=90] (circ-C.-45);
\draw (p.30) to[in=-90,out=90] (circ-D.-135);
\draw (p.12) to[in=-90,out=90] (circ-D.-45);
\draw (circ-C) to[in=-90,out=90] (a-c);
\draw (circ-D) to[in=-90,out=90] (d-f);
\end{tikzpicture}
\\
\begin{tikzpicture}[baseline=30,smallstring]
\node (a-b) at (0,0) {$\cC(a\to b)$};
\node (e-f) at (2,0) {$\cD(e\to f)$};
\node (c-d) at (4,0) {$\cC(c\to d)$};
\node (g-h) at (6,0) {$\cD(g\to h)$};
\node[draw,rectangle] (circ) at (3,2) {$\quad-\otimes_{q} -\quad$};
\node (a-c) at (1.8,3) {$\cC(ac\to bd)$};
\node (d-f) at (4.2,3) {$\cD(eg\to fh)$};
\draw (a-b) to[in=-90,out=90] (circ.-150);
\draw (e-f) to[in=-90,out=90] (circ.-120);
\draw (c-d) to[in=-90,out=90] (circ.-60);
\draw (g-h) to[in=-90,out=90] (circ.-30);
\draw (circ.150) to[in=-90,out=90] (a-c);
\draw (circ.30) to[in=-90,out=90] (d-f);
\end{tikzpicture}
&=
\begin{tikzpicture}[baseline=30,smallstring]
\node (a-b) at (0,0) {$\cC(a\to b)$};
\node (e-f) at (2,0) {$\cD(e\to f)$};
\node (c-d) at (4,0) {$\cC(c\to d)$};
\node (g-h) at (6,0) {$\cD(g\to h)$};
\node[draw,rectangle] (p) at (3,2) {$\qquad\qquad q\qquad\qquad$};
\node[draw,rectangle] (otimes-C) at (2,3) {$-\otimes_{\cC} -$};
\node[draw,rectangle] (otimes-D) at (4,3) {$-\otimes_{\cD} -$};
\node (a-c) at (1.8,4) {$\cC(ac\to bd)$};
\node (d-f) at (4.2,4) {$\cD(eg\to fh)$};
\draw (a-b) to[in=-90,out=90] (p.-168);
\draw[knot] (c-d) to[in=-90,out=90] (p.-150);
\draw[knot] (e-f) to[in=-90,out=90] (p.-30);
\draw (g-h) to[in=-90,out=90] (p.-12);
\draw (p.168) to[in=-90,out=90] (otimes-C.-135);
\draw (p.150) to[in=-90,out=90] (otimes-C.-45);
\draw (p.30) to[in=-90,out=90] (otimes-D.-135);
\draw (p.12) to[in=-90,out=90] (otimes-D.-45);
\draw (otimes-C) to[in=-90,out=90] (a-c);
\draw (otimes-D) to[in=-90,out=90] (d-f);
\end{tikzpicture}
\end{align*}
where $p$ and $q$ are each some 4-strand pure braid. One can readily check that these definitions are associative if $p$ and $q$ are either the identity or $\sigma_2^{-2}$. In general, associativity reduces to the equation in the braid group
$$
\begin{tikzpicture}[baseline=30,smallstring,x=0.4cm,y=0.6cm]
\foreach \x in {0,...,5}
{
	\draw (\x,0) -- (\x,5);
}
\draw[fill=white]  (-0.5,3) rectangle node {$\Delta_{1,2} p$} (5.5,4);
\draw[fill=white]  (-0.5,1) rectangle node {$p$} (3.5,2);
\end{tikzpicture}
=
\begin{tikzpicture}[baseline=30,smallstring,x=0.4cm,y=0.6cm]
\foreach \x in {0,...,5}
{
	\draw (\x,0) -- (\x,5);
}
\draw[fill=white]  (-0.5,3) rectangle node {$\Delta_{3,4} p$} (5.5,4);
\draw[fill=white]  (1.5,1) rectangle node {$p$} (5.5,2);
\end{tikzpicture}
$$
where $\Delta_{1,2} p$ denotes $p$ with the first two strands doubled, and similarly $\Delta_{3,4} p$ denotes $p$ with the last two strands doubled. We have not found any other solutions, nor been able to rule them out. 
This seems to be an interesting problem in braid theory!
In order to have braided interchange, we see that $p$ and $q$ must satisfy the equation
$$
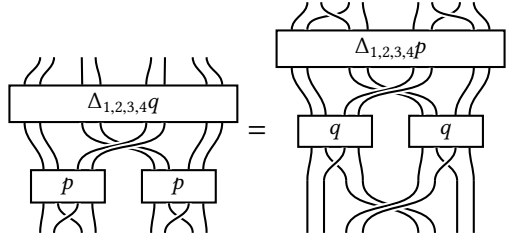
\begin{tikzpicture}[baseline=40,smallstring]
\node (bottom-1) at (0,0) {};
\node (bottom-2) at (.25,0) {};
\node (bottom-3) at (.75,0) {};
\node (bottom-4) at (1,0) {};
\node (bottom-5) at (2,0) {};
\node (bottom-6) at (2.25,0) {};
\node (bottom-7) at (2.75,0) {};
\node (bottom-8) at (3,0) {};
\node[draw,rectangle] (p-L) at (.5,1) {$\quad p\quad$};
\node[draw,rectangle] (p-R) at (2.5,1) {$\quad p\quad$};
\node (top-1) at (0,3.5) {};
\node (top-2) at (.25,3.5) {};
\node (top-3) at (.75,3.5) {};
\node (top-4) at (1,3.5) {};
\node (top-5) at (2,3.5) {};
\node (top-6) at (2.25,3.5) {};
\node (top-7) at (2.75,3.5) {};
\node (top-8) at (3,3.5) {};
\draw (bottom-1) to[in=-90,out=90] (p-L.-145);
\draw[knot] (bottom-3) to[in=-90,out=90] (p-L.-120);
\draw[knot] (bottom-2) to[in=-90,out=90] (p-L.-60);
\draw (bottom-4) to[in=-90,out=90] (p-L.-35);
\draw (bottom-5) to[in=-90,out=90] (p-R.-145);
\draw[knot] (bottom-7) to[in=-90,out=90] (p-R.-120);
\draw[knot] (bottom-6) to[in=-90,out=90] (p-R.-60);
\draw (bottom-8) to[in=-90,out=90] (p-R.-35);
\node[draw,rectangle, fill=white] (q) at (1.5,2.5) {$\quad\qquad \Delta_{1,2,3,4} q \qquad\quad$};
\draw (p-L.145) to[in=-90,out=90] (q.-169);
\draw (p-L.120) to[in=-90,out=90] (q.-167);
\draw[knot] (p-R.145) to[in=-90,out=90] (q.-155);
\draw[knot] (p-R.120) to[in=-90,out=90] (q.-140);
\draw[knot] (p-L.60) to[in=-90,out=90] (q.-40);
\draw[knot] (p-L.35) to[in=-90,out=90] (q.-25);
\draw (p-R.60) to[in=-90,out=90] (q.-13);
\draw (p-R.35) to[in=-90,out=90] (q.-11);
\draw (q.169) to[in=-90,out=90] (top-1);
\draw (q.167) to[in=-90,out=90] (top-2);
\draw (q.155) to[in=-90,out=90] (top-3);
\draw (q.140) to[in=-90,out=90] (top-4);
\draw (q.40) to[in=-90,out=90] (top-5);
\draw (q.25) to[in=-90,out=90] (top-6);
\draw (q.13) to[in=-90,out=90] (top-7);
\draw (q.11) to[in=-90,out=90] (top-8);
\end{tikzpicture}
=
\begin{tikzpicture}[baseline=40,smallstring]
\node (bottom-1) at (0,0) {};
\node (bottom-2) at (.3,0) {};
\node (bottom-3) at (.7,0) {};
\node (bottom-4) at (1,0) {};
\node (bottom-5) at (2,0) {};
\node (bottom-6) at (2.3,0) {};
\node (bottom-7) at (2.7,0) {};
\node (bottom-8) at (3,0) {};
\coordinate (mid-1) at (0,1.1);
\coordinate (mid-2) at (.3,1.1);
\coordinate (mid-3) at (.7,1.1);
\coordinate (mid-4) at (1,1.1);
\coordinate (mid-5) at (2,1.1);
\coordinate (mid-6) at (2.3,1.1);
\coordinate (mid-7) at (2.7,1.1);
\coordinate (mid-8) at (3,1.1);
\node[draw,rectangle] (q-L) at (.5,2) {$\quad q\quad$};
\node[draw,rectangle] (q-R) at (2.5,2) {$\quad q\quad$};
\node[draw,rectangle, fill=white] (p) at (1.5,3.5) {$\quad\qquad \Delta_{1,2,3,4} p \qquad\quad$};
\node (top-1) at (0,4.5) {};
\node (top-2) at (.25,4.5) {};
\node (top-3) at (.75,4.5) {};
\node (top-4) at (1,4.5) {};
\node (top-5) at (2,4.5) {};
\node (top-6) at (2.25,4.5) {};
\node (top-7) at (2.75,4.5) {};
\node (top-8) at (3,4.5) {};
\draw (bottom-1) to[in=-90,out=90] (mid-1);
\draw (bottom-2) to[in=-90,out=90] (mid-2);
\draw[knot] (bottom-5) to[in=-90,out=90] (mid-3);
\draw[knot] (bottom-6) to[in=-90,out=90] (mid-4);
\draw[knot] (bottom-3) to[in=-90,out=90] (mid-5);
\draw[knot] (bottom-4) to[in=-90,out=90] (mid-6);
\draw (bottom-7) to[in=-90,out=90] (mid-7);
\draw (bottom-8) to[in=-90,out=90] (mid-8);
\draw (mid-1) to[in=-90,out=90] (q-L.-145);
\draw[knot] (mid-3) to[in=-90,out=90] (q-L.-120);
\draw[knot] (mid-2) to[in=-90,out=90] (q-L.-60);
\draw (mid-4) to[in=-90,out=90] (q-L.-35);
\draw (mid-5) to[in=-90,out=90] (q-R.-145);
\draw[knot] (mid-7) to[in=-90,out=90] (q-R.-120);
\draw[knot] (mid-6) to[in=-90,out=90] (q-R.-60);
\draw (mid-8) to[in=-90,out=90] (q-R.-35);
\draw (q-L.145) to[in=-90,out=90] (p.-169);
\draw (q-L.120) to[in=-90,out=90] (p.-167);
\draw[knot] (q-R.145) to[in=-90,out=90] (p.-155);
\draw[knot] (q-R.120) to[in=-90,out=90] (p.-140);
\draw[knot] (q-L.60) to[in=-90,out=90] (p.-40);
\draw[knot] (q-L.35) to[in=-90,out=90] (p.-25);
\draw (q-R.60) to[in=-90,out=90] (p.-13);
\draw (q-R.35) to[in=-90,out=90] (p.-11);
\draw (p.169) to[in=-90,out=90] (top-1);
\draw[knot] (p.167) to[in=-90,out=90] (top-3);
\draw[knot] (p.155) to[in=-90,out=90] (top-2);
\draw (p.140) to[in=-90,out=90] (top-4);
\draw (p.40) to[in=-90,out=90] (top-5);
\draw[knot] (p.25) to[in=-90,out=90] (top-7);
\draw[knot] (p.13) to[in=-90,out=90] (top-6);
\draw (p.11) to[in=-90,out=90] (top-8);
\end{tikzpicture}
$$
and one readily shows that this is not the case for $p,q \in \{ \id, \sigma_2^{-2} \}$.

Thus at this point, we can only say that the obvious approaches to forming the product of $\cV$-monoidal categories do not work, and ruling out any product requires some progress on equations in braid groups.
%
%
%

\subsection{\texorpdfstring{$\cV$}{V}-monoidal functors}

A (strictly unital) \emph{$\cV$-monoidal functor} $(\cF, \alpha)$ between $\cV$-monoidal categories $\cF: \cC \to \cD$ has an underlying
$\cV$-functor $\cF: \cC \to \cD$ such that $\cF(1_\cC)=1_\cD$ and $\cF_{1_\cC \to 1_\cC} = j_{1_\cD}$,
along
with a family of $\id_\cV$-graded isomorphisms $\alpha_{a,b}: \id_\cV \to \cD(\cF(ab) \to \cF(a)\cF(b))$ with
$\alpha_{1_\cC, 1_\cC}=j_{1_\cD}$
satisfying the naturality condition
\begin{equation}
\label{eq:NaturalityForVMonoidal}
\begin{tikzcd}
\cC(a \to c)\cC(b \to d)
	\ar[r, "-\otimes_\cC-"]
	\ar[d,"\cF_{a\to c}\cF_{b\to d}"]
& \cC(ab \to cd) \ar[dr, "\cF_{ab \to cd}"]	
&
\\
\cD(\cF(a)\to \cF(c)) \cD(\cF(b)\to \cF(d)) \ar[dr, "-\otimes_\cD-"]
& &
\cD(\cF(ab) \to \cF(cd)) \ar[d, "-\circ \alpha_{c,d}"]	  
\\
& 
\cD(\cF(a)\cF(b) \to \cF(c)\cF(d)) \ar[r, "\alpha_{a,b}\circ-"] 
&
\cD(\cF(ab) \to \cF(c)\cF(d)).
\end{tikzcd}
\end{equation}
and the associativity condition
\begin{equation}
\begin{tikzcd}
&
\cD(\cF(abc) \to \cF(a)\cF(bc))\cD(\cF(a)\cF(bc)\to \cF(a)\cF(b)\cF(c))  
  \ar[dr, "-\circ_\cD-"]
\\
1_\cV 
  \ar[ur, "\alpha_{a,bc}\left((j_{\cF(a)}\alpha_{b,c}) \circ (-\otimes_\cD-)\right)"]
  \ar[dr, swap, "\alpha_{ab,c}\left((\alpha_{a,b}j_{\cF(c)}) \circ (-\otimes_\cD-)\right)"]
&&[-40pt]
\cD(\cF(abc) \to \cF(a)\cF(b)\cF(c))
\\
&
\cD(\cF(abc) \to \cF(ab)\cF(c))\cD(\cF(ab)\cF(c)\to \cF(a)\cF(b)\cF(c)) 
  \ar[ur, swap, "-\circ_\cD-"]
\end{tikzcd}
\end{equation}

We say $(\cF,\alpha)$ is strict if $\alpha_{a,b} = j_{\cF(ab)}$ for all $a,b$; in this case the associativity condition
holds automatically.

We will see in Lemma \ref{lem:SeparateNaturalityForVMonoidal} below that the naturality condition above is equivalent to two separate naturality conditions, which are easier to verify.

\begin{remark}
\label{rem:CTimesC}
Usually one thinks of a monoidal functor as a functor $\cF: \cC\to \cD$ together with a natural isomorphism
between the functors $\cC\boxtimes \cC \to \cD$ given by $(\cF\boxtimes \cF)\circ (-\otimes_\cD-)$ and $(-\otimes_\cC -)\circ \cF$.

Unfortunately, this point of view is not available in the braided enriched setting.
As discussed in Section \ref{sec:NoProductsOfVMonoidalCategories}, there is no suitable notion of product of $\cV$-monoidal categories, unless $\cV$ happens to be symmetric.

Nevertheless, by the separate naturality conditions given in Lemma \ref{lem:SeparateNaturalityForVMonoidal}, we see a monoidal functor is a functor $\cF: \cC \to \cD$ together with a family of isomorphisms $\alpha_{a,b}: \id_\cV \to \cD(\cF(ab) \to \cF(a)\cF(b))$ such that
for every $a$, $$\alpha_{a,-}: \cF(a) \otimes_\cD \cF(-) \Rightarrow \cF(a \otimes_\cC -)$$ is a natural isomorphism between functors $\cC \to \cD$,
and similarly for every $b$, 
$$
\alpha_{-, b}: \cF(-) \otimes_\cD \cF(b) \Rightarrow \cF(- \otimes_\cC b)
$$ 
is a natural isomorphism.
\end{remark}

\begin{lem}
The composite of $\cV$-monoidal functors $(\cF, \alpha): \cC \to \cD$ and $(\cG, \gamma) : \cD \to \cE$ given by $(\cF\circ \cG, \delta): \cC \to \cE$ with
\begin{align*}
\delta_{a,b} =
  1_\cV
    \xrightarrow{\alpha_{a,b} \gamma_{\cF(a), \cF(b)}}
  & \cD(\cF(ab) \to \cF(a)\cF(b)) \cE(\cG(\cF(a)\cF(b)) \to \cG(\cF(a))\cG(\cF(a)))	\\
  \xrightarrow{\cG_{\cF(ab) \to \cF(a)\cF(b)} \id}
  & \cE(\cG(\cF(ab)) \to \cG(\cF(a)\cF(b))) \cE(\cG(\cF(a)\cF(b)) \to \cG(\cF(a))\cG(\cF(a))) \\
  \xrightarrow{-\circ_\cE-}
  & \cE(\cG(\cF(ab)) \to \cG(\cF(a))\cG(\cF(a)))
\end{align*}
is a $\cV$-monoidal functor, and this composition is associative.
\end{lem}
\begin{proof}
We can check two separate naturality conditions for $\delta$, by Lemma \ref{lem:SeparateNaturalityForVMonoidal}. These are then automatic: if we fix $a$, then $\delta_{a, -}$ is exactly the horizontal composition of $\alpha_{a, -}$ and $\gamma_{\cF(a),-}$, so it is natural by Lemma \ref{lem:HorizontalComposition}; similarly when we fix $b$.

To check associativity, suppose we have three $\cV$-monoidal functors
$$
\cB \xrightarrow{(\cF, \kappa)} \cC \xrightarrow{(\cG, \lambda)} \cD \xrightarrow{(\cH, \mu)} \cE.
$$
We see 
$((\cF,\kappa)\circ (\cG, \lambda)) \circ (\cH, \mu) = (\cF\circ \cG\circ \cH, \phi)$ and
$(\cF,\kappa)\circ ((\cG, \lambda) \circ (\cH, \mu)) = (\cF\circ \cG\circ \cH, \psi)$ where
$$
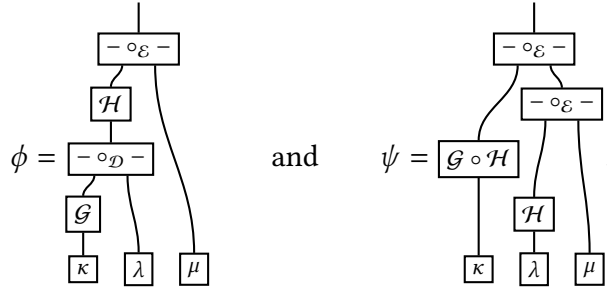
\phi=
\begin{tikzpicture}[baseline=40,smallstring]
\node[draw,rectangle] (k) at (0,0) {$\kappa$};
\node[draw,rectangle] (l) at (1,0) {$\lambda$};
\node[draw,rectangle] (m) at (2,0) {$\mu$};
\node[draw,rectangle] (G) at (0,1) {$\cG$};
\node[draw,rectangle] (circ-1) at (.5,2) {$-\circ_{\cD} -$};
\node[draw,rectangle] (H) at (.5,3) {$\cH$};
\node[draw,rectangle] (circ-2) at (1,4) {$-\circ_{\cE} -$};
\node (top) at (1,5) {};
\draw (k) to[in=-90,out=90] (G);
\draw (G) to[in=-90,out=90] (circ-1.-135);
\draw (l) to[in=-90,out=90] (circ-1.-45);
\draw (circ-1) to[in=-90,out=90] (H);
\draw (H) to[in=-90,out=90] (circ-2.-135);
\draw (m) to[in=-90,out=90] (circ-2.-45);
\draw (circ-2) to[in=-90,out=90] (top);
\end{tikzpicture}
\qquad
\text{and}
\qquad
\psi=
\begin{tikzpicture}[baseline=40,smallstring]
\node[draw,rectangle] (k) at (0,0) {$\kappa$};
\node[draw,rectangle] (l) at (1,0) {$\lambda$};
\node[draw,rectangle] (m) at (2,0) {$\mu$};
\node[draw,rectangle] (GH) at (0,2) {$\cG\circ \cH$};
\node[draw,rectangle] (H) at (1,1) {$\cH$};
\node[draw,rectangle] (circ-1) at (1.5,3) {$-\circ_{\cE} -$};
\node[draw,rectangle] (circ-2) at (1,4) {$-\circ_{\cE} -$};
\node (top) at (1,5) {};
\draw (k) to[in=-90,out=90] (GH);
\draw (l) to[in=-90,out=90] (H);
\draw (H) to[in=-90,out=90] (circ-1.-135);
\draw (m) to[in=-90,out=90] (circ-1.-45);
\draw (circ-1) to[in=-90,out=90] (circ-2.-45);
\draw (GH) to[in=-90,out=90] (circ-2.-135);
\draw (circ-2) to[in=-90,out=90] (top);
\end{tikzpicture}
.
$$
These two morphisms are equal by associativity of composition and functoriality of $\cH$.
\end{proof}

There are similar notions of oplax and lax $\cV$-monoidal functor, which we will not pursue here.
For now, we omit a discussion of natural transformations between
$\cV$-monoidal functors.

\subsection{Rigidity}
\label{sec:rigidity}

A $\cV$-monoidal category $\cC$ is \emph{rigid} 
if for every $c\in\cC$ there is:
\begin{itemize}
\item a \emph{dual} object $c^*\in\cC$ and $\id_\cV$-graded morphisms $\ev_c: \id_\cV \to \cC
(cc^* \to \id_\cC)$ and $\coev_c: \id_\cV \to \cC(\id_\cC \to c^* c)$ which satisfy the zig-zag axioms: the identity element $j_c:\id_\cV \to \cC(c\to c)$ is equal to the composite
\begin{align*}
\id_\cV \xrightarrow{j_c\coev_c\ev_cj_c} & \cC(c\to c)\cC(\id_\cC \to c^*c)\cC(cc^*\to \id_\cC) \cC(c\to c)\\
\xrightarrow{(- \otimes_\cC -)(- \otimes_\cC -)} & \cC(c\to cc^*c) \cC(cc^*c\to c) \\
\xrightarrow{(- \circ_\cC -)} & \cC(c\to c),
\end{align*}
and similarly $j_{c^*}=\left(\coev_c j_{c^*} j_{c^*}\ev_c\right)\circ\left((-\otimes_\cC -)(-\otimes_\cC
-)\right)\circ
\left(-\circ_\cC-\right)$.
\item
a \emph{predual} object $c_*\in\cC$ and $\id_\cV$-graded morphisms $\ev_c: \id_\cV \to \cC
(c_* c \to \id_\cC)$ and $\coev_c: \id_\cV \to \cC(\id_\cC \to c c_*)$ which satisfy similar zig-zag axioms.
\end{itemize}

\begin{remark}
\label{rem:StrictDuals}
For the purposes of this article, we assume that $(ba)^*=a^*b^*$ for all $a,b\in \cC$.
As with ordinary rigid monoidal categories, it is easy to see that the dual is unique up to unique isomorphism.
Thus one can always arrange the choices of dual objects in this way.
\end{remark}

Notice that this means there is an equivalence of $\cV$-monoidal categories $*:\cC\to \cC^{\text{rot}}$, where $\cC^{\text{rot}}$ is the rotation of $\cC$ discussed in Section \ref{sec:RotatingC}.
Here, $*_{a \to b}: \cC(a\to b)\to \cC(b^*\to a^*)$ is given by
{\scriptsize{
\begin{align*}
\cC(a\to b) \xrightarrow{\coev_a j_{b^*}j_{a^*} \id j_{b^*}j_{a^*}\ev_b } & \cC(\id_\cC \to a^*a)\cC(b^*\to b^*)\cC(a^*\to a^*) \cC(a\to b)\cC(b^*\to b^*) \cC(a^*\to a^*) \cC(bb^*\to \id_\cC)
\\
\xrightarrow{( -\otimes_\cC -)(-\otimes_\cC -\otimes_\cC -) (-\otimes_\cC -)} & \cC(b^* \to a^*ab^*) \cC(a^*ab^*\to a^*bb^*) \cC(a^*bb^* \to a^*)
\\
\xrightarrow{-\circ_\cC-\circ_\cC-} & \cC(b^*\to a^*),
\end{align*}
}}
and there is a similar formula
$$
*^{-1}_{b^*\to a^*} =
\left( j_a\coev_b j_{a} \id j_{b}\ev_b j_b \right) \circ
\left( ( -\otimes_\cC -)(-\otimes_\cC -\otimes_\cC -) (-\otimes_\cC -) \right) \circ
\left( -\circ_\cC-\circ_\cC- \right)
.
$$
By Remark \ref{rem:StrictDuals}, these are strict $\cV$-monoidal functors.

The reader may like to verify that when we are merely enriched in $\Vec$, the map $*_{a \to b}$ is just
\begin{align*}
\begin{tikzpicture}[baseline]
\node[draw,rectangle, thick] (f) at (0,0) {$f$};
\draw (f) -- (0,0.5);
\draw (f) -- (0,-0.5);
\end{tikzpicture}
\mapsto
\begin{tikzpicture}[baseline]
\node[draw,rectangle, thick] (f) at (0,0) {$f$};
\draw (f.north) to[out=90,in=90] ($(f.north)+(0.5,0)$) -- (0.5,-0.5);
\draw (f.south) to[out=-90,in=-90] ($(f.south)+(-0.5,0)$) -- (-0.5,0.5);
\end{tikzpicture}\,.
\end{align*}


As in ordinary rigid monoidal categories, we have \emph{Frobenius reciprocity};
one can build natural isomorphisms $\cC(a\to cb_*)\cong \cC(ab\to c)\cong \cC(b\to a^*c)$.
It is important to note that the verification that these maps are invertible uses the interchange relation.
We leave this enjoyable exercise to the reader.

\begin{remark}
It may appear as though our convention for duals is opposite to
the usual one, but given we are writing composition in the `natural' order 
(unlike in most of mathematics!) this should really be thought of as the usual convention in disguise.
\end{remark}

\section{The underlying monoidal category \texorpdfstring{$\cC^\cV$}{CV}}
\label{sec:UnderlyingTensorCategory}

Suppose $\cC$ is a $\cV$-monoidal category for a braided monoidal
category $\cV$. 
Recall that in the introduction, we defined the ordinary monoidal category $\cC^\cV$ by
$$
\cC^\cV(a \to b) = \cV(1_{\cV} \to \cC(a \to b)).
$$
The identity morphism $\id_a\in \cC^\cV(a\to a) = \cV(1_\cV\to \cC(a\to a))$ is the identity element $j_a$.
Composition of the morphisms $f\in \cC^\cV(a\to b)$ and $g\in \cC^\cV(b\to c)$ is given by
$$
1_\cV \xrightarrow{fg} \cC(a\to b) \cC(b\to c) \xrightarrow{-\circ_\cC-} \cC(a\to c).
$$
Tensor product of the morphisms $f\in \cC^\cV(a\to b)$ and $g\in \cC^\cV(c\to d)$ is given by
$$
1_\cV \xrightarrow{fg} \cC(a\to b) \cC(c\to d) \xrightarrow{-\otimes_\cC-} \cC(ac\to bd).
$$

\subsection{Properties of \texorpdfstring{$\cC^\cV$}{CV}}
It is straightforward to verify that $\cC^\cV$ is an ordinary monoidal category.

\begin{lem}
\label{lem:CVrigid}
If $\cC$ is rigid, then $\cC^\cV$ is rigid.
\end{lem}
\begin{proof}
The (co)evaluation maps for $c\in \cC^\cV$ are the same as those for $c\in\cC$.
\end{proof}

\begin{defn}
\label{defn:FunctorGV}
Suppose $\cC, \cD$ are $\cV$-monoidal categories, and $(\cG,\alpha) : \cC \to \cD$ is a strong $\cV$-monoidal functor.
We construct a functor $\cG^\cV : \cC^\cV \to \cD^\cV$ by $\cG^\cV(c) = \cG(c)$ and
$$
\cG^\cV(f: a\to b) = [1_\cV \xrightarrow{f} \cC(a\to b) \xrightarrow{\cG_{a\to b}} \cD(a\to b)].
$$ 
We see that $\cG^\cV$ is a functor from the axioms of a $\cV$-functor.
Notice that $\alpha^\cV_{a,b}:=\alpha_{a,b}\in \cV(1_\cV\to \cD(\cG(ab)\to \cG(a)\cG(b))) = \cD^\cV(\cG(ab)\to \cG(a)\cG(b))$ endows $\cG^\cV$ with the structure of a strong monoidal functor.
%
Note that if $(\cG, \alpha)$ is a lax or oplax $\cV$-monoidal functor, then so is $(\cG^\cV, \alpha^\cV)$.
\end{defn}

\subsection{The categorified `trace'}

\begin{defn}
For all $a\in \cC$, we define a functor $\cR_a: \cC^\cV\to \cV$ as follows.
On objects, we define $\cR_a(b) = \cC(a\to b)$, and  for a morphism $f\in \cC^\cV(b\to c) = \cV(1_\cV \to \cC(b\to c))$, we define $\cR_a(f)$ as the composite
$$
\cC(a\to b)
\xrightarrow{\id f}
\cC(a\to b)\cC(b\to c)
\xrightarrow{-\circ_\cC -}
\cC(a\to c).
$$
It is straightforward to verify that $\cR_a$ is a functor using the axioms of a $\cV$-enriched category.
\end{defn}

The functor $\cR_{1_\cC}$, which we also denote by $\Tr_\cV$, is of special importance because it is \emph{lax monoidal} with  \emph{laxitor} given by
\begin{equation}
\label{eq:laxitor}
\cC(1_\cC\to a)\cC(1_\cC\to b) \xrightarrow{-\otimes_\cC -} \cC(1_\cC \to ab).
\end{equation}

\begin{remark}
\label{rem:Traciator}
We use the notation $\Tr_\cV$ in the spirit of \cite{1509.02937}; when $\cC$ is $1_\cV$-graded pivotal, i.e., there is a $1_\cV$-graded natural isomorphism $\id_\cC \cong **$, we get a $1_\cV$-graded \emph{traciator} isomorphism $\tau_{a,b}:\Tr_\cV(ab) \to \Tr_\cV(ba)$ by the composite
$$
\Tr_\cV(ab) = \cC(1_\cC\to ab) \cong \cC(b^*\to a)\cong \cC(1_\cC\to b^{**}a)\cong \cC(1_\cC\to ba) = \Tr_\cV(ba).
$$
When $\cC$ is merely rigid, we only get an isomorphism $\Tr_\cV(ab) \to \Tr_\cV(b^{**}a)$.
When $\cC$ is not rigid, we get no such isomorphism.
Since our main theorem requires rigidity of $\cC$, using $\Tr_\cV$ is only a slight abuse of notation.
\end{remark}

\section{Adjunctions, mates, and evaluations}
\label{sec:Mates}

Suppose we have an adjunction $\cA(\cL(v) \to b) \cong \cB(v\to \cR(b))$.
Recall that the mate of $f\in  \cA(v\to \cR(b))$ is the corresponding map in $\cB(v\to \cR(b))$.
We record the following remark for later use:

\begin{remark}
\label{rem:Mates}
A straightforward calculation using naturality shows that the mate of $f$ is equal to $\cL(f)$ precomposed with the mate of $\id_{\cR(b)}\in \cB(\cR(b)\to \cR(b))$.
More generally, for $f_1\in  \cA(v\to \cR(b))$ and $f_2\in \cA(\cR(b)\to \cR(c))$, we have $\mate(f_1\circ f_2) = \cL(f_1)\circ\mate(f_2)$.

Similarly, the mate of $g\in \cB(\cL(v) \to b)$ is equal to $\cR(g)$ postcomposed with the mate of $\id_{\cL(v)} \in \cA(\cL(v)\to \cL(v))$. 
More generally, for $g_1\in  \cB(\cL(u)\to \cL(v))$ and $g_2\in \cB( \cL(v)\to b)$, we have $\mate(g_1\circ g_2) = \mate(g_1)\circ\cR(g_2)$.
\end{remark}

\subsection{Left adjoints of the trace}

We now assume that the categorified trace $\Tr_\cV: \cC^\cV\to \cV$ given by $\Tr_\cV(c) = \cC(1_\cC\to c)$ has a left adjoint, which we denote by $\cF: \cV\to \cC^\cV$.
This means we have an adjunction given by
\begin{equation}
\label{eq:TraceAdjunction}
\cC^\cV(\cF(v) \to c) \cong \cV(v\to \Tr_\cV(c)).
\end{equation}

\begin{remark}
\label{rem:LeftAdjointExists}
In the future, we hope to develop the theory of semisimplicity for $\cV$-enriched categories to the point that one can say that if $\cV$ is fusion, and $\cC$ is $\cV$-fusion, then this left adjoint $\cF$ automatically exists, and
$$
\cF(v) = \bigoplus_{c \in \Irr(\cC^\cV)} \cV(v \to \Tr_\cV(c)) c,
$$
suitably interpreted.
\end{remark}

\begin{defn}
Setting $c=\cF(v)$ in Adjunction \eqref{eq:TraceAdjunction}, the mate of $\id_{\cF(v)}\in \cC^\cV(\cF(v)\to \cF(v))$ is a canonical map $\eta_v: v\to \cV(v\to \Tr_\cV(\cF(v)))$ called the \emph{unit} of the adjunction.
\end{defn}

\begin{lem}
\label{lem:Oplaxitor}
The functor $\cF$ comes equipped with the structure of an oplax monoidal functor.
\end{lem}
\begin{proof}
This is well-known; the left adjoint of a lax monoidal functor is oplax, and the right adjoint of an oplax monoidal functor is lax \cite{MR0360749}.
For the reader's convenience, the \emph{oplaxitor} map $\mu_{u,v}\in\cC^\cV(\cF(uv) \to \cF(u)\cF(v))$ is given explicitly as the mate of 
$$
uv \xrightarrow{\eta_u\eta_v} \Tr_\cV(\cF(u))\Tr_\cV(\cF(v)) \xrightarrow{\text{laxitor from \eqref{eq:laxitor}}} \Tr_\cV(\cF(u)\cF(v))
$$ 
under Adjunction \eqref{eq:TraceAdjunction}.
(Observe that there's no way to build a map in the other direction.) 
Associativity of the oplax structure comes from the associativity of the tensor products in $\cC$ and
$\cV$, and associativity of the laxitor.
\end{proof}

\subsection{Extra structure from rigidity}
\label{sec:ExtraStructureFromRigidity}

We now assume that $\cC$ is rigid, so $\cC^\cV$ is also rigid by Lemma \ref{lem:CVrigid}.
We then see that the functor $\cR_a: \cC^\cV\to \cV$ given by $\cR_a(b) = \cC(a\to b)$ has left adjoint $\cL_a : \cV \to \cC^\cV$ given by $v\mapsto a\cF(v)$.
Indeed,
\begin{equation}
\label{eq:La-Adjunction}
\begin{split}
\cC^\cV(\cL_a(v) \to b) 
&=
\cC^\cV(a\cF(v) \to b)
\\&\cong
\cC^\cV(\cF(v) \to a^*b)
\\&\cong
\cV(v \to \Tr_\cV(a^*b))
\\&=
\cV(v \to \cC(1_\cC\to a^*b))
\\&\cong
\cV(v \to \cC(a\to b))
\\&=
\cV(v \to \cR_a(b)).
\end{split}
\end{equation}

We introduce the following notation to make our diagrams easier to read:
\begin{align*}
\{a \to b\} &\overset{\text{def}}{=} \cF(\cC(a \to b)).
\\
\{a \to b;c\to d;\cdots \} &\overset{\text{def}}{=} \cF(\cC(a \to b)\cC(c \to d) \cdots).
\end{align*}

\begin{defn}
The evaluation morphism (or counit) $\varepsilon_{a \to b}\in \cC^\cV(a \{a \to b\} \to b)$ is
the mate of the identity $\cV(\cC(a \to b)\to \cC(a \to b))$ under Adjunction \eqref{eq:La-Adjunction}.
\end{defn}

\subsection{Computing mates}
We now compute the mates of the composition and tensor product morphisms.
The proof of the following lemma is surprisingly difficult compared to the simplicity of its statement.
We defer its proof to Appendix \ref{sec:ProofsAdjunctionsAndMates}.

\begin{lem}
\label{lem:MateOfId - aFv}
The mate of $\id_{a\cF(v)}\in \cC^\cV(a\cF(v)\to a\cF(v))$ is given by $$(j_a \eta_v) \circ (-\otimes_\cC -)\in \cV(v\to
\cC(a\to a\cF(v))).$$
\end{lem}

By the second half of Remark \ref{rem:Mates}, we get the following corollary.

\begin{cor}
\label{cor:MateUsingTrace}
The mate of $f\in \cC^\cV(a\cF(v)\to b)$ is given by 
$
\mate(\id_{a\cF(v)}) \circ \Tr_\cV(f)
=
\begin{tikzpicture}[baseline=40,smallstring]
\node (1-L) at (0,0) {$\iV$};
\node (a-b-bottom) at (1,0) {$v$};
\node (1-R) at (2,0) {$\iV$};
\node[draw,rectangle] (j-a) at (0,1) {$j_a$};
\node[draw,rectangle] (eta) at (1,1) {$\eta_{v}$};
\node[draw,rectangle] (f) at (2,1) {$f$};
\node[draw,rectangle] (otimes) at (.5,2) {$-\otimes_\cC -$};
\node[draw,rectangle] (circ) at (1,3) {$-\circ_\cC -$};
\node (a-b-top) at (1,4) {$\cC(a\to b)$};
\draw (j-a) to[in=-90,out=90] (otimes.-135);
\draw (a-b-bottom) to[in=-90,out=90] (eta);
\draw (eta) to[in=-90,out=90] (otimes.-45);
\draw (otimes) to[in=-90,out=90] (circ.-135);
\draw (f) to[in=-90,out=90] (circ.-45);
\draw (circ) to[in=-90,out=90] (a-b-top);
\end{tikzpicture}.
$
\end{cor}

\begin{cor}
\label{cor:Simplify-j-eta-ev}
As a particular instance of Corollary \ref{cor:MateUsingTrace}, we have
$
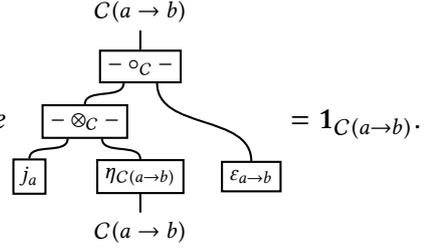
\begin{tikzpicture}[baseline=40,smallstring]
\node (1-L) at (0,0) {$\iV$};
\node (a-b-bottom) at (2,0) {$\cC(a\to b)$};
\node (1-R) at (4,0) {$\iV$};
\node[draw,rectangle] (j-a) at (0,1) {$j_a$};
\node[draw,rectangle] (eta) at (2,1) {$\eta_{\cC(a\to b)}$};
\node[draw,rectangle] (ev-a-b) at (4,1) {$\varepsilon_{a\to b}$};
\node[draw,rectangle] (otimes) at (1,2) {$-\otimes_\cC -$};
\node[draw,rectangle] (circ) at (2,3) {$-\circ_\cC -$};
\node (a-b-top) at (2,4) {$\cC(a\to b)$};
\draw (j-a) to[in=-90,out=90] (otimes.-135);
\draw (a-b-bottom) to[in=-90,out=90] (eta);
\draw (eta) to[in=-90,out=90] (otimes.-45);
\draw (otimes) to[in=-90,out=90] (circ.-135);
\draw (ev-a-b) to[in=-90,out=90] (circ.-45);
\draw (circ) to[in=-90,out=90] (a-b-top);
\end{tikzpicture}
=
\id_{\cC(a\to b)}.
$
\end{cor}
\begin{proof}
By the previous corollary, the mate of $\varepsilon_{a\to b}$ is given by the diagram on the left.
But the mate of $\varepsilon_{a\to b}$ is $\id_{\cC(a\to b)}$ by definition.
\end{proof}

The proof of the following proposition is also quite involved and deferred to Appendix \ref{sec:ProofsAdjunctionsAndMates}.

\begin{prop}
\label{prop:MateOfComposition}
The mate of the composition map $(-\circ_\cC-)\in \cV( \cC(a\to b)\cC(b\to c) \to \cC(a\to c))$ under Adjunction \eqref{eq:La-Adjunction}
with $v= \cC(a\to b)\cC(b\to c)$ and $b=c$ is given by
$$
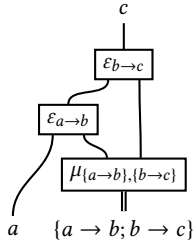
\begin{tikzpicture}[baseline=50,smallstring]
\node (a) at (0,0) {$a$};
\node (F) at (2,0) {$\{a\to b; b\to c\}$};
\node[draw,rectangle] (m) at (2,1) {$\mu_{\{a\to b\}, \{b\to c\}} $};
\draw[double] (F) to[in=-90,out=90] (m.270);
\node[draw,rectangle] (ev-a-b) at (1,2) {$\varepsilon_{a\to b}$};
\draw (a) to[in=-90,out=90] (ev-a-b.-135);
\draw (m.135) to[in=-90,out=90] (ev-a-b.-45);
\node[draw,rectangle] (ev-b-c) at (2,3) {$\varepsilon_{b\to c}$};
\draw (m.45) to[in=-90,out=90] (ev-b-c.-45);
\draw (ev-a-b) to[in=-90,out=90] (ev-b-c.-135);
\node (c) at (2,4) {$c$};
\draw (ev-b-c) to[in=-90,out=90] (c);
\end{tikzpicture}
$$
\end{prop}

We use these examples to prove the following helpful lemma.

\begin{lem}
\label{lem:AdjunctionCompatibility}
Suppose 
$f\in \cC^\cV(a\cF(u)\to b)$ and $g\in \cC^\cV(b\cF(v)\to c)$
have mates under Adjunction \eqref{eq:La-Adjunction} given respectively by 
$\tilde{f}\in \cV(u\to \cC(a\to b))$ and $\tilde{g}\in \cV(v\to \cC(b\to c))$.
Then the mate of $(\tilde{f}\tilde{g})\circ (-\circ_\cC-)$ under Adjunction \eqref{eq:TraceAdjunction} is equal to $(\id_a\mu_{u,v})\circ(f\id_{\cF(v)})\circ g$.
$$
\begin{tikzpicture}[baseline=30,smallstring]
\node (u) at (0,0) {$u$};
\node (v) at (1,0) {$v$};
\node[draw, rectangle, fill=white] (f) at (0,1) {$\tilde{f}$};
\node[draw, rectangle, fill=white] (g)  at (1,1) {$\tilde{g}$};
%
%
\draw (u) to [in= -90,out=90] (f);
\draw (v) to [in= -90,out=90] (g);
\node[draw, rectangle] (circ) at (.5,2) {$-\circ_\cC - $};
\node (a-c) at (.5,3) {$\cC(a\to c)$};
\draw (f) to [in= -90,out=90] (circ.-135);
\draw (g) to [in= -90,out=90] (circ.-45);
\draw (circ) to [in= -90,out=90] (a-c);
\end{tikzpicture}
\longleftrightarrow
\begin{tikzpicture}[baseline=30,smallstring]
\node (a) at (0,0) {$a$};
\node (F) at (2,0) {$\cF(uv)$};
\node[draw,rectangle] (m) at (2,1) {$\mu_{u,v} $};
\draw[double] (F) to[in=-90,out=90] (m.270);
\node[draw,rectangle] (f) at (1,2) {$\quad f\quad$};
\draw (a) to[in=-90,out=90] (f.-135);
\draw (m.135) to[in=-90,out=90] (f.-45);
\node[draw,rectangle] (g) at (2,3) {$\quad g\quad$};
\draw (m.45) to[in=-90,out=90] (g.-45);
\draw (f) to[in=-90,out=90] (g.-135);
\node (c) at (2,4) {$c$};
\draw (g) to[in=-90,out=90] (c);
\end{tikzpicture}
$$
\end{lem}
\begin{proof}
By Remark \ref{rem:Mates} and Proposition \ref{prop:MateOfComposition}, the mate of $(\tilde{f}\tilde{g})\circ(-\circ_\cC-)$ is given by
$$
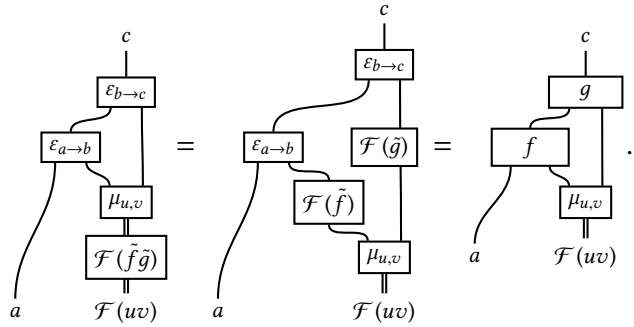
\begin{tikzpicture}[baseline=40,smallstring]
\node (a) at (0,-1) {$a$};
\node (F) at (2,-1) {$\cF(uv)$};
\node[draw,rectangle] (m) at (2,1) {$\mu_{u,v} $};
\draw[double] (F) to[in=-90,out=90] (m.270);
\node[draw,rectangle, fill=white] (G) at (2,0) {$\cF(\tilde{f}\tilde{g})$};
\node[draw,rectangle] (ev-a-b) at (1,2) {$\varepsilon_{a\to b}$};
\draw (a) to[in=-90,out=90] (ev-a-b.-135);
\draw (m.135) to[in=-90,out=90] (ev-a-b.-45);
\node[draw,rectangle] (ev-b-c) at (2,3) {$\varepsilon_{b\to c}$};
\draw (m.45) to[in=-90,out=90] (ev-b-c.-45);
\draw (ev-a-b) to[in=-90,out=90] (ev-b-c.-135);
\node (c) at (2,4) {$c$};
\draw (ev-b-c) to[in=-90,out=90] (c);
\end{tikzpicture}
=
\begin{tikzpicture}[baseline=40,smallstring]
\node (a) at (0,-1) {$a$};
\node (F) at (3,-1) {$\cF(uv)$};
\node[draw,rectangle] (m) at (3,0) {$\mu_{u,v} $};
\draw[double] (F) to[in=-90,out=90] (m.270);
\node[draw,rectangle, fill=white] (G-f) at (2,1) {$\cF(\tilde{f})$};
\node[draw,rectangle] (ev-a-b) at (1,2) {$\varepsilon_{a\to b}$};
\draw (a) to[in=-90,out=90] (ev-a-b.-135);
\draw (m.135) to[in=-90,out=90] (G-f);
\draw (G-f) to[in=-90,out=90] (ev-a-b.-45);
\node[draw,rectangle] (ev-b-c) at (3,3.5) {$\varepsilon_{b\to c}$};
\draw (m.45) to[in=-90,out=90] (ev-b-c.-45);
\node[draw,rectangle, fill=white] (G-g) at (3,2) {$\cF(\tilde{g})$};
\draw (ev-a-b) to[in=-90,out=90] (ev-b-c.-135);
\node (c) at (3,4.5) {$c$};
\draw (ev-b-c) to[in=-90,out=90] (c);
\end{tikzpicture}
=
\begin{tikzpicture}[baseline=40,smallstring]
\node (a) at (0,0) {$a$};
\node (F) at (2,0) {$\cF(uv)$};
\node[draw,rectangle] (m) at (2,1) {$\mu_{u,v} $};
\draw[double] (F) to[in=-90,out=90] (m.270);
\node[draw,rectangle] (f) at (1,2) {$\quad f\quad$};
\draw (a) to[in=-90,out=90] (f.-135);
\draw (m.135) to[in=-90,out=90] (f.-45);
\node[draw,rectangle] (g) at (2,3) {$\quad g\quad$};
\draw (m.45) to[in=-90,out=90] (g.-45);
\draw (f) to[in=-90,out=90] (g.-135);
\node (c) at (2,4) {$c$};
\draw (g) to[in=-90,out=90] (c);
\end{tikzpicture}
.
$$
In the second equality above, we again used two instances of Remark \ref{rem:Mates} for the equalities $\cL_a(\tilde{f})\circ\varepsilon_{a\to b}=f$ and $\cL_b(\tilde{g})\circ \varepsilon_{b\to c}=g$.
\end{proof}

\section{Braided oplax monoidal functors from enriched monoidal categories}
\label{sec:extracting-oplax}

As in the previous section, we assume $\Tr_\cV: \cC^\cV \to \cV$ has left adjoint $\cF : \cV \to \cC^\cV$, and that $\cC$ is rigid, so $\cC^\cV$ is also rigid by Lemma \ref{lem:CVrigid}.
We now show that we can lift $\cF$ to a braided oplax monoidal functor $\cF^{\scriptscriptstyle Z}:\cV\to Z(\cC^\cV)$.

\subsection{Half-braidings}

First, we need half-braidings $e_{c,\cF(v)}:
c\cF(v)\to \cF(v)c$.
The following lemma is in the spirit of \cite[Prop.\,5]{MR2506324} and \cite{1509.02937}, but we need only work in the rigid setting rather than the pivotal setting.

\begin{defn}
For $c\in \cC$ and $v\in \cV$, we define the half-braiding $e_{c,\cF(v)} \in \cC^\cV(c\cF(v)\to \cF(v)c)$ as the mate of the map
\begin{equation}
\label{eq:MateOfHalfBraiding}
v \xrightarrow{\eta_v j_{c}} \cC(1\to \cF(v))\cC(c\to c) \xrightarrow{-\otimes_\cC -} \cC(c\to \cF(v)c)
\end{equation}
in $\cV(v\to \cC(c\to \cF(v)c))=\cV(v\to \cR_c(\cF(v)c))$ under Adjunction \eqref{eq:La-Adjunction}.
\end{defn}

\begin{lem}
\label{lem:HalfBraidings}
The maps $e_{c,\cF(v)}$ are half-braidings for $\cF(v)$, i.e., they are natural isomorphisms which satisfy the hexagon axiom $e_{bc,\cF(v)}=(\id_b e_{c,\cF(v)})\circ (e_{b,\cF(v)} \id_c)$.
\end{lem}
\begin{proof}
First, it is easily verified that $e_{1, \cF(v)} = \id_{\cF(v)}$, since the mate of $\eta_v$ is exactly $\id_{\cF(v)}$.

Next, we verify naturality.
Suppose that $f\in \cC^\cV(a\to b) = \cV(1_\cV\to \cC(a\to b))$.
Then the mate of $e_{\cF(v),a} \circ (\id_{\cF(v)}f)$ is given by
$$
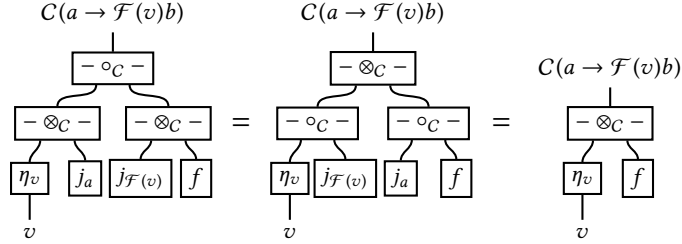
\begin{tikzpicture}[baseline=40,smallstring]
\node (v) at (0,0) {$v$};
\node (1-1) at (1,0) {$\iV$};
\node (1-2) at (2,0) {$\iV$};
\node (1-3) at (3,0) {$\iV$};
\node[draw,rectangle] (eta-v) at (0,1) {$\eta_{v}$};
\node[draw,rectangle] (j-a) at (1,1) {$j_{a}$};
\node[draw,rectangle] (j-Fv) at (2,1) {$j_{\cF(v)}$};
\node[draw,rectangle] (f) at (3,1) {$f$};
\node[draw,rectangle] (otimes-L) at (.5,2) {$-\otimes_\cC -$};
\node[draw,rectangle] (otimes-R) at (2.5,2) {$-\otimes_\cC -$};
\node[draw,rectangle] (circ) at (1.5,3) {$-\circ_\cC -$};
\node (top) at (1.5,4) {$\cC(a\to\cF(v)b)$};
\draw (v) to[in=-90,out=90] (eta-v);
\draw (eta-v) to[in=-90,out=90] (otimes-L.-135);
\draw (j-a) to[in=-90,out=90] (otimes-L.-45);
\draw (otimes-L) to[in=-90,out=90] (circ.-135);
\draw (j-Fv) to[in=-90,out=90] (otimes-R.-135);
\draw (f) to[in=-90,out=90] (otimes-R.-45);
\draw (otimes-R) to[in=-90,out=90] (circ.-45);
\draw (circ) to[in=-90,out=90] (top);
\end{tikzpicture}
\,\,=\,\,
\begin{tikzpicture}[baseline=40,smallstring]
\node (v) at (0,0) {$v$};
\node (1-1) at (1,0) {$\iV$};
\node (1-2) at (2,0) {$\iV$};
\node (1-3) at (3,0) {$\iV$};
\node[draw,rectangle] (eta-v) at (0,1) {$\eta_{v}$};
\node[draw,rectangle] (j-Fv) at (1,1) {$j_{\cF(v)}$};
\node[draw,rectangle] (j-a) at (2,1) {$j_{a}$};
\node[draw,rectangle] (f) at (3,1) {$f$};
\node[draw,rectangle] (circ-L) at (.5,2) {$-\circ_\cC -$};
\node[draw,rectangle] (circ-R) at (2.5,2) {$-\circ_\cC -$};
\node[draw,rectangle] (otimes) at (1.5,3) {$-\otimes_\cC -$};
\node (top) at (1.5,4) {$\cC(a\to\cF(v)b)$};
\draw (v) to[in=-90,out=90] (eta-v);
\draw (eta-v) to[in=-90,out=90] (circ-L.-135);
\draw (j-Fv) to[in=-90,out=90] (circ-L.-45);
\draw (circ-L) to[in=-90,out=90] (otimes.-135);
\draw (j-a) to[in=-90,out=90] (circ-R.-135);
\draw (f) to[in=-90,out=90] (circ-R.-45);
\draw (circ-R) to[in=-90,out=90] (otimes.-45);
\draw (otimes) to[in=-90,out=90] (top);
\end{tikzpicture}
\,\,=\,\,
\begin{tikzpicture}[baseline=40,smallstring]
\node (v) at (0,0) {$v$};
\node (1-1) at (1,0) {$\iV$};
\node[draw,rectangle] (eta-v) at (0,1) {$\eta_{v}$};
\node[draw,rectangle] (f) at (1,1) {$f$};
\node[draw,rectangle] (otimes) at (.5,2) {$-\otimes_\cC -$};
\node (top) at (.5,3) {$\cC(a\to\cF(v)b)$};
\draw (v) to[in=-90,out=90] (eta-v);
\draw (eta-v) to[in=-90,out=90] (otimes.-135);
\draw (f) to[in=-90,out=90] (otimes.-45);
\draw (otimes) to[in=-90,out=90] (top);
\end{tikzpicture}
$$
using the exchange relation.
We now add identity elements and use the exchange relation, followed by Corollary \ref{cor:MateUsingTrace} for $e_{b,\cF(v)}$ to obtain
$$
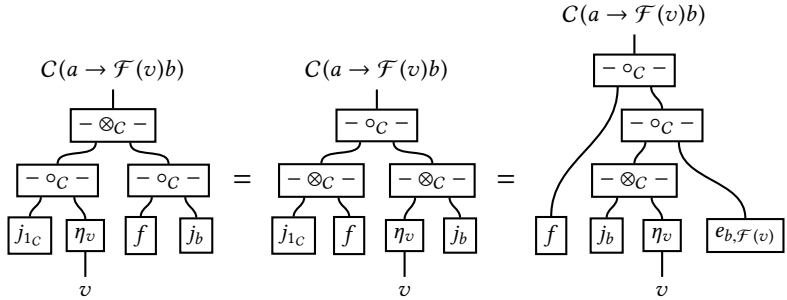
\begin{tikzpicture}[baseline=40,smallstring]
\node (1-1) at (0,0) {$\iV$};
\node (v) at (1,0) {$v$};
\node (1-2) at (2,0) {$\iV$};
\node (1-3) at (3,0) {$\iV$};
\node[draw,rectangle] (j-1) at (0,1) {$j_{1_\cC}$};
\node[draw,rectangle] (eta-v) at (1,1) {$\eta_{v}$};
\node[draw,rectangle] (f) at (2,1) {$f$};
\node[draw,rectangle] (j-b) at (3,1) {$j_{b}$};
\node[draw,rectangle] (circ-L) at (.5,2) {$-\circ_\cC -$};
\node[draw,rectangle] (circ-R) at (2.5,2) {$-\circ_\cC -$};
\node[draw,rectangle] (otimes) at (1.5,3) {$-\otimes_\cC -$};
\node (top) at (1.5,4) {$\cC(a\to\cF(v)b)$};
\draw (v) to[in=-90,out=90] (eta-v);
\draw (j-1) to[in=-90,out=90] (circ-L.-135);
\draw (eta-v) to[in=-90,out=90] (circ-L.-45);
\draw (circ-L) to[in=-90,out=90] (otimes.-135);
\draw (f) to[in=-90,out=90] (circ-R.-135);
\draw (j-b) to[in=-90,out=90] (circ-R.-45);
\draw (circ-R) to[in=-90,out=90] (otimes.-45);
\draw (otimes) to[in=-90,out=90] (top);
\end{tikzpicture}
\,\,=\,\,
\begin{tikzpicture}[baseline=40,smallstring]
\node (1-1) at (0,0) {$\iV$};
\node (1-2) at (1,0) {$\iV$};
\node (v) at (2,0) {$v$};
\node (1-3) at (3,0) {$\iV$};
\node[draw,rectangle] (j-1) at (0,1) {$j_{1_\cC}$};
\node[draw,rectangle] (f) at (1,1) {$f$};
\node[draw,rectangle] (eta-v) at (2,1) {$\eta_{v}$};
\node[draw,rectangle] (j-b) at (3,1) {$j_b$};
\node[draw,rectangle] (otimes-L) at (.5,2) {$-\otimes_\cC -$};
\node[draw,rectangle] (otimes-R) at (2.5,2) {$-\otimes_\cC -$};
\node[draw,rectangle] (circ) at (1.5,3) {$-\circ_\cC -$};
\node (top) at (1.5,4) {$\cC(a\to\cF(v)b)$};
\draw (v) to[in=-90,out=90] (eta-v);
\draw (j-1) to[in=-90,out=90] (otimes-L.-135);
\draw (f) to[in=-90,out=90] (otimes-L.-45);
\draw (otimes-L) to[in=-90,out=90] (circ.-135);
\draw (eta-v) to[in=-90,out=90] (otimes-R.-135);
\draw (j-b) to[in=-90,out=90] (otimes-R.-45);
\draw (otimes-R) to[in=-90,out=90] (circ.-45);
\draw (circ) to[in=-90,out=90] (top);
\end{tikzpicture}
\,\,=\,\,
\begin{tikzpicture}[baseline=40,smallstring]
\node (1-1) at (0,0) {$\iV$};
\node (1-2) at (1,0) {$\iV$};
\node (v) at (2,0) {$v$};
\node (1-3) at (3.5,0) {$\iV$};
\node[draw,rectangle] (f) at (0,1) {$f$};
\node[draw,rectangle] (j-b) at (1,1) {$j_b$};
\node[draw,rectangle] (eta-v) at (2,1) {$\eta_{v}$};
\node[draw,rectangle] (e) at (3.5,1) {$e_{b,\cF(v)}$};
\node[draw,rectangle] (otimes) at (1.5,2) {$-\otimes_\cC -$};
\node[draw,rectangle] (circ-1) at (2,3) {$-\circ_\cC -$};
\node[draw,rectangle] (circ-2) at (1.5,4) {$-\circ_\cC -$};
\node (top) at (1.5,5) {$\cC(a\to \cF(v)b)$};
\draw (v) to[in=-90,out=90] (eta-v);
\draw (f) to[in=-90,out=90] (circ-2.-135);
\draw (j-b) to[in=-90,out=90] (otimes.-135);
\draw (eta-v) to[in=-90,out=90] (otimes.-45);
\draw (otimes) to[in=-90,out=90] (circ-1.-135);
\draw (e) to[in=-90,out=90] (circ-1.-45);
\draw (circ-1) to[in=-90,out=90] (circ-2.-45);
\draw (circ-2) to[in=-90,out=90] (top);
\end{tikzpicture}
$$
Now applying associativity of composition and applying another exchange relation (after adding an identity), we obtain
$$
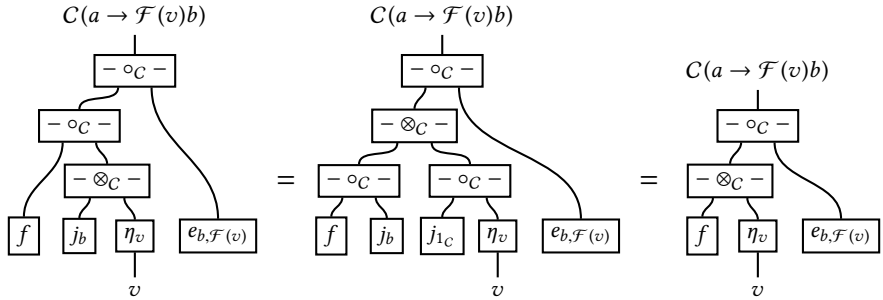
\begin{tikzpicture}[baseline=40,smallstring]
\node (1-1) at (0,0) {$\iV$};
\node (1-2) at (1,0) {$\iV$};
\node (v) at (2,0) {$v$};
\node (1-3) at (3.5,0) {$\iV$};
\node[draw,rectangle] (f) at (0,1) {$f$};
\node[draw,rectangle] (j-b) at (1,1) {$j_b$};
\node[draw,rectangle] (eta-v) at (2,1) {$\eta_{v}$};
\node[draw,rectangle] (e) at (3.5,1) {$e_{b,\cF(v)}$};
\node[draw,rectangle] (otimes) at (1.5,2) {$-\otimes_\cC -$};
\node[draw,rectangle] (circ-1) at (1,3) {$-\circ_\cC -$};
\node[draw,rectangle] (circ-2) at (2,4) {$-\circ_\cC -$};
\node (top) at (2,5) {$\cC(a\to \cF(v)b)$};
\draw (v) to[in=-90,out=90] (eta-v);
\draw (f) to[in=-90,out=90] (circ-1.-135);
\draw (j-b) to[in=-90,out=90] (otimes.-135);
\draw (eta-v) to[in=-90,out=90] (otimes.-45);
\draw (otimes) to[in=-90,out=90] (circ-1.-45);
\draw (e) to[in=-90,out=90] (circ-2.-45);
\draw (circ-1) to[in=-90,out=90] (circ-2.-135);
\draw (circ-2) to[in=-90,out=90] (top);
\end{tikzpicture}
\,\,=\,\,
\begin{tikzpicture}[baseline=40,smallstring]
\node (1-1) at (0,0) {$\iV$};
\node (1-2) at (1,0) {$\iV$};
\node (1-3) at (2,0) {$\iV$};
\node (v) at (3,0) {$v$};
\node (1-4) at (4.5,0) {$\iV$};
\node[draw,rectangle] (f) at (0,1) {$f$};
\node[draw,rectangle] (j-b) at (1,1) {$j_b$};
\node[draw,rectangle] (j-1) at (2,1) {$j_{1_\cC}$};
\node[draw,rectangle] (eta-v) at (3,1) {$\eta_{v}$};
\node[draw,rectangle] (e) at (4.5,1) {$e_{b,\cF(v)}$};
\node[draw,rectangle] (circ-L) at (.5,2) {$-\circ_\cC -$};
\node[draw,rectangle] (circ-R) at (2.5,2) {$-\circ_\cC -$};
\node[draw,rectangle] (otimes) at (1.5,3) {$-\otimes_\cC -$};
\node[draw,rectangle] (circ-top) at (2,4) {$-\circ_\cC -$};
\node (top) at (2,5) {$\cC(a\to \cF(v)b)$};
\draw (v) to[in=-90,out=90] (eta-v);
\draw (f) to[in=-90,out=90] (circ-L.-135);
\draw (j-b) to[in=-90,out=90] (circ-L.-45);
\draw (j-1) to[in=-90,out=90] (circ-R.-135);
\draw (eta-v) to[in=-90,out=90] (circ-R.-45);
\draw (circ-L) to[in=-90,out=90] (otimes.-135);
\draw (circ-R) to[in=-90,out=90] (otimes.-45);
\draw (e) to[in=-90,out=90] (circ-top.-45);
\draw (otimes) to[in=-90,out=90] (circ-top.-135);
\draw (circ-top) to[in=-90,out=90] (top);
\end{tikzpicture}
\,\,=\,\,
\begin{tikzpicture}[baseline=40,smallstring]
\node (1-1) at (0,0) {$\iV$};
\node (v) at (1,0) {$v$};
\node (1-2) at (2.5,0) {$\iV$};
\node[draw,rectangle] (f) at (0,1) {$f$};
\node[draw,rectangle] (eta-v) at (1,1) {$\eta_{v}$};
\node[draw,rectangle] (e) at (2.5,1) {$e_{b,\cF(v)}$};
\node[draw,rectangle] (otimes) at (.5,2) {$-\otimes_\cC -$};
\node[draw,rectangle] (circ) at (1,3) {$-\circ_\cC -$};
\node (top) at (1,4) {$\cC(a\to \cF(v)b)$};
\draw (v) to[in=-90,out=90] (eta-v);
\draw (f) to[in=-90,out=90] (otimes.-135);
\draw (eta-v) to[in=-90,out=90] (otimes.-45);
\draw (otimes) to[in=-90,out=90] (circ.-135);
\draw (e) to[in=-90,out=90] (circ.-45);
\draw (circ) to[in=-90,out=90] (top);
\end{tikzpicture}
$$
Finally, we claim that the right hand side above is exactly the mate of $(f\id_{\cF(v)})\circ e_{b,\cF(v)}$.
Indeed, taking its mate using Remark \ref{rem:Mates} and Corollary \ref{cor:MateUsingTrace} and applying the exchange relation, we get
$$
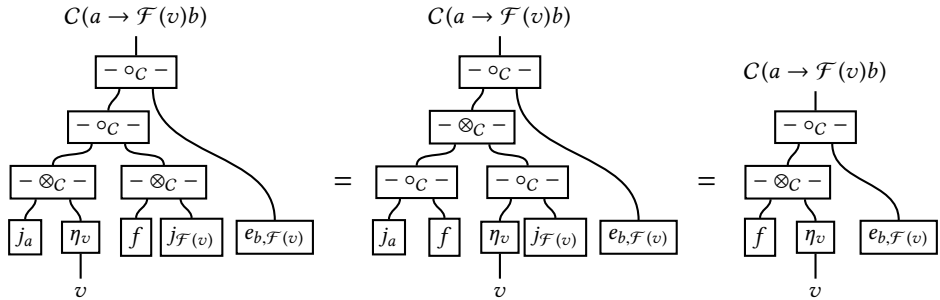
\begin{tikzpicture}[baseline=40,smallstring]
\node (1-1) at (0,0) {$\iV$};
\node (v) at (1,0) {$v$};
\node (1-2) at (2,0) {$\iV$};
\node (1-3) at (3,0) {$\iV$};
\node (1-4) at (4.5,0) {$\iV$};
\node[draw,rectangle] (j-a) at (0,1) {$j_a$};
\node[draw,rectangle] (eta-v) at (1,1) {$\eta_{v}$};
\node[draw,rectangle] (f) at (2,1) {$f$};
\node[draw,rectangle] (j-Fv) at (3,1) {$j_{\cF(v)}$};
\node[draw,rectangle] (e) at (4.5,1) {$e_{b,\cF(v)}$};
\node[draw,rectangle] (otimes-L) at (.5,2) {$-\otimes_\cC -$};
\node[draw,rectangle] (otimes-R) at (2.5,2) {$-\otimes_\cC -$};
\node[draw,rectangle] (circ) at (1.5,3) {$-\circ_\cC -$};
\node[draw,rectangle] (circ-top) at (2,4) {$-\circ_\cC -$};
\node (top) at (2,5) {$\cC(a\to \cF(v)b)$};
\draw (v) to[in=-90,out=90] (eta-v);
\draw (j-a) to[in=-90,out=90] (otimes-L.-135);
\draw (eta-v) to[in=-90,out=90] (otimes-L.-45);
\draw (f) to[in=-90,out=90] (circ-R.-135);
\draw (j-Fv) to[in=-90,out=90] (circ-R.-45);
\draw (otimes-L) to[in=-90,out=90] (circ.-135);
\draw (otimes-R) to[in=-90,out=90] (circ.-45);
\draw (e) to[in=-90,out=90] (circ-top.-45);
\draw (circ) to[in=-90,out=90] (circ-top.-135);
\draw (circ-top) to[in=-90,out=90] (top);
\end{tikzpicture}
\,\,=\,\,
\begin{tikzpicture}[baseline=40,smallstring]
\node (1-1) at (0,0) {$\iV$};
\node (1-2) at (1,0) {$\iV$};
\node (v) at (2,0) {$v$};
\node (1-3) at (3,0) {$\iV$};
\node (1-4) at (4.5,0) {$\iV$};
\node[draw,rectangle] (j-a) at (0,1) {$j_a$};
\node[draw,rectangle] (f) at (1,1) {$f$};
\node[draw,rectangle] (eta-v) at (2,1) {$\eta_{v}$};
\node[draw,rectangle] (j-Fv) at (3,1) {$j_{\cF(v)}$};
\node[draw,rectangle] (e) at (4.5,1) {$e_{b,\cF(v)}$};
\node[draw,rectangle] (circ-L) at (.5,2) {$-\circ_\cC -$};
\node[draw,rectangle] (circ-R) at (2.5,2) {$-\circ_\cC -$};
\node[draw,rectangle] (otimes) at (1.5,3) {$-\otimes_\cC -$};
\node[draw,rectangle] (circ-top) at (2,4) {$-\circ_\cC -$};
\node (top) at (2,5) {$\cC(a\to \cF(v)b)$};
\draw (v) to[in=-90,out=90] (eta-v);
\draw (j-a) to[in=-90,out=90] (circ-L.-135);
\draw (f) to[in=-90,out=90] (circ-L.-45);
\draw (eta-v) to[in=-90,out=90] (circ-R.-135);
\draw (j-Fv) to[in=-90,out=90] (circ-R.-45);
\draw (circ-L) to[in=-90,out=90] (otimes.-135);
\draw (circ-R) to[in=-90,out=90] (otimes.-45);
\draw (e) to[in=-90,out=90] (circ-top.-45);
\draw (otimes) to[in=-90,out=90] (circ-top.-135);
\draw (circ-top) to[in=-90,out=90] (top);
\end{tikzpicture}
\,\,=\,\,
\begin{tikzpicture}[baseline=40,smallstring]
\node (1-1) at (0,0) {$\iV$};
\node (v) at (1,0) {$v$};
\node (1-2) at (2.5,0) {$\iV$};
\node[draw,rectangle] (f) at (0,1) {$f$};
\node[draw,rectangle] (eta-v) at (1,1) {$\eta_{v}$};
\node[draw,rectangle] (e) at (2.5,1) {$e_{b,\cF(v)}$};
\node[draw,rectangle] (otimes) at (.5,2) {$-\otimes_\cC -$};
\node[draw,rectangle] (circ) at (1,3) {$-\circ_\cC -$};
\node (top) at (1,4) {$\cC(a\to \cF(v)b)$};
\draw (v) to[in=-90,out=90] (eta-v);
\draw (f) to[in=-90,out=90] (otimes.-135);
\draw (eta-v) to[in=-90,out=90] (otimes.-45);
\draw (otimes) to[in=-90,out=90] (circ.-135);
\draw (e) to[in=-90,out=90] (circ.-45);
\draw (circ) to[in=-90,out=90] (top);
\end{tikzpicture}
$$

Next, we verify the hexagon axiom.
Since $j_{ab}=(j_a j_b) \circ(-\otimes_\cC -)$, the mate of $e_{ab,\cF(v)}$ is equal to
$$
\begin{tikzpicture}[baseline=40,smallstring]
\node (v) at (0,0) {$v$};
\node (1-1) at (1,0) {$\iV$};
\node (1-2) at (2,0) {$\iV$};
\node[draw,rectangle] (eta-v) at (0,1) {$\eta_{v}$};
\node[draw,rectangle] (j-a) at (1,1) {$j_{a}$};
\node[draw,rectangle] (j-b) at (2,1) {$j_{b}$};
\node[draw,rectangle] (otimes-1) at (1.5,2) {$-\otimes_\cC -$};
\node[draw,rectangle] (otimes-2) at (1,3) {$-\otimes_\cC -$};
\node (top) at (1,4) {$\cC(ab\to ab\cF(v))$};
\draw (v) to[in=-90,out=90] (eta-v);
\draw (j-a) to[in=-90,out=90] (otimes-1.-135);
\draw (j-b) to[in=-90,out=90] (otimes-1.-45);
\draw (otimes-1) to[in=-90,out=90] (otimes-2.-45);
\draw (eta-v) to[in=-90,out=90] (otimes-2.-135);
\draw (otimes-2) to[in=-90,out=90] (top);
\end{tikzpicture}
\,\,=\,\,
\begin{tikzpicture}[baseline=40,smallstring]
\node (v) at (0,0) {$v$};
\node (1-1) at (1,0) {$\iV$};
\node (1-2) at (2,0) {$\iV$};
\node[draw,rectangle] (eta-v) at (0,1) {$\eta_{v}$};
\node[draw,rectangle] (j-a) at (1,1) {$j_{a}$};
\node[draw,rectangle] (j-b) at (2,1) {$j_{b}$};
\node[draw,rectangle] (otimes-1) at (.5,2) {$-\otimes_\cC -$};
\node[draw,rectangle] (otimes-2) at (1,3) {$-\otimes_\cC -$};
\node (top) at (1,4) {$\cC(ab\to ab\cF(v))$};
\draw (v) to[in=-90,out=90] (eta-v);
\draw (eta-v) to[in=-90,out=90] (otimes-1.-135);
\draw (j-a) to[in=-90,out=90] (otimes-1.-45);
\draw (otimes-1) to[in=-90,out=90] (otimes-2.-135);
\draw (j-b) to[in=-90,out=90] (otimes-2.-45);
\draw (otimes-2) to[in=-90,out=90] (top);
\end{tikzpicture}
\,\,=\,\,
\begin{tikzpicture}[baseline=40,smallstring]
\node (1-1) at (0,0) {$\iV$};
\node (v) at (1,0) {$v$};
\node (1-2) at (2.5,0) {$\iV$};
\node (1-3) at (4,0) {$\iV$};
\node[draw,rectangle] (j-a) at (0,1) {$j_{a}$};
\node[draw,rectangle] (eta-v) at (1,1) {$\eta_{v}$};
\node[draw,rectangle] (e) at (2.5,1) {$e_{a,\cF(v)}$};
\node[draw,rectangle] (j-b) at (4,1) {$j_{b}$};
\node[draw,rectangle] (otimes-1) at (.5,2) {$-\otimes_\cC -$};
\node[draw,rectangle] (circ) at (1,3) {$-\circ_\cC -$};
\node[draw,rectangle] (otimes-2) at (2,4) {$-\otimes_\cC -$};
\node (top) at (2,5) {$\cC(ab\to ab\cF(v))$};
\draw (v) to[in=-90,out=90] (eta-v);
\draw (j-a) to[in=-90,out=90] (otimes-1.-135);
\draw (eta-v) to[in=-90,out=90] (otimes-1.-45);
\draw (otimes-1) to[in=-90,out=90] (circ.-135);
\draw (e) to[in=-90,out=90] (circ.-45);
\draw (circ) to[in=-90,out=90] (otimes-2.-135);
\draw (j-b) to[in=-90,out=90] (otimes-2.-45);
\draw (otimes-2) to[in=-90,out=90] (top);
\end{tikzpicture}
$$
The first equality above follows by associativity of tensor product, and the second follows by Corollary \ref{cor:MateUsingTrace} applied to $e_{a,\cF(v)}$.
Applying the exchange relation (after adding an identity) and then using associativity of tensor product again, we obtain
$$
\begin{tikzpicture}[baseline=50,smallstring]
\node (1-1) at (0,0) {$\iV$};
\node (v) at (1,0) {$v$};
\node (1-2) at (2,0) {$\iV$};
\node (1-3) at (3.5,0) {$\iV$};
\node (1-4) at (5,0) {$\iV$};
\node[draw,rectangle] (j-a) at (0,1) {$j_{a}$};
\node[draw,rectangle] (eta-v) at (1,1) {$\eta_{v}$};
\node[draw,rectangle] (j-b-1) at (2,1) {$j_{b}$};
\node[draw,rectangle] (e) at (3.5,1) {$e_{a,\cF(v)}$};
\node[draw,rectangle] (j-b-2) at (5,1) {$j_{b}$};
\node[draw,rectangle] (otimes-1) at (.5,2) {$-\otimes_\cC -$};
\node[draw,rectangle] (otimes-2) at (4.25,2) {$-\otimes_\cC -$};
\node[draw,rectangle] (otimes-3) at (1,3) {$-\otimes_\cC -$};
\node[draw,rectangle] (circ) at (2,4) {$-\circ_\cC -$};
\node (top) at (2,5) {$\cC(ab\to ab\cF(v))$};
\draw (v) to[in=-90,out=90] (eta-v);
\draw (j-a) to[in=-90,out=90] (otimes-1.-135);
\draw (eta-v) to[in=-90,out=90] (otimes-1.-45);
\draw (otimes-1) to[in=-90,out=90] (otimes-3.-135);
\draw (j-b-1) to[in=-90,out=90] (otimes-3.-45);
\draw (otimes-3) to[in=-90,out=90] (circ.-135);
\draw (e) to[in=-90,out=90] (otimes-2.-135);
\draw (j-b-2) to[in=-90,out=90] (otimes-2.-45);
\draw (otimes-2) to[in=-90,out=90] (circ.-45);
\draw (circ) to[in=-90,out=90] (top);
\end{tikzpicture}
\,\,=\,\,
\begin{tikzpicture}[baseline=50,smallstring]
\node (1-1) at (0,0) {$\iV$};
\node (v) at (1,0) {$v$};
\node (1-2) at (2,0) {$\iV$};
\node (1-3) at (3.5,0) {$\iV$};
\node (1-4) at (5,0) {$\iV$};
\node[draw,rectangle] (j-a) at (0,1) {$j_{a}$};
\node[draw,rectangle] (eta-v) at (1,1) {$\eta_{v}$};
\node[draw,rectangle] (j-b-1) at (2,1) {$j_{b}$};
\node[draw,rectangle] (e) at (3.5,1) {$e_{a,\cF(v)}$};
\node[draw,rectangle] (j-b-2) at (5,1) {$j_{b}$};
\node[draw,rectangle] (otimes-1) at (1.5,2) {$-\otimes_\cC -$};
\node[draw,rectangle] (otimes-2) at (4.25,2) {$-\otimes_\cC -$};
\node[draw,rectangle] (otimes-3) at (1,3) {$-\otimes_\cC -$};
\node[draw,rectangle] (circ) at (2,4) {$-\circ_\cC -$};
\node (top) at (2,5) {$\cC(ab\to ab\cF(v))$};
\draw (v) to[in=-90,out=90] (eta-v);
\draw (eta-v) to[in=-90,out=90] (otimes-1.-135);
\draw (j-b-1) to[in=-90,out=90] (otimes-1.-45);
\draw (j-a) to[in=-90,out=90] (otimes-3.-135);
\draw (otimes-1) to[in=-90,out=90] (otimes-3.-45);
\draw (otimes-3) to[in=-90,out=90] (circ.-135);
\draw (e) to[in=-90,out=90] (otimes-2.-135);
\draw (j-b-2) to[in=-90,out=90] (otimes-2.-45);
\draw (otimes-2) to[in=-90,out=90] (circ.-45);
\draw (circ) to[in=-90,out=90] (top);
\end{tikzpicture}
$$
which is exactly the mate of $(\id_a e_{b, \cF(v)})\circ (e_{a,\cF(v)}\id_b)$. 

Finally, the fact that $e_{c,\cF(v)}$ is invertible follows formally from naturality, the hexagon relation, and that $e_{1_\cC, \cF(v)}=\id_{\cF(v)}$ together with rigidity in the usual way.
We have
$$
\begin{tikzpicture}[baseline=30,smallstring]
\node (Fv-bottom) at (1,0) {$\cF(v)$};
\node (c-bottom) at (0,0) {$c$};
\node[draw,rectangle] (e-c) at (.5,1) {$e_{c,\cF(v)}$};
\node[draw,rectangle] (e-c*) at (-.5,2) {$e_{c_*,\cF(v)}$};
\node (Fv-top) at (-.75,3) {$\cF(v)$};
\node (c-top) at (-1.5,3) {$c$};
\draw (c-bottom) to[in=-90,out=90] (e-c.-135);
\draw (Fv-bottom) to[in=-90,out=90] (e-c.-45);
\draw (e-c.45) -- ($ (e-c.45) + (0,1) $) to[in=90,out=90] (e-c*.45);
\draw (e-c*.-135) to[in=-90,out=-90] ($ (c-top) - (0,1.3) $) -- (c-top);
\draw (e-c.135) to[in=-90,out=90] (e-c*.-45);
\draw (e-c*.129) to[in=-90,out=90] (Fv-top);
\end{tikzpicture}
\,\,=\,\,
\begin{tikzpicture}[baseline=25,smallstring]
\node (Fv-bottom) at (.75,0) {$\cF(v)$};
\node (c-bottom) at (0,0) {$c$};
\node[draw,rectangle] (e) at (0,1) {$\quad e_{c_* c,\cF(v)}\quad$};
\node (Fv-top) at (-.5,2) {$\cF(v)$};
\node (c-top) at (-1.5,2) {$c$};
\draw (c-bottom) to[in=-90,out=90] (e);
\draw (Fv-bottom) to[in=-90,out=90] (e.-23);
\draw (e.-150) to[in=-90,out=-90] ($ (c-top) - (0,1.3) $) -- (c-top);
\draw (e.148) to[in=-90,out=90] (Fv-top);
\draw (e) -- ($ (e) + (0,.5) $) to[in=90,out=90] ($ (e.23) + (0,.2) $) -- (e.23);
\end{tikzpicture}
\,\,=
\id_{c\cF(v)}
$$
and similarly for the composite the other way.
\end{proof}


Thus $\cF:\cV\to \cC^\cV$ lifts to an oplax monoidal functor $\cF^{\scriptscriptstyle Z}: \cV\to Z(\cC^\cV)$. 
With these half-braidings in hand, we state the following proposition, whose proof is omitted as it is similar to that of Proposition \ref{prop:MateOfComposition}.

\begin{prop}
\label{prop:MateOfTensorProduct}
The mate of the tensor product map $(-\otimes_\cC -)\in \cV(\cC(a\to b) \cC(c\to d) \to \cC(ac\to bd))$ under Adjunction \eqref{eq:La-Adjunction}
with $v= \cC(a\to b)\cC(c\to d)$ and $a=ac$ and $b=bd$ is given by
$$
\begin{tikzpicture}[baseline=50,smallstring]
\node (a) at (0,0) {$a$};
\node (c) at (1,0) {$c$};
\node (F) at (3,0) {$\{a\to b;c\to d\}$};
\node[draw,rectangle] (m) at (3,1) {$\mu_{\{a\to b\}, \{c\to d\}} $};
\draw[double] (F) to[in=-90,out=90] (m.270);
\node[draw,rectangle] (ev-a-b) at (1,3) {$\varepsilon_{a\to b}$};
\node[draw,rectangle] (ev-c-d) at (3,3) {$\varepsilon_{c\to d}$};
\draw[knot] (m.135) to[in=-90,out=90] (ev-a-b.-45);
\draw[knot] (c) to[in=-90,out=90] (ev-c-d.-135);
\draw (m.45) to[in=-90,out=90] (ev-c-d.-45);
\draw (a) to[in=-90,out=90] (ev-a-b.-135);
\node (b) at (1,4) {$b$};
\node (d) at (3,4) {$d$};
\draw (ev-a-b) to[in=-90,out=90] (b);
\draw (ev-c-d) to[in=-90,out=90] (d);
\end{tikzpicture}
.
$$
\end{prop}

\begin{prop}
\label{prop:Braided}
The functor $\cF^{\scriptscriptstyle Z}$ is braided, i.e., $\cF(\beta_{u,v})\circ \mu_{v,u}=\mu_{u,v}\circ e_{\cF(u),\cF(v)}$ for all $u,v\in \cV$.
\end{prop}
\begin{proof}
We prove these maps are equal by taking mates under Adjunction \eqref{eq:TraceAdjunction}.
By Remark \ref{rem:Mates}, the mate of $\cF(\beta_{u,v})\circ \mu_{v,u}$ is equal to $\beta_{u,v}\circ  (\eta_v\eta_u)\circ (-\otimes_\cC-)$.
We may compose with identity elements at no extra cost and then use the interchange relation together with naturality of the braiding in $\cV$ to obtain
$$
\mate(\cF(\beta_{u,v})\circ \mu_{v,u})
=\,\,
\begin{tikzpicture}[baseline=40,smallstring]
\node (1-L) at (0,0) {$\iV$};
\node (u) at (1,0) {$u$};
\node (v) at (2,0) {$v$};
\node (1-R) at (3,0) {$\iV$};
\node[draw, rectangle, fill=white] (j-1) at (0,1) {$j_{1_\cC}$};
\node[draw, rectangle, fill=white] (eta-v) at (1,1) {$\eta_v$};
\node[draw, rectangle, fill=white] (eta-u) at (2,1) {$\eta_u$};
\node[draw, rectangle, fill=white] (j-u) at (3,1) {$j_{\cF(u)}$};
\node[draw, rectangle, fill=white] (circ-L) at (.5,2) {$-\circ_\cC - $};
\node[draw, rectangle, fill=white] (circ-R) at (2.5,2) {$-\circ_\cC - $};
\node[draw, rectangle, fill=white] (otimes) at (1.5,3) {$-\otimes_\cC - $};
\draw (j-1) to [in= -90,out=90] (circ-L.-135);
\draw (j-u) to [in= -90,out=90] (circ-R.-45);
\draw[knot] (v) to [in= -90,out=90] (eta-v);
\draw[knot] (u) to [in= -90,out=90] (eta-u);
\node (1-vu) at (1.5,4) {$\cC(1\to \cF(v)\cF(u))$};
\draw (eta-v) to [in= -90,out=90] (circ-L.-45);
\draw (eta-u) to [in= -90,out=90] (circ-R.-135);
\draw (circ-L) to [in= -90,out=90] (otimes.-135);
\draw (circ-R) to [in= -90,out=90] (otimes.-45);
\draw (otimes) to [in= -90,out=90] (1-vu);
\end{tikzpicture}
\,\,=\,\,
\begin{tikzpicture}[baseline=40,smallstring]
\node (1-L) at (0,0) {$\iV$};
\node (u) at (1,0) {$u$}; 
\node (v) at (2,0) {$v$};
\node (1-R) at (3,0) {$\iV$};
\node[draw, rectangle, fill=white] (j-1) at (0,1) {$j_{1_\cC}$};
\node[draw, rectangle, fill=white] (eta-u) at (1,1) {$\eta_u$};
\node[draw, rectangle, fill=white] (eta-v) at (2,1) {$\eta_v$};
\node[draw, rectangle, fill=white] (j-u) at (3,1) {$j_{\cF(u)}$};
\node[draw, rectangle, fill=white] (otimes-L) at (.5,2) {$-\otimes_\cC - $};
\node[draw, rectangle, fill=white] (otimes-R) at (2.5,2) {$-\otimes_\cC - $};
\node[draw, rectangle, fill=white] (circ) at (1.5,3) {$-\circ_\cC - $};
\node (1-vu) at (1.5,4) {$\cC(1\to \cF(v)\cF(u))$};
\draw (v) to [in= -90,out=90] (eta-v);
\draw (u) to [in= -90,out=90] (eta-u);
\draw (j-1) to [in= -90,out=90] (otimes-L.-135);
\draw (eta-u) to [in= -90,out=90] (otimes-L.-45);
\draw (eta-v) to [in= -90,out=90] (otimes-R.-135);
\draw (j-u) to [in= -90,out=90] (otimes-R.-45);
\draw (otimes-L) to [in= -90,out=90] (circ.-135);
\draw (otimes-R) to [in= -90,out=90] (circ.-45);
\draw (circ) to [in= -90,out=90] (1-vu);
\end{tikzpicture}
\,\,=
(\eta_u \tilde{e})\circ (-\circ_\cC-)
$$
where $\tilde{e}$ is the mate of $e_{\cF(u),\cF(v)}$ as in \eqref{eq:MateOfHalfBraiding} under Adjunction \eqref{eq:La-Adjunction}.
Now applying Lemma \ref{lem:AdjunctionCompatibility} with $\tilde{f}=\eta_u$ and $\tilde{g}=\tilde{e}$, we see that the above map is none other than the mate of $\mu_{u,v} \circ e_{\cF(u),\cF(v)}$.
\end{proof}

Combining Sections \ref{sec:UnderlyingTensorCategory}, \ref{sec:Mates}, and \ref{sec:extracting-oplax}, starting with a rigid $\cV$-monoidal category $\cC$ such that $\Tr_\cV = \cC(1_\cC \to -)$ admits a left adjoint $\cF$, we get an ordinary rigid monoidal category $\cC^\cV$ and a braided oplax monoidal functor $\cF^{\scriptscriptstyle Z}: \cV\to Z(\cC^\cV)$ such that $\cF = \cF^{\scriptscriptstyle Z}\circ R$ admits a right adjoint.

\section{The enriched monoidal category from a braided oplax monoidal functor}
\label{sec:construction}

We now consider the other direction. 
Given a rigid monoidal category $\cT$ and a braided oplax monoidal functor $\cF^{\scriptscriptstyle Z}: \cV \to Z(\cT)$ such that $\cF:= \cF^{\scriptscriptstyle Z}\circ R: \cV\to \cT$ admits a right adjoint, we can construct a $\cV$-monoidal category, which we call $\cT \dslash \cF$. 

\begin{remark}
\label{rem:CategorifiedTrace}
Suppose $\cF^{\scriptscriptstyle Z}$ is strong monoidal, and $\cF=\cF^{\scriptscriptstyle Z}\circ R$ has adjoint $\Tr_\cV: \cT\to \cV$, where $R: Z(\cT)\to \cT$ is the forgetful functor.
In this scenario, we may construct $\cT\dslash \cF$ using the graphical calculus of \cite{1509.02937}.
One sets $\cT \dslash \cF(a\to b) = \Tr_\cV(a^*b)$, and composition and tensor product are given by
\begin{align*}
(-\circ -) &= \mu_{a^*b, b^*c} \circ \Tr_\cV(\id_{a^*} \ev_b \id_c)
\\
(-\otimes -) &= (\id_{a^*}\id_b \tau^{-1}_{c^*,d})\circ\mu_{a^*b, dc^*}\circ\tau_{a^*bd,c^*}.
\end{align*}
Here, $\mu_{a,b}: \Tr_\cV(a) \Tr_\cV(b) \to \Tr_\cV(a b)$ is the laxitor of $\Tr_\cV$ discussed in \eqref{eq:laxitor}, and $\tau_{a,b}: \Tr_\cV(a b) \to \Tr_\cV(b^{**} a)$ is the traciator discussed in Remark \ref{rem:Traciator}.
That these maps satisfy the braided interchange relation is a challenging exercise with the graphical calculus developed in \cite{1509.02937} using only the relations therein, but it becomes an easy calculation using the anchored planar algebra technology developed in \cite{1607.06041}.
We note that pivotality is not required as every traciator is paired with an inverse traciator.
\end{remark}

\subsection{The \texorpdfstring{$\cV$}{V}-monoidal category \texorpdfstring{$\cT \dslash \cF$}{T // F}}
\label{sec:EnrichingT}

We now construct $\cT \dslash \cF$ as a $\cV$-monoidal category only assuming $\cF^{\scriptscriptstyle Z}$ is braided oplax monoidal. 
The category $\cT \dslash \cF$ has the same objects as $\cT$.
To define the hom objects, we first note that similar to Adjunction \eqref{eq:La-Adjunction}, for all $a\in\cT$, the functor  $\cL_a: \cV\to \cT$ given by $\cL_a(v)=a\cF(v)$ has a right adjoint $\bar{\cR}_a:\cT\to \cV$. 
Indeed, let $\bar{\cR}_{1_\cT}: \cT \to \cV$ be the right adjoint of $\cF: \cV\to \cT$, and define $\bar{\cR}_a (b) = \bar{\cR}_{1_\cT}(a^*b)$.
Observe that for all $a\in \cT$, $v\in\cV$, and $b\in \cT$, we have
\begin{equation*}
\begin{split}
\cT( \cL_a(v) \to b)
&=
\cT( a\cF(v) \to b)
\\&\cong
\cT(\cF(v) \to a^*b)
\\&\cong
\cV(v \to \bar{\cR}_{1_\cT}(a^*b))
\\&=
\cV(v\to \bar{\cR}_a(b)).
\end{split}
\end{equation*}
We define $\cT \dslash \cF(a\to b):=\bar{\cR}_a(b)\in \cV$.
Thus $\cT \dslash \cF(a\to b)$ satisfies the adjunction
\begin{equation}
\label{eq:adjunction}
\cV(v \to \cT \dslash \cF(a\to b)) \cong \cT( a\cF(v) \to b).
\end{equation}

\begin{defn}
The identity element $j_a \in \cV(1_\cV \to \cT\dslash \cF (a\to a))$ is the mate of $\id_a \in \cT(a\to a)$.
\end{defn}

We introduce the following notation to make our diagrams easier to read:
\begin{align*}
[a \to b] &\overset{\text{def}}{=} \cF(\cT \dslash \cF(a \to b)).
\\
[a \to b;c\to d;\cdots ] &\overset{\text{def}}{=} \cF(\cT \dslash \cF(a \to b)\cT \dslash \cF(c \to d) \cdots).
\end{align*}

\begin{defn}
The evaluation morphism $\bar{\varepsilon}_{a \to b}: a [a \to b] \to b$ in $\cT$ is
the mate of the identity $\cV(\cT \dslash \cF(a \to b)\to \cT \dslash \cF(a \to b))$ under \eqref{eq:adjunction}.
\end{defn}


\subsection{Composition}

%
%
\begin{defn}
Following Proposition \ref{prop:MateOfComposition}, we define the composition map
$$
(-\circ_{\cT \dslash \cF} -):\cT \dslash \cF(a \to b) \cT \dslash \cF(b \to c) \to \cT \dslash \cF(a \to c)
$$
as the mate of the following map in $\cT(a[a\to b][b\to c]\to c)$ under the adjunction \eqref{eq:adjunction}:
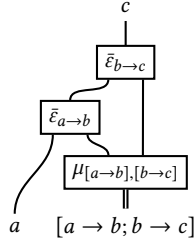
$$
\begin{tikzpicture}[baseline=50,smallstring]
\node (a) at (0,0) {$a$};
\node (F) at (2,0) {$[a\to b; b\to c]$};
\node[draw,rectangle] (m) at (2,1) {$\mu_{[a\to b], [b\to c]} $};
\draw[double] (F) to[in=-90,out=90] (m.270);
\node[draw,rectangle] (ev-a-b) at (1,2) {$\bar{\varepsilon}_{a\to b}$};
\draw (a) to[in=-90,out=90] (ev-a-b.-135);
\draw (m.135) to[in=-90,out=90] (ev-a-b.-45);
\node[draw,rectangle] (ev-b-c) at (2,3) {$\bar{\varepsilon}_{b\to c}$};
\draw (m.45) to[in=-90,out=90] (ev-b-c.-45);
\draw (ev-a-b) to[in=-90,out=90] (ev-b-c.-135);
\node (c) at (2,4) {$c$};
\draw (ev-b-c) to[in=-90,out=90] (c);
\end{tikzpicture}
.
$$
\end{defn}

\begin{lem}
\label{lem:CompositionFormula}
Composition is compatible with evaluation, in the sense that the following diagram commutes:
$$
\begin{tikzcd}[column sep=huge]
a [a\to b; b\to c] \ar[r, "\id \cF(-\circ_{\cT \dslash \cF}-)"] \ar[dd, "\id \mu"] &
a [a\to c]  \ar[dr, "\bar{\varepsilon}_{a\to c}"]
\\
&&c
\\
a [a\to b][b\to c] \ar[r, "\bar{\varepsilon}_{a\to b}\id"] &
b[b\to c] \ar[ur, "\bar{\varepsilon}_{b\to c}"]
\end{tikzcd}
$$
\end{lem}
\begin{proof}
This follows from \eqref{eq:adjunction} and Remark \ref{rem:Mates} by taking $\cA=\cT$, $\cB=\cV$, $\cL=\cL_a$ and  $\cR=\bar{\cR}_a,$ $v=\cT \dslash \cF(a \to b) \cT \dslash \cF(b \to c)$, and $f=(-\circ_{\cT \dslash \cF} -): v\to \bar{\cR}_a(c)=\cT \dslash \cF(a\to c)$.
\end{proof}

\begin{lem}
\label{lem:CompositionAssociative}
Composition is associative.
\end{lem}
\begin{proof}
By Remark \ref{rem:Mates} and Lemma \ref{lem:CompositionFormula}, the mate of 
$$
\left(\id(-\circ_{\cT \dslash \cF} -)\right) \circ \left(-\circ_{\cT \dslash \cF} -\right) : \cT \dslash \cF(a \to b) \cT \dslash \cF(b \to c)\cT \dslash \cF(c \to d)\to \cT \dslash \cF(a \to d)
$$ 
is given by
$$
\begin{tikzpicture}[baseline=30,smallstring]
\node (a) at (-.5,-1) {$a$};
\node (F) at (2,-1) {$[a\to b; b\to c;c\to d]$};
\node[draw,rectangle, fill=white] (C) at (2,0) {$\id\cF(-\circ_{\cT \dslash \cF}-)$};
\draw[triple={[line width=.3mm,black] in
      [line width=.9mm,white] in
      [line width=1.5mm,black]}] (F) to[in=-90,out=90] (C.270);
\node[draw,rectangle] (m) at (2,1) {$\mu_{[a\to b], [b\to d]} $};
\draw[double] (C) to[in=-90,out=90] (m.270);
\node[draw,rectangle] (ev-a-b) at (1,2) {$\bar{\varepsilon}_{a\to b}$};
\draw (a) to[in=-90,out=90] (ev-a-b.-135);
\draw (m.135) to[in=-90,out=90] (ev-a-b.-45);
\node[draw,rectangle] (ev-b-d) at (2,3) {$\bar{\varepsilon}_{b\to d}$};
\draw (m.45) to[in=-90,out=90] (ev-b-d.-45);
\draw (ev-a-b) to[in=-90,out=90] (ev-b-d.-135);
\node (d) at (2,4) {$d$};
\draw (ev-b-d) to[in=-90,out=90] (d);
\end{tikzpicture}
=
\begin{tikzpicture}[baseline=30,smallstring]
\node (a) at (0,-1) {$a$};
\node (F) at (3,-1) {$[a\to b; b\to c;c\to d]$};
\node[draw,rectangle] (m) at (3,0) {$\mu_{[a\to b], [b\to c;c\to d]} $};
\draw[triple={[line width=.3mm,black] in
      [line width=.9mm,white] in
      [line width=1.5mm,black]}] (F) to[in=-90,out=90] (m);
\node[draw,rectangle] (ev-a-b) at (1,1.5) {$\bar{\varepsilon}_{a\to b}$};
\node[draw,rectangle, fill=white] (C) at (3.3,1.5) {$\cF(-\circ_{\cT \dslash \cF}-)$};
\draw (a) to[in=-90,out=90] (ev-a-b.-135);
\draw (m.135) to[in=-90,out=90] (ev-a-b);
\draw[double] (m.45) to[in=-90,out=90] (C);
\node[draw,rectangle] (ev-b-d) at (3,3) {$\bar{\varepsilon}_{b\to c}$};
\draw (C) to[in=-90,out=90] (ev-b-d.-45);
\draw (ev-a-b) to[in=-90,out=90] (ev-b-d.-135);
\node (d) at (3,4) {$d$};
\draw (ev-b-d) to[in=-90,out=90] (d);
\end{tikzpicture}
$$
by naturality of $\mu$. 
Now using Lemma \ref{lem:CompositionFormula}, the right hand side above is equal to
$$
\begin{tikzpicture}[baseline=50,smallstring]
\node (a) at (0,0) {$a$};
\node (F) at (3,0) {$[a\to b;b\to c;c\to d]$};
\node[draw,rectangle] (m1) at (3,1) {$\mu_{[a\to b],[b\to c; c\to d]}$};
\draw[triple={[line width=.3mm,black] in
      [line width=.9mm,white] in
      [line width=1.5mm,black]}] (F) to[in=-90,out=90] (m1.270);
\node[draw,rectangle] (ev-a-b) at (1,2.5) {$\bar{\varepsilon}_{a\to b}$};
\draw (a) to[in=-90,out=90] (ev-a-b.-135);
\node[draw,rectangle] (m2) at (3,2.5) {$\mu_{[b\to c],[c\to d]}$};
\draw (m1.135) to[in=-90,out=90] (ev-a-b.-45);
\draw[double] (m1) to[in=-90,out=90] (m2.270);
\node[draw,rectangle] (ev-b-c) at (2,3.5) {$\bar{\varepsilon}_{b\to c}$};
\draw (m2.135) to[in=-90,out=90] (ev-b-c.-45);
\draw (ev-a-b) to[in=-90,out=90] (ev-b-c.-135);
\node[draw,rectangle] (ev-c-d) at (3,4.5) {$\bar{\varepsilon}_{c\to d}$};
\draw (m2) to[in=-90,out=90] (ev-c-d);
\draw (ev-b-c) to[in=-90,out=90] (ev-c-d.-135);
\node (d) at (3,5.5) {$d$};
\draw (ev-c-d) to[in=-90,out=90] (d);
\end{tikzpicture}
=
\begin{tikzpicture}[baseline=50,smallstring]
\node (a) at (0,0) {$a$};
\node (F) at (3,0) {$[a\to b;b\to c;c\to d]$};
\node[draw,rectangle] (m1) at (3,1) {$\mu_{[a\to b;b\to c],[c\to d]}$};
\draw[triple={[line width=.3mm,black] in
      [line width=.9mm,white] in
      [line width=1.5mm,black]}] (F) to[in=-90,out=90] (m1.270);
\node[draw,rectangle] (ev-a-b) at (1,3) {$\bar{\varepsilon}_{a\to b}$};
\draw (a) to[in=-90,out=90] (ev-a-b.-135);
\node[draw,rectangle] (m2) at (2,2) {$\mu_{[a\to b],[b\to c]}$};
\draw (m2.135) to[in=-90,out=90] (ev-a-b.-45);
\draw[double] (m1.135) to[in=-90,out=90] (m2.270);
\node[draw,rectangle] (ev-b-c) at (2,4) {$\bar{\varepsilon}_{b\to c}$};
\draw (m2.45) to[in=-90,out=90] (ev-b-c.-45);
\draw (ev-a-b) to[in=-90,out=90] (ev-b-c.-135);
\node[draw,rectangle] (ev-c-d) at (3,5) {$\bar{\varepsilon}_{c\to d}$};
\draw (m1.45) to[in=-90,out=90] (ev-c-d.-45);
\draw (ev-b-c) to[in=-90,out=90] (ev-c-d.-135);
\node (d) at (3,6) {$d$};
\draw (ev-c-d) to[in=-90,out=90] (d);
\end{tikzpicture}
$$
which is the mate of $\left((-\circ_{\cT \dslash \cF} -)\id \right) \circ \left(-\circ_{\cT \dslash \cF} -\right)$, again using the naturality of $\mu$.
\end{proof}

\subsection{Tensor product}
%

\begin{defn}
Following Proposition \ref{prop:MateOfTensorProduct}, we define the tensor product morphism $\cT \dslash \cF(a \to b)\cT \dslash \cF(c \to d) \to \cT \dslash \cF(ac \to bd)$ as the mate of the following map under the adjunction \eqref{eq:adjunction}:
$$
\begin{tikzpicture}[baseline=50,smallstring]
\node (a) at (0,0) {$a$};
\node (c) at (1,0) {$c$};
\node (F) at (3,0) {$[a\to b;c\to d]$};
\node[draw,rectangle] (m) at (3,1) {$\mu_{[a\to b], [c\to d]} $};
\draw[double] (F) to[in=-90,out=90] (m.270);
\node[draw,rectangle] (ev-a-b) at (1,3) {$\bar{\varepsilon}_{a\to b}$};
\node[draw,rectangle] (ev-c-d) at (3,3) {$\bar{\varepsilon}_{c\to d}$};
\draw[knot] (m.135) to[in=-90,out=90] (ev-a-b.-45);
\draw[knot] (c) to[in=-90,out=90] (ev-c-d.-135);
\draw (m.45) to[in=-90,out=90] (ev-c-d.-45);
\draw (a) to[in=-90,out=90] (ev-a-b.-135);
\node (b) at (1,4) {$b$};
\node (d) at (3,4) {$d$};
\draw (ev-a-b) to[in=-90,out=90] (b);
\draw (ev-c-d) to[in=-90,out=90] (d);
\end{tikzpicture}
.
$$
Since $\cT \dslash \cF(a \to b)$ is an object of $\cV$, the functor $\cF^{\scriptscriptstyle Z}: \cV \to Z(\cT)$ includes data to commute $[a \to b]$ past the object $c\in \cT$. 
\end{defn}

We omit the proofs of the following two lemmas, which are similar to Lemmas \ref{lem:CompositionFormula} and \ref{lem:CompositionAssociative}.

\begin{lem}
\label{lem:AlternateTensorProductFormula}
The mate of the tensor product map $(-\otimes_{\cT \dslash \cF} -):\cT\dslash \cF(a\to b) \cT\dslash\cF(c\to d) \to \cT\dslash\cF(ac\to bd)$ is also given by $(\id_{ac}\cF(-\otimes_{\cT \dslash \cF} -))\circ \bar{\varepsilon}_{ac\to bd}$.
\end{lem}

\begin{lem}
Tensor product is associative.
\end{lem}

\subsection{Braided interchange}
We now prove that the braided interchange law is satisfied.
As in the previous sections, this is checked by taking mates.
We use the shorthand notation of one rectangle labelled $\mu$ for composites of $\mu$'s, since the oplaxitor is associative.
After expanding the left hand side of the braided interchange law \eqref{eq:BraidedInterchange}, we get the following diagram:
$$
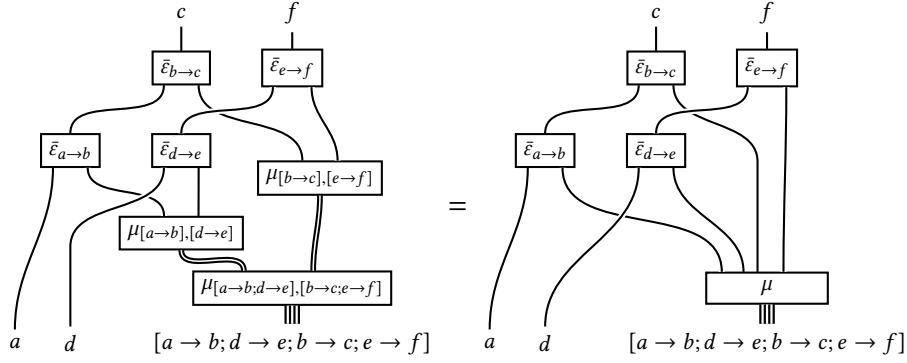
\begin{tikzpicture}[baseline=50,smallstring]
\node (a) at (0,0) {$a$};
\node (d) at (1,0) {$d$};
\node (F) at (5,0) {$[a\to b;d\to e;b\to c;e\to f]$};
\node[draw,rectangle] (m1) at (5,1) {$\mu_{[a\to b;d\to e], [b\to c;e\to f]} $};
\draw[quadruple={[line width=.3mm,white] in
      [line width=.9mm,black] in
      [line width=1.5mm,white] in
      [line width=2.1mm,black]}] (F) to[in=-90,out=90] (m1.270);
\node[draw,rectangle] (m2) at (3,2) {$\mu_{[a\to b], [d\to e]} $};
\draw[double] (m1.160) to[in=-90,out=90] (m2.270);
\node[draw,rectangle] (ev-a-b) at (1,3.5) {$\bar{\varepsilon}_{a\to b}$};
\node[draw,rectangle] (ev-d-e) at (3,3.5) {$\bar{\varepsilon}_{d\to e}$};
\draw[knot] (m2.135) to[in=-90,out=90] (ev-a-b.-45);
\draw[knot] (d) -- ($ (d) + (0,1.75) $) to[in=-90,out=90] (ev-d-e.-135);
\draw (a) to[in=-90,out=90] (ev-a-b.-135);
\draw (m2.45) to[in=-90,out=90] (ev-d-e.-45);
\node[draw,rectangle] (m3) at (5.5,3) {$\mu_{[b\to c], [e\to f]} $};
\draw[double] (m1.40) to[in=-90,out=90] (m3.270);
\node[draw,rectangle] (ev-b-c) at (3,5) {$\bar{\varepsilon}_{b\to c}$};
\node[draw,rectangle] (ev-e-f) at (5,5) {$\bar{\varepsilon}_{e\to f}$};
\draw (m3.135) to[in=-90,out=90] (ev-b-c.-45);
\draw[knot] (m3.45) to[in=-90,out=90] (ev-e-f.-45);
\draw[knot] (ev-d-e.90) to[in=-90,out=90] (ev-e-f.-135);
\draw (ev-a-b.90) to[in=-90,out=90] (ev-b-c.-135);
\node (c) at (3,6) {$c$};
\node (f) at (5,6) {$f$};
\draw (ev-b-c) to[in=-90,out=90] (c);
\draw (ev-e-f) to[in=-90,out=90] (f);
\end{tikzpicture}
=
\begin{tikzpicture}[baseline=50,smallstring]
\node (a) at (0,0) {$a$};
\node (d) at (1,0) {$d$};
\node (F) at (5,0) {$[a\to b;d\to e;b\to c;e\to f]$};
\node[draw,rectangle] (m1) at (5,1) {$\qquad\mu\qquad $};
\draw[quadruple={[line width=.3mm,white] in
      [line width=.9mm,black] in
      [line width=1.5mm,white] in
      [line width=2.1mm,black]}] (F) to[in=-90,out=90] (m1.270);
\node[draw,rectangle] (ev-a-b) at (1,3.5) {$\bar{\varepsilon}_{a\to b}$};
\node[draw,rectangle] (ev-d-e) at (3,3.5) {$\bar{\varepsilon}_{d\to e}$};
\draw[knot] (m1.160) to[in=-90,out=90] (ev-a-b.-45);
\draw[knot] (d) to[in=-90,out=90] (ev-d-e.-135);
\draw (a) to[in=-90,out=90] (ev-a-b.-135);
\draw (m1.145) to[in=-90,out=90] (ev-d-e.-45);
\node[draw,rectangle] (ev-b-c) at (3,5) {$\bar{\varepsilon}_{b\to c}$};
\node[draw,rectangle] (ev-e-f) at (5,5) {$\bar{\varepsilon}_{e\to f}$};
\draw (m1.120) -- ($ (m1.120) + (0,2) $) to[in=-90,out=90] (ev-b-c.-45);
\draw[knot] (m1.45)  to[in=-90,out=90] (ev-e-f.-45);
\draw[knot] (ev-d-e.90) to[in=-90,out=90] (ev-e-f.-135);
\draw (ev-a-b.90) to[in=-90,out=90] (ev-b-c.-135);
\node (c) at (3,6) {$c$};
\node (f) at (5,6) {$f$};
\draw (ev-b-c) to[in=-90,out=90] (c);
\draw (ev-e-f) to[in=-90,out=90] (f);
\end{tikzpicture}
.
$$
Now we perform isotopy to move $\bar{\varepsilon}_{d\to e}$ closer to $\bar{\varepsilon}_{e\to f}$, which braids the $[d\to e]$ strand over the $[b\to c]$ strand using the half-braiding $e_{[d\to e],[b\to c]}$.
Using naturality of $\mu$, and that $\cF$ is braided oplax monoidal, we have 
$$
\mu \circ \left(\id e_{[d\to e],[b\to c]}\id \right) = \cF\left(\id \beta_{\cT\dslash\cF(d\to e),\cT\dslash\cF(b\to c)} \id \right) \circ \mu,
$$
so the diagram on the right hand side above is equal to the following diagram:
$$
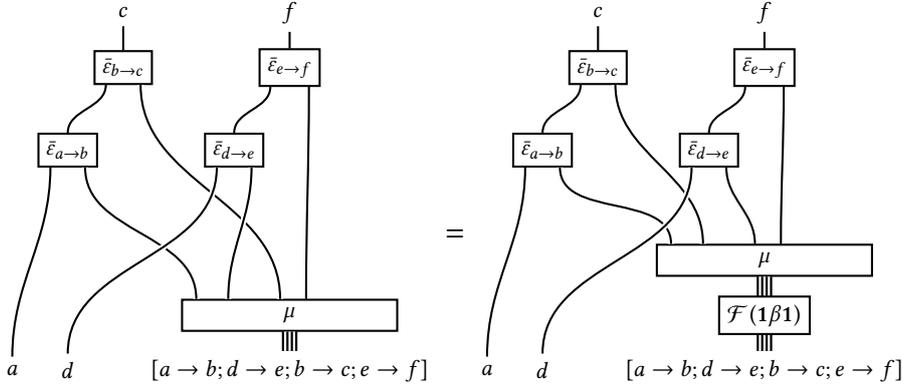
\begin{tikzpicture}[baseline=50,smallstring]
\node (a) at (0,0) {$a$};
\node (d) at (1,0) {$d$};
\node (F) at (5,0) {$[a\to b;d\to e;b\to c;e\to f]$};
\node[draw,rectangle] (m1) at (5,1) {$\qquad\qquad\mu\qquad \qquad$};
\draw[quadruple={[line width=.3mm,white] in
      [line width=.9mm,black] in
      [line width=1.5mm,white] in
      [line width=2.1mm,black]}] (F) to[in=-90,out=90] (m1.270);
\node[draw,rectangle] (ev-a-b) at (1,4) {$\bar{\varepsilon}_{a\to b}$};
\node[draw,rectangle] (ev-d-e) at (4,4) {$\bar{\varepsilon}_{d\to e}$};
\draw[knot] (m1.170) to[in=-90,out=90] (ev-a-b.-45);
\draw (a) to[in=-90,out=90] (ev-a-b.-135);
\node[draw,rectangle] (ev-b-c) at (2,5.5) {$\bar{\varepsilon}_{b\to c}$};
\node[draw,rectangle] (ev-e-f) at (5,5.5) {$\bar{\varepsilon}_{e\to f}$};
\draw (m1.45)  to[in=-90,out=90] (ev-e-f.-45);
\draw[knot] (m1.120) to[in=-90,out=90] (ev-b-c.-45);
\draw[knot] (d) to[in=-90,out=90] (ev-d-e.-135);
\draw[knot] (m1.165) to[in=-90,out=90] (ev-d-e.-45);
\draw (ev-d-e.90) to[in=-90,out=90] (ev-e-f.-135);
\draw (ev-a-b.90) to[in=-90,out=90] (ev-b-c.-135);
\node (c) at (2,6.5) {$c$};
\node (f) at (5,6.5) {$f$};
\draw (ev-b-c) to[in=-90,out=90] (c);
\draw (ev-e-f) to[in=-90,out=90] (f);
\end{tikzpicture}
=
\begin{tikzpicture}[baseline=50,smallstring]
\node (a) at (0,0) {$a$};
\node (d) at (1,0) {$d$};
\node (F) at (5,0) {$[a\to b;d\to e;b\to c;e\to f]$};
\node[draw,rectangle] (m1) at (5,2) {$\qquad\qquad\mu\qquad \qquad$};
\draw[quadruple={[line width=.3mm,white] in
      [line width=.9mm,black] in
      [line width=1.5mm,white] in
      [line width=2.1mm,black]}] (F) to[in=-90,out=90] (m1.270);
\node[draw,rectangle, fill=white] (B) at (5,1) {$\cF(\id\beta\id)$};
\node[draw,rectangle] (ev-a-b) at (1,4) {$\bar{\varepsilon}_{a\to b}$};
\node[draw,rectangle] (ev-d-e) at (4,4) {$\bar{\varepsilon}_{d\to e}$};
\draw[knot] (m1.170) to[in=-90,out=90] (ev-a-b.-45);
\draw (a) to[in=-90,out=90] (ev-a-b.-135);
\draw (m1.120) to[in=-90,out=90] (ev-d-e.-45);
\node[draw,rectangle] (ev-b-c) at (2,5.5) {$\bar{\varepsilon}_{b\to c}$};
\node[draw,rectangle] (ev-e-f) at (5,5.5) {$\bar{\varepsilon}_{e\to f}$};
\draw (m1.45)  to[in=-90,out=90] (ev-e-f.-45);
\draw[knot] (m1.165) to[in=-90,out=90] (ev-b-c.-45);
\draw[knot] (d) to[in=-90,out=90] (ev-d-e.-135);
\draw (ev-d-e.90) to[in=-90,out=90] (ev-e-f.-135);
\draw (ev-a-b.90) to[in=-90,out=90] (ev-b-c.-135);
\node (c) at (2,6.5) {$c$};
\node (f) at (5,6.5) {$f$};
\draw (ev-b-c) to[in=-90,out=90] (c);
\draw (ev-e-f) to[in=-90,out=90] (f);
\end{tikzpicture}
$$
The diagram on the right is the mate of the other diagram in the interchange law after unpacking.
\qed

\subsection{Rigidity and a left adjoint}

\begin{lem}
If $\cT$ is rigid, so is $\cT\dslash \cF$.
\end{lem}
\begin{proof}
For $c\in \cT$, we define $\ev_c\in \cV\left(1_\cV \to \cT\dslash \cF(cc^* \to 1_\cT)\right)$ to be the mate of $\ev_c\in \cT(cc^*\to 1_\cT)$ under Adjunction \eqref{eq:adjunction} with $a=cc^*$, $v=1_\cV$, and $b=1_\cT$.
The coevaluation is defined similarly. 
That these maps satisfy the zig-zag relations is straightforward.
\end{proof}

We now show that $\Tr_\cV=\cT \dslash \cF(1_{\cT \dslash \cF} \to -) : \cT \dslash \cF^\cV \to \cV$ admits a left adjoint $\cF': \cV \to \cT \dslash \cF^\cV$.
We begin with constructing an equivalence between $\cT \dslash \cF^\cV$ and $\cT$.

\begin{defn}
We define $G: \cT \dslash \cF^\cV \to \cT $ as follows.
Recall that $\cT$ and $\cT \dslash \cF^\cV$ have the same objects, so we define $G(a) = a$ for all $a\in \cT \dslash \cF^\cV$.
For $f\in \cT^\cV\dslash \cF(a\to b) = \cV(1_\cV \to \cT\dslash \cF(a\to b))$, we define $G(f)$ by
$$
G(f) 
= 
(\id_a \cF(f))\circ\bar{\varepsilon}_{a\to b}
=
\cL_a(f)\circ \bar{\varepsilon}_{a\to b}
=
\begin{tikzpicture}[baseline=1.6cm,smallstring]
\node (a) at (0,0) {$a$};
\node[draw,rectangle] (f) at (1.5,1) {$\cF(f)$};
\node[draw,rectangle] (evb-a-b) at (1,2) {$\bar{\varepsilon}_{a\to b} $};
\draw (a) to[in=-90,out=90] (evb-a-b.-135);
\draw (f) to[in=-90,out=90] (evb-a-b.-45);
\node (b) at (1,3) {$b$};
\draw (evb-a-b) to[in=-90,out=90] (b);
\end{tikzpicture}
$$
Notice that $G(f)\in \cT(a\to b)$ is the mate of $f$ under the adjunction \eqref{eq:adjunction} with $v=1_\cV$:
\begin{equation}
\label{eq:GFullyFaithful}
\cT^\cV\dslash \cF(a\to b)
=
\cV(1_\cV \to \cT\dslash \cF(a\to b))
\cong 
\cT(a\to b).
\end{equation}
\end{defn}

\begin{prop}
\label{prop:TIsomorphism}
The assignment $G$ is a monoidal equivalence of categories, where the tensorator $G(ab) \to G(a)G(b)$ is the identity.
\end{prop}
\begin{proof}
Since $\cT$ and $\cT \dslash \cF^\cV$ have the same objects $G$ will automatically be essentially surjective provided it is a functor.
By \eqref{eq:GFullyFaithful}, $G$ will automatically be fully faithful provided it is a functor.

We show $G$ is a functor.
Since we defined $j_a\in \cV(1_\cV \to  \cT\dslash \cF(a\to a))$ to be the mate of the identity $1_a\in \cT(a\to a)$, we see that $G$ preserves identities.
Suppose now $f\in \cT^\cV\dslash \cF(a\to b)$ and $g\in \cT^\cV\dslash \cF(b\to c)$.
Recall that the composite of $f$ and $g$ in $\cT^\cV\dslash \cF$ is given by $(fg)\circ(-\circ_{\cT\dslash \cF} -)$.
We calculate using naturality of $\mu$ that
$$
G\big((fg)\circ (-\circ_{\cT\dslash \cF} -)\big)
=
\begin{tikzpicture}[baseline=50,smallstring]
\node (a) at (0,0) {$a$};
\node[draw,rectangle] (m) at (2,2) {$\mu_{[a\to b], [b\to c]} $};
\node[draw,rectangle, fill=white] (F) at (2,1) {$\cF(fg)$};
\draw[double] (F) to[in=-90,out=90] (m.270);
\node[draw,rectangle] (ev-a-b) at (1,3) {$\bar{\varepsilon}_{a\to b}$};
\draw (a) to[in=-90,out=90] (ev-a-b.-135);
\draw (m.135) to[in=-90,out=90] (ev-a-b.-45);
\node[draw,rectangle] (ev-b-c) at (2,4) {$\bar{\varepsilon}_{b\to c}$};
\draw (m.45) to[in=-90,out=90] (ev-b-c.-45);
\draw (ev-a-b) to[in=-90,out=90] (ev-b-c.-135);
\node (c) at (2,5) {$c$};
\draw (ev-b-c) to[in=-90,out=90] (c);
\end{tikzpicture}
\,\,\,\,=
\begin{tikzpicture}[baseline=50,smallstring]
\node (a) at (0,0) {$a$};
\node[draw,rectangle, fill=white] (f) at (1.5,1) {$\cF(f)$};
\node[draw,rectangle] (ev-a-b) at (1,2) {$\bar{\varepsilon}_{a\to b}$};
\draw (a) to[in=-90,out=90] (ev-a-b.-135);
\draw (f) to[in=-90,out=90] (ev-a-b.-45);
\node[draw,rectangle] (ev-b-c) at (2,3) {$\bar{\varepsilon}_{b\to c}$};
\node[draw,rectangle, fill=white] (g) at (3,2) {$\cF(g)$};
\draw (g) to[in=-90,out=90] (ev-b-c.-45);
\draw (ev-a-b) to[in=-90,out=90] (ev-b-c.-135);
\node (c) at (2,4) {$c$};
\draw (ev-b-c) to[in=-90,out=90] (c);
\end{tikzpicture}
\,\,=
G(f)\circ G(g),
$$
which proves $G$ is a functor.
Thus $G$ is an equivalence of categories.

We now endow $G$ with the structure of a strong monoidal functor by taking the tensorator $G(ab) \to G(a)G(b)$ to be the identity.
Indeed, for $f\in \cT \dslash \cF^\cV(a\to b)$ and $g\in \cT \dslash \cF^\cV(c\to d)$, by naturality of $\mu$, together with Lemma \ref{lem:AlternateTensorProductFormula}, we see that
$$
G(fg)
=
\begin{tikzpicture}[baseline=50,smallstring]
\node (a) at (0,-.5) {$a$};
\node (c) at (1,-.5) {$c$};
\node[draw,rectangle] (F) at (3,0) {$\cF(fg)$};
\node[draw,rectangle] (m) at (3,1) {$\mu_{[a\to b], [c\to d]} $};
\draw[double] (F) to[in=-90,out=90] (m.270);
\node[draw,rectangle] (ev-a-b) at (1,3) {$\bar{\varepsilon}_{a\to b}$};
\node[draw,rectangle] (ev-c-d) at (2.5,3) {$\bar{\varepsilon}_{c\to d}$};
\draw[knot] (m.135) to[in=-90,out=90] (ev-a-b.-45);
\draw[knot] (c) to[in=-90,out=90] (ev-c-d.-135);
\draw (m.45) to[in=-90,out=90] (ev-c-d.-45);
\draw (a) to[in=-90,out=90] (ev-a-b.-135);
\node (b) at (1,4) {$b$};
\node (d) at (2.5,4) {$d$};
\draw (ev-a-b) to[in=-90,out=90] (b);
\draw (ev-c-d) to[in=-90,out=90] (d);
\end{tikzpicture}
=
\begin{tikzpicture}[baseline=50,smallstring]
\node (a) at (-.5,0) {$a$};
\node (c) at (1,0) {$c$};
\node[draw,rectangle] (Ff) at (3,1) {$\cF(f)$};
\node[draw,rectangle] (Fg) at (4.5,1) {$\cF(g)$};
\node[draw,rectangle] (ev-a-b) at (1.5,3) {$\bar{\varepsilon}_{a\to b}$};
\node[draw,rectangle] (ev-c-d) at (3,3) {$\bar{\varepsilon}_{c\to d}$};
\draw[knot] (Ff) to[in=-90,out=90] (ev-a-b.-45);
\draw (Fg) to[in=-90,out=90] (ev-c-d.-45);
\draw[knot] (c) to[in=-90,out=90] (ev-c-d.-135);
\draw (a) to[in=-90,out=90] (ev-a-b.-135);
\node (b) at (1.5,4) {$b$};
\node (d) at (3,4) {$d$};
\draw (ev-a-b) to[in=-90,out=90] (b);
\draw (ev-c-d) to[in=-90,out=90] (d);
\end{tikzpicture}
=
\begin{tikzpicture}[baseline=50,smallstring]
\node (a) at (0,0) {$a$};
\node (c) at (3,0) {$c$};
\node[draw,rectangle] (Ff) at (2,1) {$\cF(f)$};
\node[draw,rectangle] (Fg) at (5,1) {$\cF(g)$};
\node[draw,rectangle] (ev-a-b) at (1,2) {$\bar{\varepsilon}_{a\to b}$};
\node[draw,rectangle] (ev-c-d) at (4,2) {$\bar{\varepsilon}_{c\to d}$};
\draw[knot] (Ff) to[in=-90,out=90] (ev-a-b.-45);
\draw (Fg) to[in=-90,out=90] (ev-c-d.-45);
\draw[knot] (c) to[in=-90,out=90] (ev-c-d.-135);
\draw (a) to[in=-90,out=90] (ev-a-b.-135);
\node (b) at (1,3) {$b$};
\node (d) at (4,3) {$d$};
\draw (ev-a-b) to[in=-90,out=90] (b);
\draw (ev-c-d) to[in=-90,out=90] (d);
\end{tikzpicture}
=
G(f)G(g).
$$
Thus $G$ is a monoidal equivalence of categories.
\end{proof}

Now since $\cT$ and $\cT\dslash \cF$ have the same objects, and since $G: \cT \dslash \cF^\cV \to \cT$ is a monoidal equivalence which is the identity on objects, we may unambiguously define $G^{-1} : \cT \to \cT \dslash \cF^\cV$.
We now use $G^{-1}$ to define $\cF' : \cV \to \cT \dslash \cF^\cV$ as the composite $\cF \circ G^{-1}$.
Thus we see that starting with a braided oplax monoidal functor $\cF^{\scriptscriptstyle Z}: \cV \to Z(\cT)$ where $\cT$ is rigid, we get a rigid $\cV$-monoidal category $\cT \dslash \cF$ such that the functor $\cT \dslash \cF(1_\cT \to -)$ admits a left adjoint $\cF'$.
Indeed,
\begin{equation}
\label{eq:CompositeAdjunction}
\cT'(\cF'(v) \to c) 
\underset{G}{\cong}
\cT(\cF(v) \to c)
\underset{\eqref{eq:adjunction}}{\cong}
\cV(v \to \cT \dslash \cF(1_{\cT \dslash \cF} \to c)).
\end{equation}

%
%

\section{Equivalence}

Finally, we shall prove that the constructions described in the previous two sections are inverses.
Theorems \ref{thm:FromFunctorToEnrichedAndBack-oplax} and \ref{thm:FromEnrichedToFunctorAndBack-oplax} below combine to prove Theorem \ref{thm:Main}.

As in the previous sections, we assume $\cV$ is braided.
Before stating Theorem \ref{thm:FromFunctorToEnrichedAndBack-oplax}, we recall the definition of a morphism between two pairs $(\cT, \cF)$ and $(\cT', \cF')$ from \cite[Def.~3.2]{1607.06041}, suitably modified for oplax monoidal functors.

\begin{defn}
\label{defn:PairMorphism}
Suppose we have two rigid monoidal categories $\cT,\cT'$ equipped with braided oplax monoidal functors $(\cF^{\scriptscriptstyle Z},\mu): \cV \to Z(\cT)$ and $({\cF'}^{\scriptscriptstyle Z},\mu'): \cV\to Z(\cT')$.
A morphism $(G, \nu, \gamma):(\cT', {\cF'}^{\scriptscriptstyle Z}) \to (\cT, \cF^{\scriptscriptstyle Z}) $ consists of an oplax monoidal functor $(G,\nu): \cT'\to \cT$ and an \emph{action coherence} monoidal natural isomorphism $\gamma: \cF' \circ G \Rightarrow \cF$.
This consists of a family of natural isomorphisms $\gamma_v: G(\cF'(v)) \to \cF(v)$ such that the following diagram commutes:
\begin{equation}
\label{eq:PairMorphism}
\begin{tikzcd}
G(\cF'(uv))
	\ar[r, "\gamma_{uv}"]
	\ar[d, "G(\mu'_{u,v})"]
&
\cF(uv)
	\ar[r, "\mu_{u,v}"]
&
\cF(u)\cF(v)
\\
G(\cF'(u)\cF'(v))
	\ar[rr, "\nu_{\cF'(u),\cF'(v)}"]
&&
G(\cF'(u))G(\cF'(v))
	\ar[u, "\gamma_u\gamma_v"]
\end{tikzcd}
\end{equation}
We also require that $\gamma$ is strictly unital, i.e., $\gamma_{1_\cT}: G(\cF'(1_\cV))=1_{\cT} \to \cF(1_\cV)=1_{\cT}$ is equal to the identity.

Moreover, we require the following compatibility with the half-braidings.
For all $c\in\cT$ and $v\in \cV$,
\begin{equation}
\label{eq:CompatibilityWithHalfBraidings}
\begin{tikzcd}
G(c\cF'(v))
	\ar[rr,"\nu_{c,\cF'(v)}"]
	\ar[d, "G(e_{c,\cF'(v)})"]
&&
G(c)G(\cF'(v))
	\ar[rr, "\id_{G(c)}\gamma_{v}"]
&&
G(c)\cF(v)
	\ar[d,"e_{G(c), \cF(v)}"]
\\
G(\cF'(v)c)
	\ar[rr, "\nu_{\cF(v),c}"]
&&
G(\cF'(v))G(c)
	\ar[rr, "\gamma_v\id_{G(c)}"]
&&
\cF(v)G(c)
\end{tikzcd}
\end{equation}

A morphism $(G, \nu, \gamma): (\cT', {\cF'}^{\scriptscriptstyle Z})\to (\cT, \cF^{\scriptscriptstyle Z})$ is called an equivalence if $G$ is a strong monoidal equivalence of categories.
\end{defn}

\begin{remark}
Following \cite[\S~3.1]{1607.06041}, there is a 2-category of pairs $(\cT, \cF^{\scriptscriptstyle Z})$ where $\cT$ is a monoidal category and $\cF^{\scriptscriptstyle Z}$ is a braided oplax monoidal functor $\cV \to Z(\cT)$.
One can then discuss what it means for a 1-morphism to be an isomorphism in this 2-category.
However, we will not pursue this here, and we will be content to show a morphism is an equivalence as defined above.
\end{remark}

Recall that in Proposition \ref{prop:TIsomorphism}, we constructed a monoidal equivalence $G: \cT \dslash \cF^\cV \to \cT$ which allowed us to define $\cF' = \cF \circ G^{-1}$, which is a left adjoint of the functor $\cT\dslash \cF(1_\cT \to -)$.
By Section \ref{sec:Mates}, $\cF'$ lifts to a braided oplax monoidal functor $({\cF'}^{\scriptscriptstyle Z},\mu'):\cV \to Z(\cT \dslash \cF^\cV)$.
%

\begin{thm}
\label{thm:FromFunctorToEnrichedAndBack-oplax}
Suppose $\cT$ is rigid and $\cF^{\scriptscriptstyle Z}: \cV \to Z(\cT)$ is a braided oplax monoidal functor which admits a right adjoint.
The pairs $(\cT \dslash \cF^{\cV}, {\cF'}^{\scriptscriptstyle Z})$ and $(\cT, \cF^{\scriptscriptstyle Z})$ are equivalent in the sense of Definition \ref{defn:PairMorphism} above.
%
\end{thm}
\begin{proof}
For $v\in \cV$, we define the action coherence morphism $\gamma_v : G(\cF'(v)) \to \cF(v)$ to be the identity $\id_{\cF(v)}$.
We must now prove that Equations \eqref{eq:PairMorphism} and \eqref{eq:CompatibilityWithHalfBraidings} hold.
To do so, we note that the braided oplax monoidal lift ${\cF'}^{\scriptscriptstyle Z}$ of $\cF'$ is defined using Adjunction \eqref{eq:TraceAdjunction}, which in our case, is Adjunction \eqref{eq:CompositeAdjunction}.
Thus \eqref{eq:CompositeAdjunction} factors through $G$ and \eqref{eq:adjunction} with $a=1_{\cT \dslash \cF}$.
This means that to calculate $G(\mu'_{u,v})$ and $G(e_{c , \cF'(v)})$, it suffices to calculate the mates of $\mu'_{u,v}$ and $e_{c , \cF'(v)}$ respectively under Adjunction \eqref{eq:adjunction}.

%

To prove \eqref{eq:PairMorphism} holds, we must show that $\mu_{u,v}$ and $G(\mu'_{u,v})$, which are both maps in $\cT(\cF(uv) \to \cF(u)\cF(v))$, are equal on the nose.
Recall from Lemma \ref{lem:Oplaxitor} that the oplaxitor $\mu'_{u,v}\in \cT \dslash \cF^\cV(\cF'(uv) \to \cF'(u)\cF'(v))$ is defined as the mate of $(\eta_u\eta_v)\circ(-\otimes_{\cT \dslash \cF} -)\in \cV(uv \to \cT\dslash \cF(1_\cT \to \cF(u)\cF(v)))$ under Adjunction \eqref{eq:TraceAdjunction}.
By the above paragraph, we have that $G(\mu'_{u,v})$ is given by the mate of $(\eta_u\eta_v)\circ(-\otimes_{\cT \dslash \cF} -)\in \cV(uv \to \cT\dslash \cF(1_\cT \to \cF(u)\cF(v)))$ under Adjunction \eqref{eq:adjunction}.
Thus by Proposition \ref{prop:MateOfTensorProduct} and Remark \ref{rem:Mates},
$$
G(\mu'_{u,v})
=
\begin{tikzpicture}[baseline=40,smallstring]
\node (a) at (-3,-1) {$1_\cT$};
\node (c) at (-1,-1) {$1_\cT$};
\node (Fuv) at (3,-1) {$\cF(uv)$};
\node[draw,rectangle] (F) at (3,0) {$\cF(\eta_u\eta_v)$};
\node[draw,rectangle] (m) at (3,1) {$\mu_{[1_\cT\to \cF(u)], [1_\cT\to \cF(u)]} $};
\draw[double] (Fuv) to[in=-90,out=90] (F.270);
\draw[double] (F) to[in=-90,out=90] (m.270);
\node[draw,rectangle] (ev-1-u) at (0,3) {$\bar{\varepsilon}_{1_\cT\to \cF(u)}$};
\node[draw,rectangle] (ev-1-v) at (2,3) {$\bar{\varepsilon}_{1_\cT\to \cF(v)}$};
\draw[knot] (m.135) to[in=-90,out=90] (ev-1-u.-45);
\draw[knot, dotted] (c) to[in=-90,out=90] (ev-1-v.-135);
\draw (m.45) to[in=-90,out=90] (ev-1-v.-45);
\draw[dotted] (a) to[in=-90,out=90] (ev-1-u.-135);
\node (b) at (0,4) {$\cF(u)$};
\node (d) at (2,4) {$\cF(v)$};
\draw (ev-1-u) to[in=-90,out=90] (b);
\draw (ev-1-v) to[in=-90,out=90] (d);
\end{tikzpicture}
=
\begin{tikzpicture}[baseline=40,smallstring]
\node (a) at (-1,-1) {$1_\cT$};
\node (c) at (1,-1) {$1_\cT$};
\node (Fuv) at (4,-1) {$\cF(uv)$};
\node[draw,rectangle] (m) at (4,0) {$\mu_{u,v}$};
\node[draw,rectangle] (Ff) at (3,1) {$\cF(\eta_u)$};
\node[draw,rectangle] (Fg) at (5,1) {$\cF(\eta_v)$};
\node[draw,rectangle] (ev-a-b) at (1,3) {$\bar{\varepsilon}_{1_\cT\to \cF(u)}$};
\node[draw,rectangle] (ev-c-d) at (3,3) {$\bar{\varepsilon}_{1_\cT\to \cF(v)}$};
\draw[double] (Fuv) to[in=-90,out=90] (m.-90);
\draw (m.135) to[in=-90,out=90] (Ff.-90);
\draw (m.45) to[in=-90,out=90] (Fg.-90);
\draw[knot] (Ff) to[in=-90,out=90] (ev-a-b.-45);
\draw (Fg) to[in=-90,out=90] (ev-c-d.-45);
\draw[knot, dotted] (c) to[in=-90,out=90] (ev-c-d.-135);
\draw[dotted] (a) to[in=-90,out=90] (ev-a-b.-135);
\node (b) at (1,4) {$\cF(u)$};
\node (d) at (3,4) {$\cF(v)$};
\draw (ev-a-b) to[in=-90,out=90] (b);
\draw (ev-c-d) to[in=-90,out=90] (d);
\end{tikzpicture}
=
\mu_{u,v},
$$
since for all $v\in \cV$, $\cF(\eta_v)\circ \bar{\varepsilon}_{1_\cT\to \cF(v)} = \id_{\cF(v)}$.

Finally, proving \eqref{eq:CompatibilityWithHalfBraidings} reduces to showing that the half-braidings $e_{c,\cF(v)}$ and $G(e_{c, \cF'(v)})$, which are both maps in $\cT(a\cF(v) \to \cF(v)a)$, are equal on the nose.
Recall from \eqref{eq:MateOfHalfBraiding} that $e_{c, \cF'(v)}$ is the mate of $(\eta_v j_c)\circ(-\otimes_{\cT\dslash \cF} -)$ under Adjunction \eqref{eq:La-Adjunction}.
Again by the first paragraph in the proof, we have that $G(e_{c, \cF'(v)})$ is the mate of $(\eta_v j_c)\circ(-\otimes_{\cT\dslash \cF} -)$ under Adjunction \eqref{eq:adjunction}.
By Proposition \ref{prop:MateOfTensorProduct} and Remark \ref{rem:Mates},
$$
G(e_{c, \cF'(v)})
=
\begin{tikzpicture}[baseline=40,smallstring]
\node (a) at (-3,-1) {$1_\cT$};
\node (c) at (-1,-1) {$c$};
\node (Fv) at (3,-1) {$\cF(v)$};
\node[draw,rectangle] (F) at (3,0) {$\cF(\eta_vj_c)$};
\node[draw,rectangle] (m) at (3,1) {$\mu_{[1_\cT\to \cF(v)], [c\to c]} $};
\draw (Fv) to[in=-90,out=90] (F.270);
\draw[double] (F) to[in=-90,out=90] (m.270);
\node[draw,rectangle] (ev-1-v) at (0,3) {$\bar{\varepsilon}_{1_\cT\to \cF(v)}$};
\node[draw,rectangle] (ev-c-c) at (2,3) {$\bar{\varepsilon}_{c\to c}$};
\draw[knot] (m.135) to[in=-90,out=90] (ev-1-v.-45);
\draw[knot] (c) to[in=-90,out=90] (ev-c-c.-135);
\draw (m.45) to[in=-90,out=90] (ev-c-c.-45);
\draw[dotted] (a) to[in=-90,out=90] (ev-1-v.-135);
\node (b) at (0,4) {$\cF(v)$};
\node (d) at (2,4) {$c$};
\draw (ev-1-v) to[in=-90,out=90] (b);
\draw (ev-c-c) to[in=-90,out=90] (d);
\end{tikzpicture}
=
\begin{tikzpicture}[baseline=50,smallstring]
\node (a) at (-.5,0) {$1_\cT$};
\node (c) at (1,0) {$c$};
\node (Fv) at (3,0) {$\cF(v)$};
\node[draw,rectangle] (Ff) at (3,1) {$\cF(\eta_v)$};
\node[draw,rectangle] (Fg) at (4.5,1) {$\cF(j_c)$};
\node[draw,rectangle] (ev-a-b) at (1.25,3) {$\bar{\varepsilon}_{1_\cT\to \cF(v)}$};
\node[draw,rectangle] (ev-c-d) at (3,3) {$\bar{\varepsilon}_{c\to c}$};
\draw (Fv) to[in=-90,out=90] (Ff.-90);
\draw[knot] (Ff) to[in=-90,out=90] (ev-a-b.-45);
\draw (Fg) to[in=-90,out=90] (ev-c-d.-45);
\draw[knot] (c) to[in=-90,out=90] (ev-c-d.-135);
\draw[dotted] (a) to[in=-90,out=90] (ev-a-b.-135);
\node (b) at (1.25,4) {$\cF(v)$};
\node (d) at (3,4) {$c$};
\draw (ev-a-b) to[in=-90,out=90] (b);
\draw (ev-c-d) to[in=-90,out=90] (d);
\end{tikzpicture}
=
\begin{tikzpicture}[baseline=20,smallstring]
\node (bottom-L) at (0,0) {$c$};
\node (bottom-R) at (2,0) {$\cF(v)$};
\node (top-L) at (0,2) {$\cF(v)$};
\node (top-R) at (2,2) {$c$};
\draw[knot] (bottom-R.90)  to[in=-90,out=90] (top-L.-90);
\draw[knot] (bottom-L.90)  to[in=-90,out=90] (top-R.-90);
\end{tikzpicture}
=
e_{c,\cF(v)},
$$
since
$\cF(\eta_v)\circ \bar{\varepsilon}_{1_\cT\to \cF(v)} = \id_{\cF(v)}$
and
$(\id_c \cF(j_c))\circ \bar{\varepsilon}_{c\to c}=\cL_c(j_c)\circ \bar{\varepsilon}_{c\to c}=\id_c$.
We are finished.
\end{proof}

As before, we have  $\cV$ is a fixed a braided monoidal category.
Suppose $\cC$ is a rigid $\cV$-monoidal category such that the categorified `trace' $\Tr_\cV = \cC(1_\cC \to -): \cC^\cV \to \cV$ admits a left adjoint $\cF : \cV \to \cC^\cV$.
Following Section \ref{sec:extracting-oplax}, $\cF$ lifts to a braided oplax monoidal functor $\cF^{\scriptscriptstyle Z} : \cV \to Z(\cC^\cV)$.

\begin{thm}
\label{thm:FromEnrichedToFunctorAndBack-oplax}
%
There is a $\cV$-monoidal equivalence between $\cC$ and $\cC^\cV \dslash \cF$.
\end{thm}

Before we begin the proof, we give some helpful notation and lemmas.
Thinking of $\cC^\cV$ as $\cT$, as before, we define
\begin{align*}
[a\to b] &\overset{\text{def}}{=} \cF(\cC^\cV\dslash \cF(a\to b))
\\
[a \to b;c\to d;\cdots ] &\overset{\text{def}}{=} \cF(\cC^\cV\dslash \cF(a \to b)\cC^\cV\dslash \cF(c \to d) \cdots)
\\
\{a \to b\} &\overset{\text{def}}{=} \cF(\cC(a \to b))
\\
\{a \to b;c\to d;\cdots \} &\overset{\text{def}}{=} \cF(\cC(a \to b)\cC(c \to d) \cdots).
\end{align*}

We have two adjunctions: \eqref{eq:La-Adjunction} and \eqref{eq:adjunction} with $\cT=\cC^\cV$.
\begin{align}
\label{eq:adjunctionCV}
\cC^\cV(a \cF(v) \to b)
&\cong
\cV(v \to \cC(a \to b) ) 
\\
\cC^\cV(a \cF(v) \to b)
&\cong 
\cV(v \to \cC^\cV \dslash \cF(a \to b) ) 
\label{eq:adjunctionCVbar}
\end{align}
The first adjunction uses the right adjoint $\cR_a: \cC^\cV \to \cV$ of $\cL_a$ given by $b\mapsto \cC(a\to b)$ from \S \ref{sec:ExtraStructureFromRigidity}.
The second adjunction uses the right adjoint $\bar{\cR}_a: \cC^\cV \to \cV$ of $\cL_a$ given by $\bar{\cR}_a(v) = \cC^\cV\dslash \cF(a\to b)$ from \S \ref{sec:EnrichingT}. 

As in the previous sections, $\varepsilon_{a\to b}: a\{a\to b\}\to b$ is the mate of $\id\in \cV(\cC(a\to b) \to \cC(a\to b))$ under \eqref{eq:adjunctionCV}, and $\bar{\varepsilon}_{a\to b}:a[a\to b]\to b$ is the mate of $\id\in \cV(\cC^\cV\dslash \cF(a\to b) \to \cC^\cV\dslash \cF(a\to b))$ under \eqref{eq:adjunctionCVbar}.



We now use Adjunctions \eqref{eq:adjunctionCV} and \eqref{eq:adjunctionCVbar} to construct $\cV$-monoidal functors
$\cG: \cC^\cV\dslash \cF\to \cC$ and $\cH : \cC\to \cC^\cV\dslash\cF$ which witness an equivalence of $\cV$-monoidal categories.

\begin{defn}
We define $\cG: \cC\to \cC^\cV\dslash \cF$ by $\cG(c) = c$ on objects and we define $\cG_{a\to b}\in \cV(\cC(a\to b)\to\cC^\cV \dslash \cF(a \to b))$ as the mate of $\varepsilon_{a\to b}$ under Adjunction \eqref{eq:adjunctionCVbar} with $v=\cC(a\to b)$.
We define the tensorator $\alpha^\cG$  by 
$$
\alpha^\cG_{a,b} = j^{\cC^\cV\dslash\cF}_{ab}: 1_\cV \to \cC^\cV\dslash\cF(\cG(ab) \to \cG(a)\cG(b))=\cC^\cV\dslash\cF(ab \to ab).
$$

We define $\cH: \cC^\cV\dslash \cF\to \cC$ by $\cH(c) = c$ on objects and we define $\cH_{a\to b}\in \cV(\cC^\cV \dslash \cF(a \to b)\to\cC(a \to b))$ as the mate of $\bar{\varepsilon}_{a\to b}$ under Adjunction \eqref{eq:adjunctionCV} with $v=\cC^\cV\dslash \cF(a\to b)$ .
We define the tensorator 
$$
\alpha^\cH_{a,b}= j^{\cC}_{ab}: 1_\cV \to \cC(\cH(ab) \to \cH(a)\cH(b))=\cC(ab \to ab).
$$

In summary, the definition of $\cG_{a\to b}$  is given by taking mates as follows:
\begin{equation}
\label{eq:MatesAndEvs}
\begin{tikzcd}
\cV\left(\cC(a\to b) \to \cC(a \to b) \right) 
\ar[r, leftrightarrow, "\cong"]
&
\cC^\cV\left(a \{a\to b\} \to b\right)
\ar[r, leftrightarrow, "\cong"]
&
\cV\left(\cC(a\to b) \to \cC^\cV \dslash \cF(a \to b) \right) 
\\[-22.5pt]
\id_{\cC(a\to b)}
\ar[r, leftrightarrow]
&
\varepsilon_{a\to b}
\ar[r, leftrightarrow]
&
\cG_{a\to b}
\end{tikzcd}
\end{equation}
and for $\cH_{a\to b}$ as:
\begin{equation*}
\begin{tikzcd}
\cV\left(\cC^\cV \dslash \cF(a\to b) \to \cC^\cV \dslash \cF(a \to b) \right) 
\ar[r, leftrightarrow, "\cong"]
&
\cC^\cV\left(a [a\to b] \to b\right)
\ar[r, leftrightarrow, "\cong"]
&
\cV\left(\cC^\cV \dslash \cF(a\to b) \to \cC(a \to b) \right) 
\\[-22.5pt]
\id_{\cC^\cV\dslash \cF (a\to b)}
\ar[r, leftrightarrow]
&
\bar{\varepsilon}_{a\to b}
\ar[r, leftrightarrow]
&
\cH_{a\to b}
\end{tikzcd}
\end{equation*}
\end{defn}

Applying Remark \ref{rem:Mates} with $\cA= \cC^\cV$, $\cB=\cV$, $\cL=\cL_a$ and $\cR=\bar{\cR}_a$, $v=\cC (a\to b)$, and $f=\cG_{a\to b}$ gives us the formula:
\begin{equation}
\label{eq:EvBarGFormula}
\begin{tikzpicture}[baseline=.9cm,smallstring]
\node (a) at (0,0) {$a$};
\node (F) at (1,0) {$\{a\to b\}$};
\node[draw,rectangle] (evb-a-b) at (1,1) {$\varepsilon_{a\to b} $};
\draw (a) to[in=-90,out=90] (evb-a-b.-135);
\draw (F) to[in=-90,out=90] (evb-a-b);
\node (b) at (1,2) {$b$};
\draw (evb-a-b) to[in=-90,out=90] (b);
\end{tikzpicture}
=
\varepsilon_{a\to b} = (\id_a \cF(\cG_{a\to b})) \circ\bar{\varepsilon}_{a\to b}
=
\begin{tikzpicture}[baseline=1.6cm,smallstring]
\node (a) at (0,0) {$a$};
\node (F) at (2,0) {$\{a\to b\}$};
\node[draw,rectangle] (H) at (2,1) {$\cF(\cG_{a\to b})$};
\draw (F) to[in=-90,out=90] (H);
\node[draw,rectangle] (evb-a-b) at (1,2.5) {$\bar{\varepsilon}_{a\to b} $};
\draw (a) to[in=-90,out=90] (evb-a-b.-135);
\draw (H) to[in=-90,out=90] (evb-a-b.-45);
\node (b) at (1,3.5) {$b$};
\draw (evb-a-b) to[in=-90,out=90] (b);
\end{tikzpicture}.
\end{equation}
Similarly, applying Remark \ref{rem:Mates} with $\cA= \cC^\cV$, $\cB=\cV$, $\cL=\cL_a$ and $\cR=\cR_a$, $v=\cC^\cV\dslash\cF (a\to b)$, and $f=\cH_{a\to b}$ gives us the formula:
\begin{equation}
\label{eq:EvBarHFormula}
\begin{tikzpicture}[baseline=.9cm,smallstring]
\node (a) at (0,0) {$a$};
\node (F) at (1,0) {$[a\to b]$};
\node[draw,rectangle] (evb-a-b) at (1,1) {$\bar{\varepsilon}_{a\to b} $};
\draw (a) to[in=-90,out=90] (evb-a-b.-135);
\draw (F) to[in=-90,out=90] (evb-a-b);
\node (b) at (1,2) {$b$};
\draw (evb-a-b) to[in=-90,out=90] (b);
\end{tikzpicture}
=
\bar{\varepsilon}_{a\to b} = (\id_a \cF(\cH_{a\to b})) \circ\varepsilon_{a\to b}
=
\begin{tikzpicture}[baseline=1.6cm,smallstring]
\node (a) at (0,0) {$a$};
\node (F) at (2,0) {$[a\to b]$};
\node[draw,rectangle] (H) at (2,1) {$\cF(\cH_{a\to b})$};
\draw (F) to[in=-90,out=90] (H);
\node[draw,rectangle] (evb-a-b) at (1,2.5) {$\varepsilon_{a\to b} $};
\draw (a) to[in=-90,out=90] (evb-a-b.-135);
\draw (H) to[in=-90,out=90] (evb-a-b.-45);
\node (b) at (1,3.5) {$b$};
\draw (evb-a-b) to[in=-90,out=90] (b);
\end{tikzpicture}.
\end{equation}

\begin{proof}[Proof of Theorem \ref{thm:FromEnrichedToFunctorAndBack-oplax}]
First, we need to check that the diagrams
\eqref{eq:CompositionAxiomForVFunctors} and \eqref{eq:TriangleIdentityAxiomForVFunctors} for $\cG$ commute.
We only show \eqref{eq:CompositionAxiomForVFunctors} as the other calculation is easier.
By Remark \ref{rem:Mates}, the mate of $\left(\cG_{a\to b}\cG_{b\to c} \right) \circ \left(-\circ_{\cC^\cV \dslash \cF} -\right)$ is given by
$$
\begin{tikzpicture}[baseline=50,smallstring]
\node (a) at (0,-1) {$a$};
\node (F) at (2,-1) {$\{a\to b; b\to c\}$};
\node[draw,rectangle] (m) at (2,1) {$\mu_{[a\to b], [b\to c]} $};
\draw[double] (F) to[in=-90,out=90] (m.270);
\node[draw,rectangle, fill=white] (G) at (2,0) {$\cF(\cG_{a\to b}\cG_{b\to c})$};
\node[draw,rectangle] (ev-a-b) at (1,2) {$\bar{\varepsilon}_{a\to b}$};
\draw (a) to[in=-90,out=90] (ev-a-b.-135);
\draw (m.135) to[in=-90,out=90] (ev-a-b.-45);
\node[draw,rectangle] (ev-b-c) at (2,3) {$\bar{\varepsilon}_{b\to c}$};
\draw (m.45) to[in=-90,out=90] (ev-b-c.-45);
\draw (ev-a-b) to[in=-90,out=90] (ev-b-c.-135);
\node (c) at (2,4) {$c$};
\draw (ev-b-c) to[in=-90,out=90] (c);
\end{tikzpicture}
=
\begin{tikzpicture}[baseline=50,smallstring]
\node (a) at (0,-1) {$a$};
\node (F) at (3,-1) {$\{a\to b; b\to c\}$};
\node[draw,rectangle] (m) at (3,0) {$\mu_{\{a\to b\}, \{b\to c\}} $};
\draw[double] (F) to[in=-90,out=90] (m.270);
\node[draw,rectangle, fill=white] (G-a-b) at (2,1) {$\cF(\cG_{a\to b})$};
\node[draw,rectangle] (ev-a-b) at (1,2) {$\bar{\varepsilon}_{a\to b}$};
\draw (a) to[in=-90,out=90] (ev-a-b.-135);
\draw (m.135) to[in=-90,out=90] (G-a-b);
\draw (G-a-b) to[in=-90,out=90] (ev-a-b.-45);
\node[draw,rectangle] (ev-b-c) at (3,3.5) {$\bar{\varepsilon}_{b\to c}$};
\draw (m.45) to[in=-90,out=90] (ev-b-c.-45);
\node[draw,rectangle, fill=white] (G-b-c) at (3,2) {$\cF(\cG_{b\to c})$};
\draw (ev-a-b) to[in=-90,out=90] (ev-b-c.-135);
\node (c) at (3,4.5) {$c$};
\draw (ev-b-c) to[in=-90,out=90] (c);
\end{tikzpicture}
=
\begin{tikzpicture}[baseline=50,smallstring]
\node (a) at (0,0) {$a$};
\node (F) at (2,0) {$\{a\to b; b\to c\}$};
\node[draw,rectangle] (m) at (2,1) {$\mu_{\{a\to b\}, \{b\to c\}} $};
\draw[double] (F) to[in=-90,out=90] (m.270);
\node[draw,rectangle] (ev-a-b) at (1,2) {${\varepsilon}_{a\to b}$};
\draw (a) to[in=-90,out=90] (ev-a-b.-135);
\draw (m.135) to[in=-90,out=90] (ev-a-b.-45);
\node[draw,rectangle] (ev-b-c) at (2,3) {$\varepsilon_{b\to c}$};
\draw (m.45) to[in=-90,out=90] (ev-b-c.-45);
\draw (ev-a-b) to[in=-90,out=90] (ev-b-c.-135);
\node (c) at (2,4) {$c$};
\draw (ev-b-c) to[in=-90,out=90] (c);
\end{tikzpicture}
=
\begin{tikzpicture}[baseline=1.6cm,smallstring]
\node (a) at (0,0) {$a$};
\node (F) at (2,0) {$\{a\to b; b\to c\}$};
\node[draw,rectangle] (H) at (2,1) {$\cF(-\circ_\cC-)$};
\draw[double] (F) to[in=-90,out=90] (H);
\node[draw,rectangle] (evb-a-b) at (1,2.5) {${\varepsilon}_{a\to b} $};
\draw (a) to[in=-90,out=90] (evb-a-b.-135);
\draw (H) to[in=-90,out=90] (evb-a-b.-45);
\node (c) at (1,3.5) {$c$};
\draw (evb-a-b) to[in=-90,out=90] (c);
\end{tikzpicture}
$$
which is the mate of $(-\circ_{\cC}-) \circ \cG_{a\to c}$.
The first equality follows by the naturality of $\mu$, the second by \eqref{eq:EvBarGFormula}, and the last from Proposition \ref{prop:MateOfComposition}.

That the tensorator $\alpha^\cG$ satisfies the naturality and associativity axioms is clear, since all the maps are identity elements.
Proving that $\cH$ is a $\cV$-monoidal functor is similar, and the proof is omitted.

Finally, it remains to show that $\cG$ and $\cH$ witness an equivalence.
We see that $\cG_{a\to b}\circ\cH_{a\to b}=\id_{\cC(a\to b)} \in \cV( \cC(a\to b)\to \cC(a\to b))$ by taking mates.
Indeed, $\mate(\cG_{a\to b}\circ\cH_{a\to b}) \in \cC^\cV(a\{a\to b\} \to b)$ is exactly the right hand side of \eqref{eq:EvBarGFormula}!
Thus it is equal to ${\varepsilon}_{a\to b}$, which is the mate of $\id_{\cC(a\to b)}$ by \eqref{eq:MatesAndEvs}.
That $\cH_{a\to b}\circ \cG_{a\to b}= \id_{\cC^\cV\dslash\cF(a\to b)}$ is similar using \eqref{eq:EvBarHFormula}.
\end{proof}

\appendix

\section{Technical details.}

\subsection{Transporting enrichment}
\label{sec:TransportingEnrichment}

Suppose $\cC$ is a $\cV$-monoidal category, and $(\cF,\mu): \cV\to \cW$ is a braided lax monoidal functor. (This means
that $(\cF,\mu)$ is lax monoidal, and also satisfies the compatibility $ \mu_{u,v}\circ \cF(\beta^\cV_{u,v}) = \beta^\cW_{\cF(u),\cF(v)}\circ \mu_{\cF(u),\cF
(v)}$ for all $u,v\in\cV$.)

We may form a $\cW$-monoidal category, denoted $\cF_*\cC$, whose objects are
the same as those of $\cC$, whose hom
objects are given by $\cF_*\cC(a \to b) = \cF(\cC(a\to b))$, and whose
composition and tensor product are given by $\mu \circ \cF
(-\circ_\cC-)$ and $\mu \circ \cF(-\otimes_\cC-)$ respectively.
Since $(\cF,\mu)$ is lax monoidal, we easily see that composition and tensor product are associative.
The interchange relation uses the compatibility with the braiding.
Indeed, we see
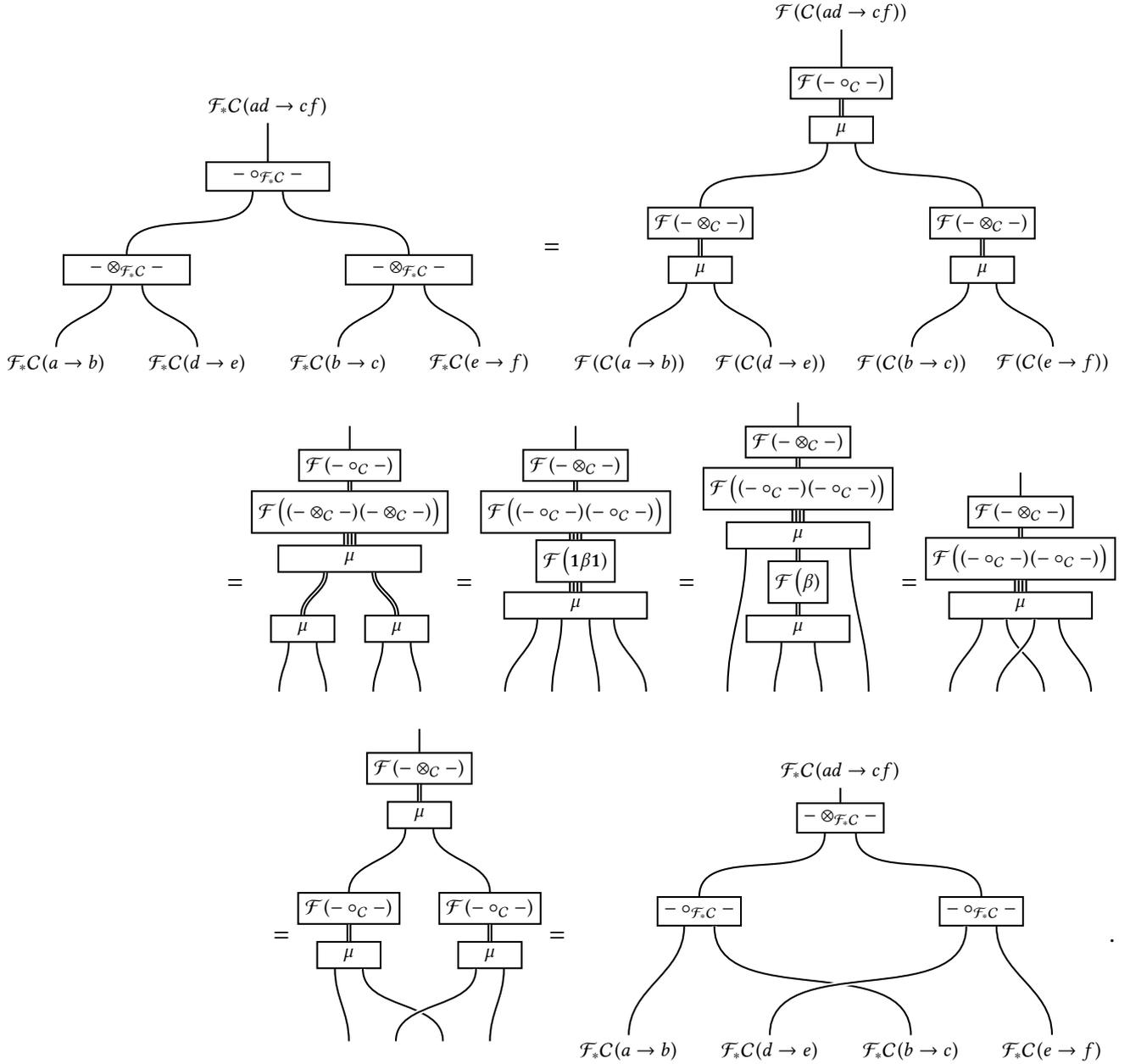
\begin{align*}
\begin{tikzpicture}[baseline=50,smallstring]
\node (a-b) at (0,0) {$\cF_*\cC(a \to b)$};
\node (d-e) at (3,0) {$\cF_*\cC(d \to e)$};
\node (b-c) at (6,0) {$\cF_*\cC(b \to c)$};
\node (e-f) at (9,0) {$\cF_*\cC(e \to f)$};
\node[draw,rectangle] (t1) at (1.5,2) {$\quad -\otimes_{\cF_*\cC}-\quad $};
\draw (a-b) to[in=-90,out=90] (t1.-135);
\draw (d-e) to[in=-90,out=90] (t1.-45);
\node[draw,rectangle] (t2) at (7.5,2) {$\quad -\otimes_{\cF_*\cC}-\quad $};
\draw (b-c) to[in=-90,out=90] (t2.-135);
\draw (e-f) to[in=-90,out=90] (t2.-45);
\node[draw,rectangle] (c) at (4.5,4) {$\quad - \circ _{\cF_*\cC} - \quad $};
\draw (t1) to[in=-90,out=90]  (c.-135);
\draw (t2) to[in=-90,out=90]  (c.-45);
\node (r) at (4.5,5.5) {$\cF_*\cC(ad \to cf)$};
\draw (c) -- (r);
\end{tikzpicture}
=
\begin{tikzpicture}[baseline=50,smallstring]
\node (a-b) at (0,0) {$\cF(\cC(a \to b))$};
\node (d-e) at (3,0) {$\cF(\cC(d \to e))$};
\node (b-c) at (6,0) {$\cF(\cC(b \to c))$};
\node (e-f) at (9,0) {$\cF(\cC(e \to f))$};
\node[draw,rectangle] (mu-L) at (1.5,2) {$\quad \mu\quad $};
\node[draw,rectangle] (otimes-L) at (1.5,3) {$\cF(-\otimes_{\cC}-)$};
\draw[double] (mu-L) to[in=-90,out=90]  (otimes-L);
\draw (a-b) to[in=-90,out=90] (mu-L.-135);
\draw (d-e) to[in=-90,out=90] (mu-L.-45);
\node[draw,rectangle] (mu-R) at (7.5,2) {$\quad\mu\quad$};
\node[draw,rectangle] (otimes-R) at (7.5,3) {$\cF( -\otimes_{\cC}-)$};
\draw[double] (mu-R) to[in=-90,out=90]  (otimes-R);
\draw (b-c) to[in=-90,out=90] (mu-R.-135);
\draw (e-f) to[in=-90,out=90] (mu-R.-45);
\node[draw,rectangle] (mu-top) at (4.5,5) {$\quad\mu\quad$};
\node[draw,rectangle] (circ) at (4.5,6) {$\cF( - \circ _{\cC} - )$};
\draw[double] (mu-top) to[in=-90,out=90]  (circ);
\draw (otimes-L) to[in=-90,out=90]  (mu-top.-135);
\draw (otimes-R) to[in=-90,out=90]  (mu-top.-45);
\node (top) at (4.5,7.5) {$\cF(\cC(ad \to cf))$};
\draw (circ) -- (top);
\end{tikzpicture}
\displaybreak[1]\\
=
\begin{tikzpicture}[baseline=50,smallstring]
\node (a-b) at (3,0) {
};
\node (d-e) at (4,0) {
};
\node (b-c) at (5,0) {
};
\node (e-f) at (6,0) {
};
\node[draw,rectangle] (mu-L) at (3.5,1.5) {$\quad\mu\quad $};
\node[draw,rectangle] (mu-R) at (5.5,1.5) {$\quad\mu\quad$};
\node[draw,rectangle] (mu-top) at (4.5,3) {$\quad\qquad\mu\qquad\quad$};
\node[draw,rectangle] (otimes) at (4.5,4) {$\cF\big((-\otimes_\cC-)(-\otimes_\cC-)\big)$};
\draw[quadruple={[line width=.3mm,white] in
      [line width=.9mm,black] in
      [line width=1.5mm,white] in
      [line width=2.1mm,black]}] (mu-top) to[in=-90,out=90] (otimes.270);
\node[draw,rectangle] (circ) at (4.5,5) {$\cF(-\circ_\cC-)$};
\draw[double] (otimes) to[in=-90,out=90] (circ.270);
\node (top) at (4.5,6) {
};
\draw (a-b) to[in=-90,out=90] (mu-L.-135);
\draw (d-e) to[in=-90,out=90] (mu-L.-45);
\draw (b-c) to[in=-90,out=90] (mu-R.-135);
\draw (e-f) to[in=-90,out=90] (mu-R.-45);
\draw[double] (mu-L) to[in=-90,out=90]  (mu-top.-150);
\draw[double] (mu-R) to[in=-90,out=90]  (mu-top.-30);
\draw (circ) -- (top);
\end{tikzpicture}
=
\begin{tikzpicture}[baseline=50,smallstring]
\node (a-b) at (3,0) {
};
\node (d-e) at (4,0) {
};
\node (b-c) at (5,0) {
};
\node (e-f) at (6,0) {
};
\node[draw,rectangle] (mu) at (4.5,2) {$\quad\qquad\mu\qquad\quad$};
\node[draw,rectangle] (beta) at (4.5,2.95) {$\cF\big(\id\beta\id)$};
\node[draw,rectangle] (circ) at (4.5,4) {$\cF\big((-\circ_\cC-)(-\circ_\cC-)\big)$};
\draw[quadruple={[line width=.3mm,white] in
      [line width=.9mm,black] in
      [line width=1.5mm,white] in
      [line width=2.1mm,black]}] (mu) to[in=-90,out=90] (beta.270);
\draw[quadruple={[line width=.3mm,white] in
      [line width=.9mm,black] in
      [line width=1.5mm,white] in
      [line width=2.1mm,black]}] (beta) to[in=-90,out=90] (circ.270);
\node[draw,rectangle] (otimes) at (4.5,5) {$\cF(-\otimes_\cC-)$};
\draw[double] (circ) to[in=-90,out=90] (otimes.270);
\node (top) at (4.5,6) {
};
\draw (a-b) to[in=-90,out=90] (mu.-160);
\draw (d-e) to[in=-90,out=90] (mu.-135);
\draw (b-c) to[in=-90,out=90] (mu.-45);
\draw (e-f) to[in=-90,out=90] (mu.-20);
\draw (otimes) -- (top);
\end{tikzpicture}
=
\begin{tikzpicture}[baseline=50,smallstring]
\node (a-b) at (3,0) {
};
\node (d-e) at (4,0) {
};
\node (b-c) at (5,0) {
};
\node (e-f) at (6,0) {
};
\node[draw,rectangle] (mu-lower) at (4.5,1.5) {$\qquad\mu\qquad$};
\node[draw,rectangle] (beta) at (4.5,2.5) {$\cF\big(\beta)$};
\node[draw,rectangle] (mu-upper) at (4.5,3.5) {$\quad\qquad\mu\qquad\quad$};
\node[draw,rectangle] (circ) at (4.5,4.5) {$\cF\big((-\circ_\cC-)(-\circ_\cC-)\big)$};
\draw[quadruple={[line width=.3mm,white] in
      [line width=.9mm,black] in
      [line width=1.5mm,white] in
      [line width=2.1mm,black]}] (mu-upper) to[in=-90,out=90] (circ.270);
\node[draw,rectangle] (otimes) at (4.5,5.5) {$\cF(-\otimes_\cC-)$};
\draw[double] (mu-lower) to[in=-90,out=90] (beta.270);
\draw[double] (beta) to[in=-90,out=90] (mu-upper.270);
\draw[double] (circ) to[in=-90,out=90] (otimes.270);
\node (top) at (4.5,6.5) {
};
\draw (a-b) to[in=-90,out=90] (mu-upper.-165);
\draw (d-e) to[in=-90,out=90] (mu-lower.-140);
\draw (b-c) to[in=-90,out=90] (mu-lower.-40);
\draw (e-f) to[in=-90,out=90] (mu-upper.-15);
\draw (otimes) -- (top);
\end{tikzpicture}
=
\begin{tikzpicture}[baseline=50,smallstring]
\node (a-b) at (3,0) {
};
\node (d-e) at (4,0) {
};
\node (b-c) at (5,0) {
};
\node (e-f) at (6,0) {
};
\node[draw,rectangle] (mu) at (4.5,2) {$\quad\qquad\mu\qquad\quad$};
\node[draw,rectangle] (circ) at (4.5,3) {$\cF\big((-\circ_\cC-)(-\circ_\cC-)\big)$};
\draw[quadruple={[line width=.3mm,white] in
      [line width=.9mm,black] in
      [line width=1.5mm,white] in
      [line width=2.1mm,black]}] (mu) to[in=-90,out=90] (circ.270);
\node[draw,rectangle] (otimes) at (4.5,4) {$\cF(-\otimes_\cC-)$};
\draw[double] (circ) to[in=-90,out=90] (otimes.270);
\node (top) at (4.5,5) {
};
\draw (a-b) to[in=-90,out=90] (mu.-160);
\draw[knot] (b-c) to[in=-90,out=90] (mu.-135);
\draw[knot] (d-e) to[in=-90,out=90] (mu.-45);
\draw (e-f) to[in=-90,out=90] (mu.-20);
\draw (otimes) -- (top);
\end{tikzpicture}
\displaybreak[1]\\
=
\begin{tikzpicture}[baseline=70,smallstring]
\node (a-b) at (3,1) {
};
\node (d-e) at (4,1) {
};
\node (b-c) at (5,1) {
};
\node (e-f) at (6,1) {
};
\node[draw,rectangle] (mu-L) at (3,3) {$\quad \mu\quad $};
\node[draw,rectangle] (circ-L) at (3,4) {$\cF(-\circ_{\cC}-)$};
\draw[double] (mu-L) to[in=-90,out=90]  (circ-L);
\node[draw,rectangle] (mu-R) at (6,3) {$\quad\mu\quad$};
\node[draw,rectangle] (circ-R) at (6,4) {$\cF( -\circ_{\cC}-)$};
\node[draw,rectangle] (mu-top) at (4.5,6) {$\quad\mu\quad$};
\node[draw,rectangle] (otimes) at (4.5,7) {$\cF( - \otimes _{\cC} - )$};
\node (top) at (4.5,8) {
};
\draw[double] (mu-R) to[in=-90,out=90]  (circ-R);
\draw[knot] (b-c) to[in=-90,out=90] (mu-L.-45);
\draw (e-f) to[in=-90,out=90] (mu-R.-45);
\draw (a-b) to[in=-90,out=90] (mu-L.-135);
\draw[knot] (d-e) to[in=-90,out=90] (mu-R.-135);
\draw[double] (mu-top) to[in=-90,out=90]  (otimes);
\draw (circ-L) to[in=-90,out=90]  (mu-top.-135);
\draw (circ-R) to[in=-90,out=90]  (mu-top.-45);
\draw (otimes) -- (top);
\end{tikzpicture}
=
\begin{tikzpicture}[baseline=50,smallstring]
\node (a-b) at (0,0) {$\cF_*\cC(a \to b)$};
\node (d-e) at (3,0) {$\cF_*\cC(d \to e)$};
\node (b-c) at (6,0) {$\cF_*\cC(b \to c)$};
\node (e-f) at (9,0) {$\cF_*\cC(e \to f)$};
\node[draw,rectangle] (t1) at (1.5,3) {$ -\circ_{\cF_*\cC}- $};
\node[draw,rectangle] (t2) at (7.5,3) {$ -\circ_{\cF_*\cC}- $};
\node[draw,rectangle] (c) at (4.5,5) {$ - \otimes_{\cF_*\cC} -  $};
\node (r) at (4.5,6) {$\cF_*\cC(ad \to cf)$};
\draw[knot] (b-c) to[in=-90,out=90] (t1.-45);
\draw (e-f) to[in=-90,out=90] (t2.-45);
\draw (a-b) to[in=-90,out=90] (t1.-135);
\draw[knot] (d-e) to[in=-90,out=90] (t2.-135);
\draw (t1) to[in=-90,out=90]  (c.-135);
\draw (t2) to[in=-90,out=90]  (c.-45);
\draw (c) -- (r);
\end{tikzpicture}
.
\end{align*}
\begin{example}
The functor $\cV \to \Vec$ given by $\cV(1_\cV \to -)$ is braided lax monoidal.
Transporting enrichment using this functor yields the underlying monoidal category $\cC^\cV$.
\end{example}

\subsection{Proofs for \texorpdfstring{$\cV$}{V}-natural transformations}
\label{sec:NaturalityOfNaturalTransformationComposition}

We now prove naturality of horizontal composition of $\cV$-natural transformations.

\begin{proof}[Proof of Lemma \ref{lem:HorizontalComposition}]
To verify naturality of the horizontal composition, we calculate
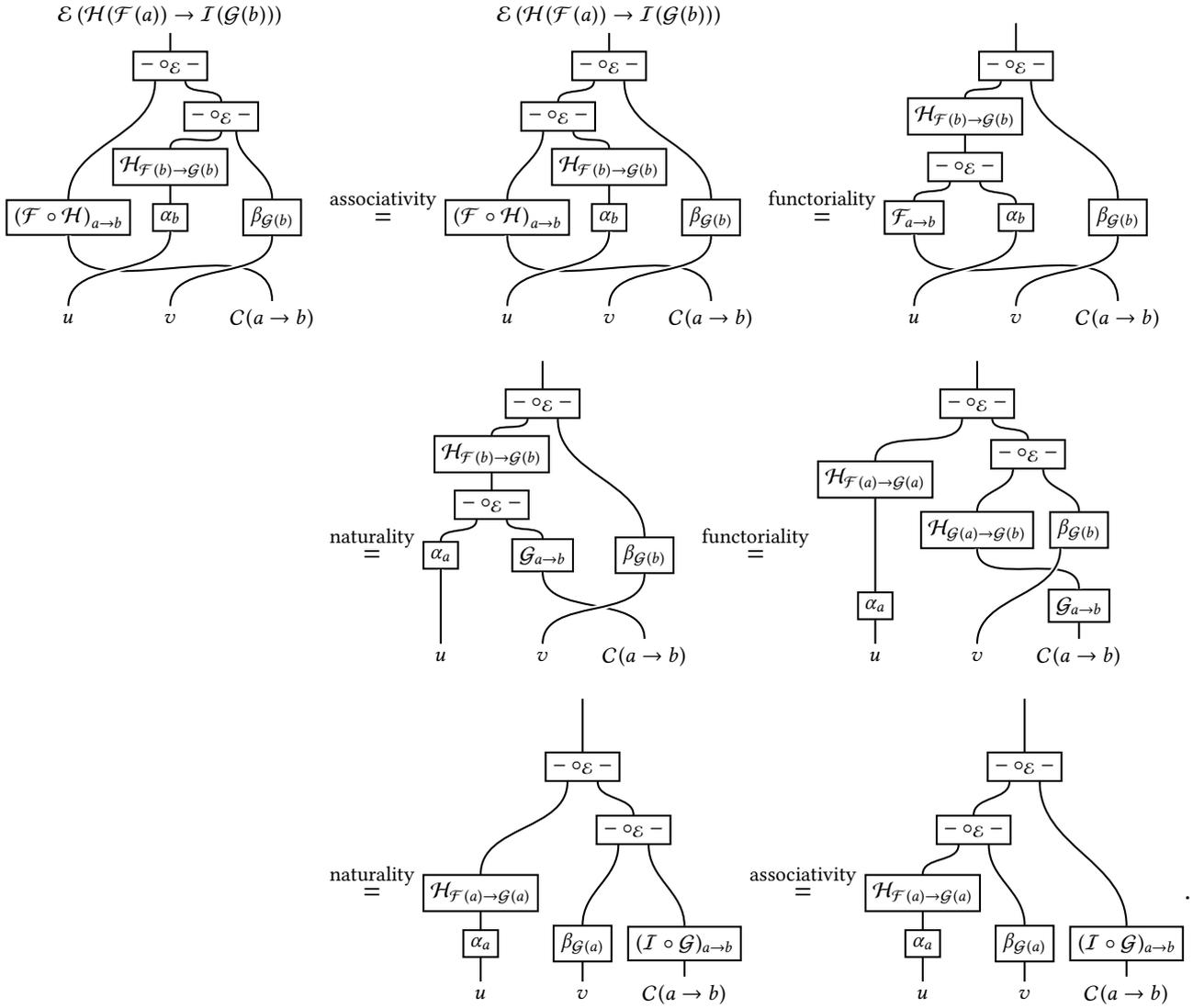
\begin{align*}
\begin{tikzpicture}[baseline=40,smallstring]
\node (u) at (0,0) {$u$};
\node (v) at (2,0) {$v$};
\node (a-b) at (4,0) {$\cC(a \to b)$};
\node[draw] (FH) at (0,2) {$\left(\cF \circ \cH\right)_{a \to b}$};
\node[draw] (alpha) at (2,2) {$\alpha_b$};
\node[draw] (beta) at (4,2) {$\beta_{\cG(b)}$};
\draw[knot] (a-b) to[in=-90, out=90] (FH);
\draw[knot] (u) to[in=-90, out=90] (alpha);
\draw[knot] (v) to[in=-90, out=90] (beta);
\node[draw] (H) at (2,3) {$\cH_{\cF(b) \to \cG(b)}$};
\draw (alpha) -- (H);
\node[draw] (compose1) at (3,4) {$-\circ_\cE-$};
\node[draw] (compose2) at (2,5) {$-\circ_\cE-$};
\draw (compose1) to[in=-90, out=90] (compose2.-45);
\draw (FH) to[in=-90, out=90] (compose2.-135);
\draw (H) to[in=-90, out=90] (compose1.-90);
\draw (beta) to[in=-90, out=90] (compose1.-45);
\node (target) at (2,6) {$\cE\left(\cH(\cF(a)) \to \cI(\cG(b))\right)$};
\draw (compose2) -- (target);
\end{tikzpicture}
& \overset{\text{associativity}}{=}
\begin{tikzpicture}[baseline=40,smallstring]
\node (u) at (0,0) {$u$};
\node (v) at (2,0) {$v$};
\node (a-b) at (4,0) {$\cC(a \to b)$};
\node[draw] (FH) at (0,2) {$\left(\cF \circ \cH\right)_{a \to b}$};
\node[draw] (alpha) at (2,2) {$\alpha_b$};
\node[draw] (beta) at (4,2) {$\beta_{\cG(b)}$};
\draw[knot] (a-b) to[in=-90, out=90] (FH);
\draw[knot] (u) to[in=-90, out=90] (alpha);
\draw[knot] (v) to[in=-90, out=90] (beta);
\node[draw] (H) at (2,3) {$\cH_{\cF(b) \to \cG(b)}$};
\draw (alpha) -- (H);
\node[draw] (compose1) at (1,4) {$-\circ_\cE-$};
\node[draw] (compose2) at (2,5) {$-\circ_\cE-$};
\draw (compose1) to[in=-90, out=90] (compose2.-135);
\draw (FH) to[in=-90, out=90] (compose1.-135);
\draw (H) to[in=-90, out=90] (compose1.-45);
\draw (beta) to[in=-90, out=90] (compose2.-45);
\node (target) at (2,6) {$\cE\left(\cH(\cF(a)) \to \cI(\cG(b))\right)$};
\draw (compose2) -- (target);
\end{tikzpicture}
 \overset{\text{functoriality}}{=}
\begin{tikzpicture}[baseline=40,smallstring]
\node (u) at (0,0) {$u$};
\node (v) at (2,0) {$v$};
\node (a-b) at (4,0) {$\cC(a \to b)$};
\node[draw] (F) at (0,2) {$\cF_{a \to b}$};
\node[draw] (alpha) at (2,2) {$\alpha_b$};
\node[draw] (beta) at (4,2) {$\beta_{\cG(b)}$};
\draw[knot] (a-b) to[in=-90, out=90] (F);
\draw[knot] (u) to[in=-90, out=90] (alpha);
\draw[knot] (v) to[in=-90, out=90] (beta);
\node[draw] (compose1) at (1,3) {$-\circ_\cE-$};
\draw (F) to[in=-90, out=90] (compose1.-135);
\draw (alpha) to[in=-90, out=90] (compose1.-45);
\node[draw] (H) at (1,4) {$\cH_{\cF(b) \to \cG(b)}$};
\draw (compose1) to[in=-90, out=90] (H);
\node[draw] (compose2) at (2,5) {$-\circ_\cE-$};
\draw (H) to[in=-90, out=90] (compose2.-135);
\draw (beta) to[in=-90, out=90] (compose2.-45);
\node (target) at (2,6) {};
\draw (compose2) -- (target);
\end{tikzpicture} \displaybreak[1]\\
& \overset{\text{naturality}}{=}
\begin{tikzpicture}[baseline=40,smallstring]
\node (u) at (0,0) {$u$};
\node (v) at (2,0) {$v$};
\node (a-b) at (4,0) {$\cC(a \to b)$};
\node[draw] (G) at (2,2) {$\cG_{a \to b}$};
\node[draw] (alpha) at (0,2) {$\alpha_a$};
\node[draw] (beta) at (4,2) {$\beta_{\cG(b)}$};
\draw[knot] (a-b) to[in=-90, out=90] (G);
\draw[knot] (u) to[in=-90, out=90] (alpha);
\draw[knot] (v) to[in=-90, out=90] (beta);
\node[draw] (compose1) at (1,3) {$-\circ_\cE-$};
\draw (G) to[in=-90, out=90] (compose1.-45);
\draw (alpha) to[in=-90, out=90] (compose1.-135);
\node[draw] (H) at (1,4) {$\cH_{\cF(b) \to \cG(b)}$};
\draw (compose1) to[in=-90, out=90] (H);
\node[draw] (compose2) at (2,5) {$-\circ_\cE-$};
\draw (H) to[in=-90, out=90] (compose2.-135);
\draw (beta) to[in=-90, out=90] (compose2.-45);
\node (target) at (2,6) {};
\draw (compose2) -- (target);
\end{tikzpicture} 
\overset{\text{functoriality}}{=}
\begin{tikzpicture}[baseline=40,smallstring]
\node (u) at (0,0) {$u$};
\node (v) at (2,0) {$v$};
\node (a-b) at (4,0) {$\cC(a \to b)$};
\node[draw] (G) at (4,1) {$\cG_{a \to b}$};
\node[draw] (alpha) at (0,1) {$\alpha_a$};
\node[draw] (beta) at (4,2.5) {$\beta_{\cG(b)}$};
\draw (a-b) to[in=-90, out=90] (G);
\draw (u) to[in=-90, out=90] (alpha);
\node[draw] (H1) at (0,3.5) {$\cH_{\cF(a) \to \cG(a)}$};
\node[draw] (H2) at (2,2.5) {$\cH_{\cG(a) \to \cG(b)}$};
\draw[knot] (G) to[in=-90, out=90] (H2);
\draw[knot] (v) to[in=-90, out=90] (beta.-135);
\draw (alpha) to[in=-90, out=90] (H1);
\node[draw] (compose1) at (3,4) {$-\circ_\cE-$};
\node[draw] (compose2) at (2,5) {$-\circ_\cE-$};
\draw (H2) to[in=-90, out=90] (compose1.-135);
\draw (beta) to[in=-90, out=90] (compose1.-45);
\draw (H1) to[in=-90, out=90] (compose2.-135);
\draw (compose1) to[in=-90, out=90] (compose2.-45);
\node (target) at (2,6) {};
\draw (compose2) -- (target);
\end{tikzpicture} \displaybreak[1] \\
& \overset{\text{naturality}}{=}
\begin{tikzpicture}[baseline=40,smallstring]
\node (u) at (0,0) {$u$};
\node (v) at (2,0) {$v$};
\node (a-b) at (4,0) {$\cC(a \to b)$};
\node[draw] (IG) at (4,1) {$(\cI \circ \cG)_{a \to b}$};
\node[draw] (alpha) at (0,1) {$\alpha_a$};
\node[draw] (beta) at (2,1) {$\beta_{\cG(a)}$};
\draw (a-b) to[in=-90, out=90] (IG);
\draw (u) to[in=-90, out=90] (alpha);
\node[draw] (H1) at (0,2) {$\cH_{\cF(a) \to \cG(a)}$};
\draw (v) to[in=-90, out=90] (beta);
\draw (alpha) to[in=-90, out=90] (H1);
\node[draw] (compose1) at (3,3.25) {$-\circ_\cE-$};
\node[draw] (compose2) at (2,4.5) {$-\circ_\cE-$};
\draw (IG) to[in=-90, out=90] (compose1.-45);
\draw (compose1) to[in=-90, out=90] (compose2.-45);
\draw (beta) to[in=-90, out=90] (compose1.-135);
\draw (H1) to[in=-90, out=90] (compose2.-135);
\node (target) at (2,6) {};
\draw (compose2) -- (target);
\end{tikzpicture}
\overset{\text{associativity}}{=}
\begin{tikzpicture}[baseline=40,smallstring]
\node (u) at (0,0) {$u$};
\node (v) at (2,0) {$v$};
\node (a-b) at (4,0) {$\cC(a \to b)$};
\node[draw] (IG) at (4,1) {$(\cI \circ \cG)_{a \to b}$};
\node[draw] (alpha) at (0,1) {$\alpha_a$};
\node[draw] (beta) at (2,1) {$\beta_{\cG(a)}$};
\draw (a-b) to[in=-90, out=90] (IG);
\draw (u) to[in=-90, out=90] (alpha);
\node[draw] (H1) at (0,2) {$\cH_{\cF(a) \to \cG(a)}$};
\draw (v) to[in=-90, out=90] (beta);
\draw (alpha) to[in=-90, out=90] (H1);
\node[draw] (compose1) at (1,3.25) {$-\circ_\cE-$};
\node[draw] (compose2) at (2,4.5) {$-\circ_\cE-$};
\draw (IG) to[in=-90, out=90] (compose2.-45);
\draw (compose1) to[in=-90, out=90] (compose2.-135);
\draw (beta) to[in=-90, out=90] (compose1.-45);
\draw (H1) to[in=-90, out=90] (compose1.-135);
\node (target) at (2,6) {};
\draw (compose2) -- (target);
\end{tikzpicture}.
\end{align*}

Associativity is now straightforward. If we have
$$
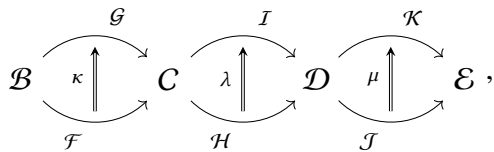
\begin{tikzpicture}[baseline=0]
\node (B) at (-2,0) {$\cB$};
\node (C) at (0,0) {$\cC$};
\node (D) at (2,0) {$\cD$};
\node (E) at (4,0) {$\cE$};
\draw[->] (B) to[out=-45,in=-135] node (BC1) {} node[below,xshift=-3mm] (F) {\tiny $\cF$}  (C);
\draw[->] (B) to[out=45,in=135] node (BC2) {} node[above,xshift=3mm] (G) {\tiny $\cG$}(C);
\draw[->] (C) to[out=-45,in=-135] node (CD1) {} node[below,xshift=-3mm] (H) {\tiny $\cH$}  (D);
\draw[->] (C) to[out=45,in=135] node (CD2) {} node[above,xshift=3mm] (I) {\tiny $\cI$}(D);
\draw[->] (D) to[out=-45,in=-135] node (DE1) {} node[below,xshift=-3mm] (J) {\tiny $\cJ$}  (E);
\draw[->] (D) to[out=45,in=135] node (DE2) {} node[above,xshift=3mm] (K) {\tiny $\cK$}(E);
\draw[double,arrows={-stealth}] (BC1) -- node[left] {\tiny $\kappa$} (BC2);
\draw[double,arrows={-stealth}] (CD1) -- node[left] {\tiny $\lambda$} (CD2);
\draw[double,arrows={-stealth}] (DE1) -- node[left] {\tiny $\mu$} (DE2);
\end{tikzpicture},
$$ then

\begin{align*}
(\kappa \lambda) \mu & =
\begin{tikzpicture}[smallstring,x=1cm,y=0.8cm,baseline=30]
\node[draw] (kappa)  at (0,0) {$\kappa$};
\node[draw] (lambda) at (1,0) {$\lambda$};
\node[draw] (mu)     at (2,0) {$\mu$}  ;
\node[draw] (H)      at (0,1) {$\cH$};
\node[draw] (c1)     at (0.5,2) { $- \circ -$ } ;
\node[draw] (J)      at (0.5,3) {$\cJ$};
\node[draw] (c2)     at (1,4)  { $- \circ -$ };
\draw (kappa) --++(0,-0.75);
\draw (lambda) --++(0,-0.75);
\draw (mu) --++(0,-0.75);
\draw (kappa) -- (H);
\draw (H)      to[out=90,in=-90] (c1.-135);
\draw (lambda) to[out=90,in=-90] (c1.-45);
\draw (c1) -- (J);
\draw (J)  to[out=90,in=-90] (c2.-135);
\draw (mu) to[out=90,in=-90] (c2.-45);
\draw (c2) --++(0,0.75);
\end{tikzpicture}
= 
\begin{tikzpicture}[smallstring,x=1cm,y=0.8cm,baseline=30]
\node[draw] (kappa)  at (-0.35,0) {$\kappa$};
\node[draw] (lambda) at (1,0) {$\lambda$};
\node[draw] (mu)     at (2,0) {$\mu$}  ;
\node[draw] (HJ)      at (-0.35,1) {$\cH {\circ} \cJ$};
\node[draw] (J)      at (1,1) {$\cJ$};
\node[draw] (c1)     at (1.5,3) { $- \circ -$ } ;
\node[draw] (c2)     at (1,4)  { $- \circ -$ };
\draw (kappa) --++(0,-0.75);
\draw (lambda) --++(0,-0.75);
\draw (mu) --++(0,-0.75);
\draw (kappa) -- (HJ);
\draw (lambda) to[out=90,in=-90] (J);
\draw (J)  to[out=90,in=-90] (c1.-135);
\draw (mu) to[out=90,in=-90] (c1.-45);
\draw (c1) to[out=90,in=-90] (c2.-45);
\draw (HJ)     to[out=90,in=-90] (c2.-135);
\draw (c2) --++(0,0.75);
\end{tikzpicture}
=\kappa(\lambda \mu). \qedhere
\end{align*}
\end{proof}

We now prove that $\cV$-natural transformations satisfy a braided interchange relation.

\begin{proof}[Proof of Lemma \ref{lem:BraidedInterchangeForNaturalTransformations}]
We have
\begin{align*}
\sigma_2 \circ ((\kappa \circ \mu)( \lambda \circ \nu)) & =
\begin{tikzpicture}[smallstring,x=1cm,y=1cm,baseline=30]
\node[draw] (kappa)  at (0,0) {$\kappa$};
\node[draw] (mu)     at (1,0) {$\mu$};
\node[draw] (lambda) at (2,0) {$\lambda$};
\node[draw] (nu)     at (3,0) {$\nu$};
\draw (kappa) --++(0,-1);
\draw       (mu)     to[out=-90,in=90] (2,-1);
\draw[knot] (lambda) to[out=-90,in=90] (1,-1);
\draw (nu) --++(0,-1);
\node[draw] (c1) at (0.5,1) {$ - {\circ} -$};
\node[draw] (c2) at (2.5,1) {$ - {\circ} -$};
\node[draw] (I) at (0.5,2) {$\cI$};
\node[draw] (c3) at (1.5,3) {$ - {\circ} -$};
\draw (kappa)  to[out=90,in=-90] (c1.-135);
\draw (mu)     to[out=90,in=-90] (c1.-45);
\draw (lambda) to[out=90,in=-90] (c2.-135);
\draw (nu)     to[out=90,in=-90] (c2.-45);
\draw (c1) -- (I);
\draw (c2) to[out=90,in=-90] (c3.-45);
\draw (I)  to[out=90,in=-90] (c3.-135);
\draw (c3) --++(0,0.75);
\end{tikzpicture}
\overset{\parbox{2cm}{\center \tiny functoriality \\ and \\ associativity}}{=}
\begin{tikzpicture}[smallstring,x=1cm,y=1cm,baseline=30]
\node[draw] (kappa)  at (0,-0.5) {$\kappa$};
\node[draw] (mu)     at (2,-0.5) {$\mu$};
\node[draw] (lambda) at (2,1) {$\lambda$};
\node[draw] (nu)     at (3,-0.5) {$\nu$};
\node[draw] (c) at (1.5,3) {$ - \circ - $};
\node[draw] (I1) at (0,1) {$\cI$};
\node[draw] (I2) at (1,1) {$\cI$};
\draw (kappa)  to[out=90,in=-90] (I1);
\draw (mu)     to[out=90,in=-90] (I2);
\draw (I1)     to[out=90,in=-90] (c.-150);
\draw (I2)     to[out=90,in=-90] (c.-110);
\draw (lambda) to[out=90,in=-90] (c.-70);
\draw (nu)     to[out=90,in=-90] (c.-30);
\draw (c) --++(0,0.75);
\draw (kappa) --++(0,-0.5);
\draw       (mu)     to[out=-90,in=90] (2,-1);
\draw[knot] (lambda) to[out=-90,in=90] (1,-0.3) -- (1,-1);
\draw (nu) --++(0,-0.5);
\end{tikzpicture} \displaybreak[1] \\
&
\overset{\parbox{2cm}{\tiny \center naturality \\ of $\lambda$}}{=}
\begin{tikzpicture}[smallstring,x=1cm,y=1cm,baseline=30]
\node[draw] (kappa)  at (0,-0.25) {$\kappa$};
\node[draw] (mu)     at (2,-0.25) {$\mu$};
\node[draw] (lambda) at (1,-0.25) {$\lambda$};
\node[draw] (nu)     at (3,-0.25) {$\nu$};
\node[draw] (c) at (1.5,3) {$ - \circ - $};
\node[draw] (I) at (0,1) {$\cI$};
\node[draw] (J) at (2,1) {$\cJ$};
\draw (kappa)  to[out=90,in=-90] (I);
\draw (mu)     to[out=90,in=-90] (J);
\draw (I)     to[out=90,in=-90] (c.-150);
\draw (J)     to[out=90,in=-90] (c.-70);
\draw (lambda) to[out=90,in=-90] (c.-110);
\draw (nu)     to[out=90,in=-90] (c.-30);
\draw (c) --++(0,0.75);
\draw (kappa) --++(0,-0.75);
\draw (lambda) --++(0,-0.75);
\draw (mu)     --++ (0,-0.75);
\draw (nu) --++(0,-0.75);
\end{tikzpicture}
=(\kappa \lambda)\circ( \mu \nu).
\qedhere
\end{align*}
\end{proof}

\subsection{Separate naturality conditions for \texorpdfstring{$\cV$}
{V}-monoidal functors}

The following lemma is useful for proving naturality for a $\cV$-monoidal functor.

\begin{lem}
\label{lem:SeparateNaturalityForVMonoidal}
The naturality condition \eqref{eq:NaturalityForVMonoidal} is equivalent to two separate naturality conditions, one in each factor: for all $a\in \cC$
\begin{equation}
\label{eq:LeftNaturalityForVMonoidal}
\begin{tikzcd}
\cC(a \to c)
	\ar[r, "-\otimes_\cC j_b"]
	\ar[d,"\cF_{a\to c}"]
& \cC(ab \to cb) \ar[dr, "\cF_{ab \to cb}"]	
&
\\
\cD(\cF(a)\to \cF(c)) \ar[dr, "-\otimes_\cD j_{\cF(b)}"]
& &
\cD(\cF(ab) \to \cF(cb)) \ar[d, "-\circ \alpha_{c,b}"]	  
\\
& 
\cD(\cF(a)\cF(b) \to \cF(c)\cF(b)) \ar[r, "\alpha_{a,b}\circ-"] 
&
\cD(\cF(ab) \to \cF(c)\cF(b))
\end{tikzcd}
\end{equation}
and for all $b \in \cC$
\begin{equation}
\label{eq:RightNaturalityForVMonoidal}
\begin{tikzcd}
\cC(b \to d)
	\ar[r, "j_a \otimes_\cC-"]
	\ar[d,"\cF_{b\to d}"]
& \cC(ab \to ad) \ar[dr, "\cF_{ab \to ad}"]	
&
\\
\cD(\cF(b)\to \cF(d)) \ar[dr, "j_{\cF(a)} \otimes_\cD-"]
& &
\cD(\cF(ab) \to \cF(ad)) \ar[d, "-\circ \alpha_{a,d}"]	  
\\
& 
\cD(\cF(a)\cF(b) \to \cF(a)\cF(d)) \ar[r, "\alpha_{a,b}\circ-"] 
&
\cD(\cF(ab) \to \cF(a)\cF(d)).
\end{tikzcd}
\end{equation}
\end{lem}
\begin{proof}
The naturality condition \eqref{eq:NaturalityForVMonoidal} gives the pair of conditions \eqref{eq:LeftNaturalityForVMonoidal} and \eqref{eq:RightNaturalityForVMonoidal} by taking $b=d$ and precomposing with $1 j_b$,
or by taking  $a=c$ and precomposing with $j_a 1$.
We give a terse calculation that \eqref{eq:NaturalityForVMonoidal} follows from \eqref{eq:LeftNaturalityForVMonoidal} and \eqref{eq:RightNaturalityForVMonoidal}, without fully labelling all the boxes.
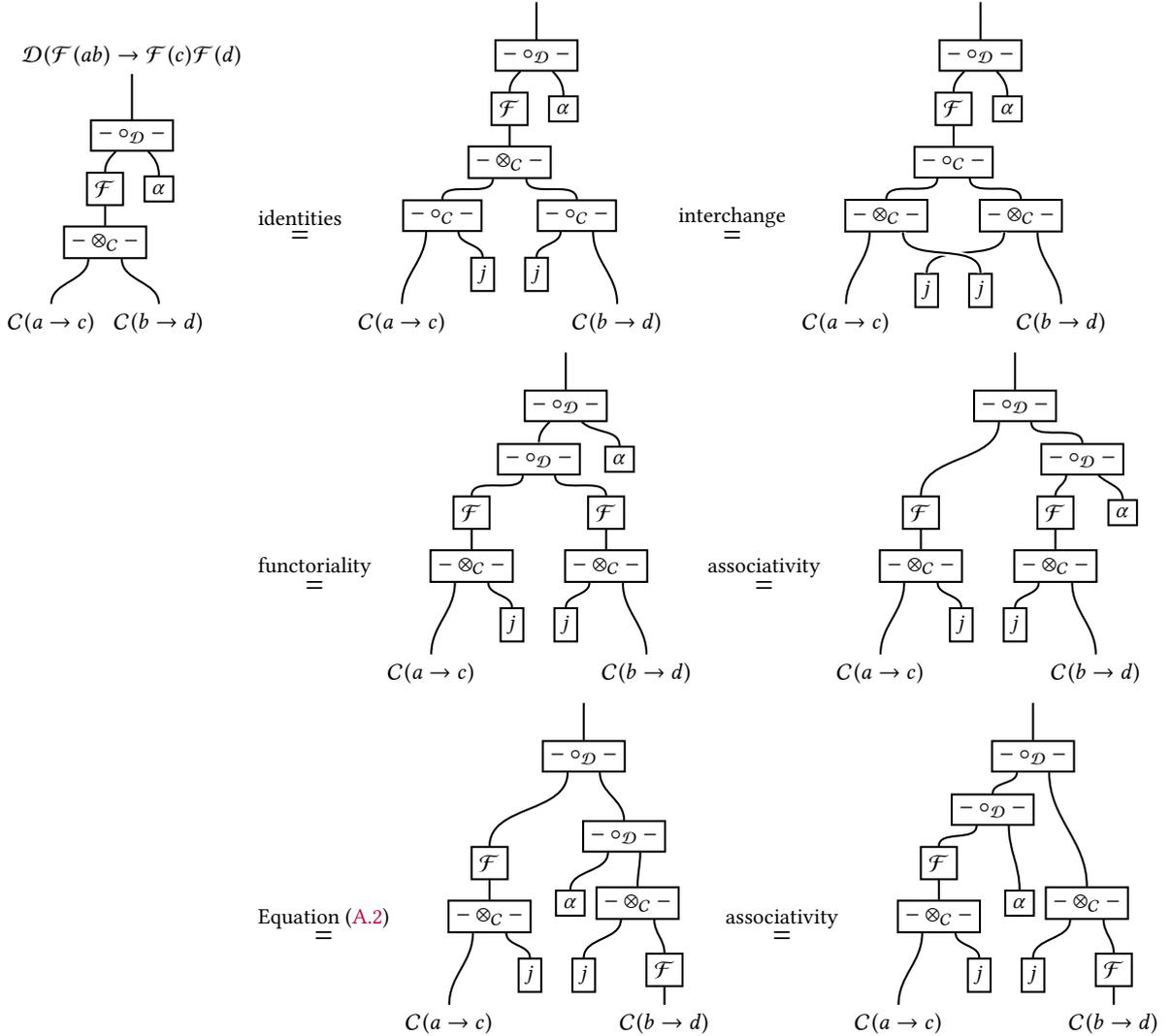
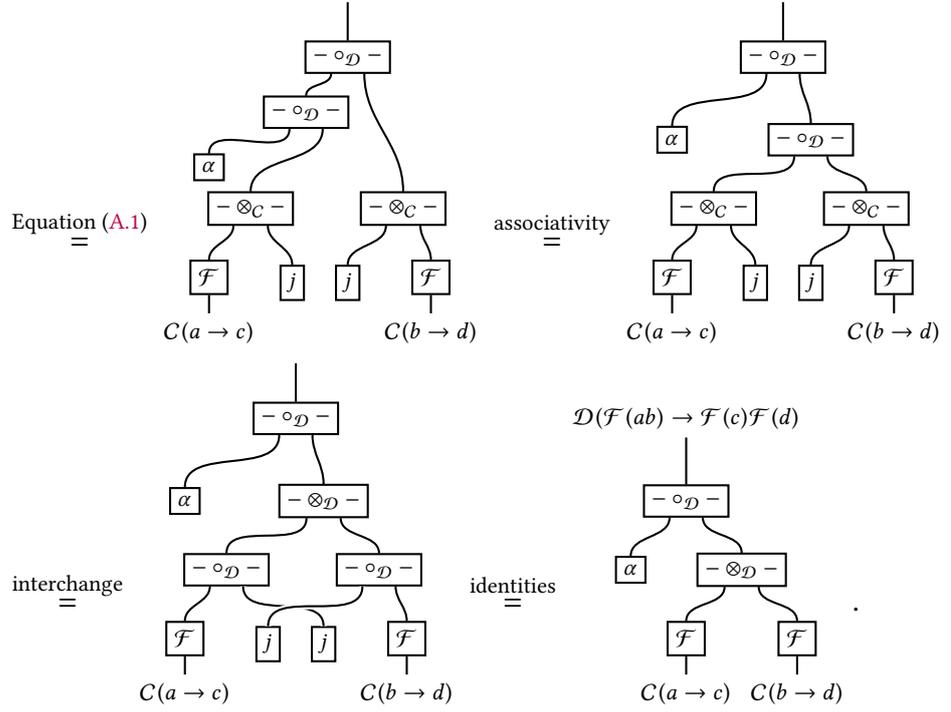
\begin{align*}
\begin{tikzpicture}[baseline=32,smallstring]
\node (ac) at (0,0) {$\cC(a \to c)$};
\node (bd) at (2,0) {$\cC(b \to d)$};
\node[draw] (ab-cd) at (1,1.5) {$- \otimes_\cC -$};
\draw (ac) to[in=-90,out=90] (ab-cd.-135);
\draw (bd) to[in=-90,out=90] (ab-cd.-45);
\node[draw] (F) at (1,2.5) {$\cF$};
\draw (ab-cd) -- (F);
\node[draw] (compose) at (1.5,3.5) {$- \circ_\cD -$};
\node[draw] (alpha) at (2,2.5) {$\alpha$};
\draw (F) to[in=-135,out=90] (compose);
\draw (alpha) to[in=-45,out=90] (compose);
\node (top) at (1.5,5) {$\cD(\cF(ab) \to \cF(c) \cF(d)$};
\draw (compose) -- (top);
\end{tikzpicture}
& \overset{\text{identities}}{=}
\begin{tikzpicture}[baseline,smallstring]
\node (ac) at (-1,-1.5) {$\cC(a \to c)$};
\node (bd) at (3,-1.5) {$\cC(b \to d)$};
\node[draw] (j1) at (0.5,-0.6) {$j$};
\node[draw] (j2) at (1.5,-0.6) {$j$};
\node[draw] (c1) at (-0.25,0.5) {$-\circ_\cC-$};
\node[draw] (c2) at (2.25,0.5) {$-\circ_\cC-$};
\node[draw] (ab-cd) at (1,1.5) {$- \otimes_\cC -$};
\draw (ac) to[in=-90,out=90] (c1.-135);
\draw (j1) to[in=-90,out=90] (c1.-45);
\draw (bd) to[in=-90,out=90] (c2.-45);
\draw (j2) to[in=-90,out=90] (c2.-135);
\draw (c1) to[in=-90,out=90] (ab-cd.-135);
\draw (c2) to[in=-90,out=90] (ab-cd.-45);
\node[draw] (F) at (1,2.5) {$\cF$};
\draw (ab-cd) -- (F);
\node[draw] (compose) at (1.5,3.5) {$- \circ_\cD -$};
\node[draw] (alpha) at (2,2.5) {$\alpha$};
\draw (F) to[in=-135,out=90] (compose);
\draw (alpha) to[in=-45,out=90] (compose);
\draw (compose) -- ++(0,1);
\end{tikzpicture}
\overset{\text{interchange}}{=}
\begin{tikzpicture}[baseline,smallstring]
\node (ac) at (-1,-1.5) {$\cC(a \to c)$};
\node (bd) at (3,-1.5) {$\cC(b \to d)$};
\node[draw] (j1) at (0.5,-0.9) {$j$};
\node[draw] (j2) at (1.5,-0.9) {$j$};
\node[draw] (c1) at (-0.25,0.5) {$-\otimes_\cC-$};
\node[draw] (c2) at (2.25,0.5) {$-\otimes_\cC-$};
\node[draw] (ab-cd) at (1,1.5) {$- \circ_\cC -$};
\draw (ac) to[in=-90,out=90] (c1.-135);
\draw[knot] (j1) to[in=-90,out=90] (c2.-135);
\draw (bd) to[in=-90,out=90] (c2.-45);
\draw[knot] (j2) to[in=-90,out=90] (c1.-45);
\draw (c1) to[in=-90,out=90] (ab-cd.-135);
\draw (c2) to[in=-90,out=90] (ab-cd.-45);
\node[draw] (F) at (1,2.5) {$\cF$};
\draw (ab-cd) -- (F);
\node[draw] (compose) at (1.5,3.5) {$- \circ_\cD -$};
\node[draw] (alpha) at (2,2.5) {$\alpha$};
\draw (F) to[in=-135,out=90] (compose);
\draw (alpha) to[in=-45,out=90] (compose);
\draw (compose) -- ++(0,1);
\end{tikzpicture} \displaybreak[1]\\
&
\overset{\text{functoriality}}{=}
\begin{tikzpicture}[baseline,smallstring]
\node (ac) at (-1,-1.5) {$\cC(a \to c)$};
\node (bd) at (3,-1.5) {$\cC(b \to d)$};
\node[draw] (j1) at (0.5,-0.6) {$j$};
\node[draw] (j2) at (1.5,-0.6) {$j$};
\node[draw] (c1) at (-0.25,0.5) {$-\otimes_\cC-$};
\node[draw] (c2) at (2.25,0.5) {$-\otimes_\cC-$};
\draw (ac) to[in=-90,out=90] (c1.-135);
\draw (j1) to[in=-90,out=90] (c1.-45);
\draw (bd) to[in=-90,out=90] (c2.-45);
\draw (j2) to[in=-90,out=90] (c2.-135);
\node[draw] (F1) at (-0.25,1.5) {$\cF$};
\node[draw] (F2) at (2.25,1.5) {$\cF$};
\draw (c1) to[in=-90,out=90] (F1);
\draw (c2) to[in=-90,out=90] (F2);
\node[draw] (ab-cd) at (1,2.5) {$- \circ_\cD -$};
\draw (F1) to[in=-90,out=90] (ab-cd.-135);
\draw (F2) to[in=-90,out=90] (ab-cd.-45);
\node[draw] (compose) at (1.5,3.5) {$- \circ_\cD -$};
\node[draw] (alpha) at (2.5,2.5) {$\alpha$};
\draw (F) to[in=-135,out=90] (compose);
\draw (alpha) to[in=-45,out=90] (compose);
\draw (compose) -- ++(0,1);
\end{tikzpicture}
\overset{\text{associativity}}{=}
\begin{tikzpicture}[baseline,smallstring]
\node (ac) at (-1,-1.5) {$\cC(a \to c)$};
\node (bd) at (3,-1.5) {$\cC(b \to d)$};
\node[draw] (j1) at (0.5,-0.6) {$j$};
\node[draw] (j2) at (1.5,-0.6) {$j$};
\node[draw] (c1) at (-0.25,0.5) {$-\otimes_\cC-$};
\node[draw] (c2) at (2.25,0.5) {$-\otimes_\cC-$};
\draw (ac) to[in=-90,out=90] (c1.-135);
\draw (j1) to[in=-90,out=90] (c1.-45);
\draw (bd) to[in=-90,out=90] (c2.-45);
\draw (j2) to[in=-90,out=90] (c2.-135);
\node[draw] (F1) at (-0.25,1.5) {$\cF$};
\node[draw] (F2) at (2.25,1.5) {$\cF$};
\draw (c1) to[in=-90,out=90] (F1);
\draw (c2) to[in=-90,out=90] (F2);
\node[draw] (ab-cd) at (2.75,2.5) {$- \circ_\cD -$};
\node[draw] (compose) at (1.5,3.5) {$- \circ_\cD -$};
\draw (F1) to[in=-90,out=90] (compose.-135);
\draw (F2) to[in=-90,out=90] (ab-cd.-135);
\node[draw] (alpha) at (3.5,1.5) {$\alpha$};
\draw (ab-cd) to[in=-90,out=90] (compose.-45);
\draw (alpha) to[in=-45,out=90] (ab-cd.-45);
\draw (compose) -- ++(0,1);
\end{tikzpicture} \displaybreak[1]\\
& \overset{\text{Equation \eqref{eq:RightNaturalityForVMonoidal}}}{=}
\begin{tikzpicture}[baseline,smallstring]
\node (ac) at (-1,-1.5) {$\cC(a \to c)$};
\node (bd) at (3,-1.5) {$\cC(b \to d)$};
\node[draw] (j1) at (0.5,-0.6) {$j$};
\node[draw] (j2) at (1.5,-0.6) {$j$};
\node[draw] (c1) at (-0.25,0.5) {$-\otimes_\cC-$};
\node[draw] (c2) at (2.5,0.75) {$-\otimes_\cC-$};
\draw (ac) to[in=-90,out=90] (c1.-135);
\draw (j1) to[in=-90,out=90] (c1.-45);
\draw (j2) to[in=-90,out=90] (c2.-135);
\node[draw] (F1) at (-0.25,1.5) {$\cF$};
\node[draw] (F2) at (3,-0.5) {$\cF$};
\draw (bd) to[in=-90,out=90] (F2);
\draw (c1) to[in=-90,out=90] (F1);
\node[draw] (ab-cd) at (2.25,2) {$- \circ_\cD -$};
\draw (c2) to[in=-90,out=90] (ab-cd.-45);
\node[draw] (compose) at (1.5,3.5) {$- \circ_\cD -$};
\draw (F1) to[in=-90,out=90] (compose.-135);
\draw (F2) to[in=-90,out=90] (c2.-45);
\node[draw] (alpha) at (1.25,0.75) {$\alpha$};
\draw (ab-cd) to[in=-90,out=90] (compose.-45);
\draw (alpha) to[in=-90,out=90] (ab-cd.-135);
\draw (compose) -- ++(0,1);
\end{tikzpicture}
\overset{\text{associativity}}{=}
\begin{tikzpicture}[baseline,smallstring]
\node (ac) at (-1,-1.5) {$\cC(a \to c)$};
\node (bd) at (3,-1.5) {$\cC(b \to d)$};
\node[draw] (j1) at (0.5,-0.6) {$j$};
\node[draw] (j2) at (1.5,-0.6) {$j$};
\node[draw] (c1) at (-0.25,0.5) {$-\otimes_\cC-$};
\node[draw] (c2) at (2.5,0.75) {$-\otimes_\cC-$};
\draw (ac) to[in=-90,out=90] (c1.-135);
\draw (j1) to[in=-90,out=90] (c1.-45);
\draw (j2) to[in=-90,out=90] (c2.-135);
\node[draw] (F1) at (-0.25,1.5) {$\cF$};
\node[draw] (F2) at (3,-0.5) {$\cF$};
\draw (bd) to[in=-90,out=90] (F2);
\draw (c1) to[in=-90,out=90] (F1);
\node[draw] (ab-cd) at (0.75,2.5) {$- \circ_\cD -$};
\node[draw] (compose) at (1.5,3.5) {$- \circ_\cD -$};
\draw (c2) to[in=-90,out=90] (compose.-45);
\draw (F1) to[in=-90,out=90] (ab-cd.-135);
\draw (F2) to[in=-90,out=90] (c2.-45);
\node[draw] (alpha) at (1.25,0.75) {$\alpha$};
\draw (ab-cd) to[in=-90,out=90] (compose.-135);
\draw (alpha) to[in=-90,out=90] (ab-cd.-45);
\draw (compose) -- ++(0,1);
\end{tikzpicture} \displaybreak[1]\\
& \overset{\text{Equation \eqref{eq:LeftNaturalityForVMonoidal}}}{=}
\begin{tikzpicture}[baseline,smallstring]
\node (ac) at (-1,-1.5) {$\cC(a \to c)$};
\node (bd) at (3,-1.5) {$\cC(b \to d)$};
\node[draw] (j1) at (0.5,-0.6) {$j$};
\node[draw] (j2) at (1.5,-0.6) {$j$};
\node[draw] (c1) at (-0.25,0.75) {$-\otimes_\cC-$};
\node[draw] (c2) at (2.5,0.75) {$-\otimes_\cC-$};
\draw (j1) to[in=-90,out=90] (c1.-45);
\draw (j2) to[in=-90,out=90] (c2.-135);
\node[draw] (F1) at (-1,-0.5) {$\cF$};
\node[draw] (F2) at (3,-0.5) {$\cF$};
\draw (ac) to[in=-90,out=90] (F1);
\draw (bd) to[in=-90,out=90] (F2);
\draw (c1) to[in=-90,out=90] (ab-cd.-45);
\node[draw] (ab-cd) at (0.75,2.5) {$- \circ_\cD -$};
\node[draw] (compose) at (1.5,3.5) {$- \circ_\cD -$};
\draw (c2) to[in=-90,out=90] (compose.-45);
\draw (F1) to[in=-90,out=90] (c1.-135);
\draw (F2) to[in=-90,out=90] (c2.-45);
\node[draw] (alpha) at (-1,1.5) {$\alpha$};
\draw (ab-cd) to[in=-90,out=90] (compose.-135);
\draw (alpha) to[in=-90,out=90] (ab-cd.-135);
\draw (compose) -- ++(0,1);
\end{tikzpicture}
\overset{\text{associativity}}{=}
\begin{tikzpicture}[baseline,smallstring]
\node (ac) at (-1,-1.5) {$\cC(a \to c)$};
\node (bd) at (3,-1.5) {$\cC(b \to d)$};
\node[draw] (j1) at (0.5,-0.6) {$j$};
\node[draw] (j2) at (1.5,-0.6) {$j$};
\node[draw] (c1) at (-0.25,0.75) {$-\otimes_\cC-$};
\node[draw] (c2) at (2.5,0.75) {$-\otimes_\cC-$};
\draw (j1) to[in=-90,out=90] (c1.-45);
\draw (j2) to[in=-90,out=90] (c2.-135);
\node[draw] (F1) at (-1,-0.5) {$\cF$};
\node[draw] (F2) at (3,-0.5) {$\cF$};
\draw (ac) to[in=-90,out=90] (F1);
\draw (bd) to[in=-90,out=90] (F2);
\node[draw] (ab-cd) at (1.5,2) {$- \circ_\cD -$};
\node[draw] (compose) at (1,3.5) {$- \circ_\cD -$};
\draw (c1) to[in=-90,out=90] (ab-cd.-135);
\draw (c2) to[in=-90,out=90] (ab-cd.-45);
\draw (F1) to[in=-90,out=90] (c1.-135);
\draw (F2) to[in=-90,out=90] (c2.-45);
\node[draw] (alpha) at (-1,2) {$\alpha$};
\draw (ab-cd) to[in=-90,out=90] (compose.-45);
\draw (alpha) to[in=-90,out=90] (compose.-135);
\draw (compose) -- ++(0,1);
\end{tikzpicture}
\displaybreak[1]\\
& \overset{\text{interchange}}{=}
\begin{tikzpicture}[baseline,smallstring]
\node (ac) at (-1,-1.5) {$\cC(a \to c)$};
\node (bd) at (3,-1.5) {$\cC(b \to d)$};
\node[draw] (j1) at (0.5,-0.6) {$j$};
\node[draw] (j2) at (1.5,-0.6) {$j$};
\node[draw] (c1) at (-0.25,0.75) {$-\circ_\cD-$};
\node[draw] (c2) at (2.5,0.75) {$-\circ_\cD-$};
\draw[knot] (j2) to[in=-90,out=90] (c1.-45);
\draw[knot] (j1) to[in=-90,out=90] (c2.-135);
\node[draw] (F1) at (-1,-0.5) {$\cF$};
\node[draw] (F2) at (3,-0.5) {$\cF$};
\draw (ac) to[in=-90,out=90] (F1);
\draw (bd) to[in=-90,out=90] (F2);
\node[draw] (ab-cd) at (1.5,2) {$- \otimes_\cD -$};
\node[draw] (compose) at (1,3.5) {$- \circ_\cD -$};
\draw (c1) to[in=-90,out=90] (ab-cd.-135);
\draw (c2) to[in=-90,out=90] (ab-cd.-45);
\draw (F1) to[in=-90,out=90] (c1.-135);
\draw (F2) to[in=-90,out=90] (c2.-45);
\node[draw] (alpha) at (-1,2) {$\alpha$};
\draw (ab-cd) to[in=-90,out=90] (compose.-45);
\draw (alpha) to[in=-90,out=90] (compose.-135);
\draw (compose) -- ++(0,1);
\end{tikzpicture}
\overset{\text{identities}}{=}
\begin{tikzpicture}[baseline,smallstring]
\node (ac) at (1,-1.5) {$\cC(a \to c)$};
\node (bd) at (3,-1.5) {$\cC(b \to d)$};
\node[draw] (F1) at (1,-0.5) {$\cF$};
\node[draw] (F2) at (3,-0.5) {$\cF$};
\draw (ac) to[in=-90,out=90] (F1);
\draw (bd) to[in=-90,out=90] (F2);
\node[draw] (ab-cd) at (2,0.75) {$- \otimes_\cD -$};
\node[draw] (compose) at (1,2) {$- \circ_\cD -$};
\draw (F1) to[in=-90,out=90] (ab-cd.-135);
\draw (F2) to[in=-90,out=90] (ab-cd.-45);
\node[draw] (alpha) at (0,0.75) {$\alpha$};
\draw (ab-cd) to[in=-90,out=90] (compose.-45);
\draw (alpha) to[in=-90,out=90] (compose.-135);
\node (top) at (1,3.5) {$\cD(\cF(ab) \to \cF(c) \cF(d)$};
\draw (compose) -- (top);
\end{tikzpicture}.
\qedhere
\end{align*}
\end{proof}

\section{Proofs using adjunctions and mates}
\label{sec:ProofsAdjunctionsAndMates}

\noindent
\textbf{Lemma \ref{lem:MateOfId - aFv}}
\textit{
The mate of $\id_{a\cF(v)}\in \cC^\cV(a\cF(v)\to a\cF(v))$ is given by $$(j_a \eta_v) \circ (-\otimes_\cC -)\in \cV(v\to
\cC(a\to a\cF(v))).$$
}
\begin{proof}
We begin with $\id_{a\cF(v)}\in \cC^\cV(a\cF(v)\to a\cF(v))$ and perform the first isomorphism in Adjunction \eqref{eq:La-Adjunction} to obtain $(\coev_a\id_{\cF(v)})(-\otimes_\cC-)\in\cC^\cV(\cF(v)\to a^*a\cF(v))$.
Applying the next isomorphism gives us its mate 
$$
\begin{tikzpicture}[baseline=50,smallstring]
\node (v) at (2,0) {$v$};
\node (1-2) at (4,0) {$\iV$};
\node (1-3) at (6,0) {$\iV$};
\node[draw,rectangle] (eta-v) at (2,1) {$\eta_{v}$};
\node[draw,rectangle] (j-Fv-1) at (6,1) {$j_{\cF(v)}$};
\node[draw,rectangle] (coev-a) at (4,1) {$\coev_{a}$};
\node[draw,rectangle] (otimes-1) at (5,2) {$-\otimes_\cC -$};
\node[draw,rectangle] (circ-1) at (4,3) {$-\circ_\cC -$};
\node (1-a*aFv) at (4,4) {$\cC(1_\cC\to a^*a\cF(v))$};
\draw (v) to[in=-90,out=90] (eta-v);
\draw (circ-1) to[in=-90,out=90] (1-a*aFv);
\draw (eta-v) to[in=-90,out=90] (circ-1.-135);
\draw (otimes-1) to[in=-90,out=90] (circ-1.-45);
\draw (coev-a) to[in=-90,out=90] (otimes-1.-135);
\draw (j-Fv-1) to[in=-90,out=90] (otimes-1.-45);
\end{tikzpicture}
\in \cV(v\to \cC(1_\cC\to a^*a\cF(v)))
$$
under Adjunction \ref{eq:TraceAdjunction}, which we computed using Remark \ref{rem:Mates}.
We now apply Frobenius reciprocity in $\cC$ to obtain the following map in $\cV(v\to \cC(a\to a\cF(v)))$,
which we need to show is equal to $(j_a\eta_v)(-\otimes_\cC-)$:
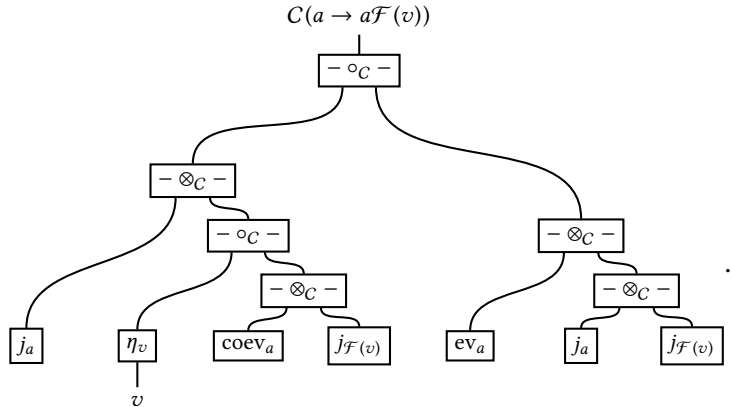
$$
\begin{tikzpicture}[baseline=50,smallstring]
\node (1-1) at (0,0) {$\iV$};
\node (v) at (2,0) {$v$};
\node (1-2) at (4,0) {$\iV$};
\node (1-3) at (6,0) {$\iV$};
\node (1-4) at (8,0) {$\iV$};
\node (1-5) at (10,0) {$\iV$};
\node (1-6) at (12,0) {$\iV$};
\node[draw,rectangle] (j-a-1) at (0,1) {$j_a$};
\node[draw,rectangle] (eta-v) at (2,1) {$\eta_{v}$};
\node[draw,rectangle] (j-Fv-1) at (6,1) {$j_{\cF(v)}$};
\node[draw,rectangle] (coev-a) at (4,1) {$\coev_{a}$};
\node[draw,rectangle] (ev-a) at (8,1) {$\ev_a$};
\node[draw,rectangle] (j-a-2) at (10,1) {$j_{a}$};
\node[draw,rectangle] (j-Fv-2) at (12,1) {$j_{\cF(v)}$};
\node[draw,rectangle] (otimes-1) at (5,2) {$-\otimes_\cC -$};
\node[draw,rectangle] (otimes-2) at (11,2) {$-\otimes_\cC -$};
\node[draw,rectangle] (otimes-3) at (10,3) {$-\otimes_\cC -$};
\node[draw,rectangle] (otimes-4) at (3,4) {$-\otimes_\cC -$};
\node[draw,rectangle] (circ-1) at (4,3) {$-\circ_\cC -$};
\node[draw,rectangle] (circ-2) at (6,6) {$-\circ_\cC -$};
\node (a-aFv) at (6,7) {$\cC(a\to a\cF(v))$};
\draw (v) to[in=-90,out=90] (eta-v);
\draw (j-a-1) to[in=-90,out=90] (otimes-4.-135);
\draw (circ-1) to[in=-90,out=90] (otimes-4.-45);
\draw (eta-v) to[in=-90,out=90] (circ-1.-135);
\draw (otimes-1) to[in=-90,out=90] (circ-1.-45);
\draw (coev-a) to[in=-90,out=90] (otimes-1.-135);
\draw (j-Fv-1) to[in=-90,out=90] (otimes-1.-45);
\draw (otimes-4) to[in=-90,out=90] (circ-2.-135);
\draw (otimes-3) to[in=-90,out=90] (circ-2.-45);
\draw (otimes-2) to[in=-90,out=90] (otimes-3.-45);
\draw (ev-a) to[in=-90,out=90] (otimes-3.-135);
\draw (j-a-2) to[in=-90,out=90] (otimes-2.-135);
\draw (j-Fv-2) to[in=-90,out=90] (otimes-2.-45);
\draw (circ-2) to[in=-90,out=90] (a-aFv);
\end{tikzpicture}
.
$$
Applying the exchange relation allows us to simplify the lower right side.
We then use associativity of tensor product in two places to obtain
$$
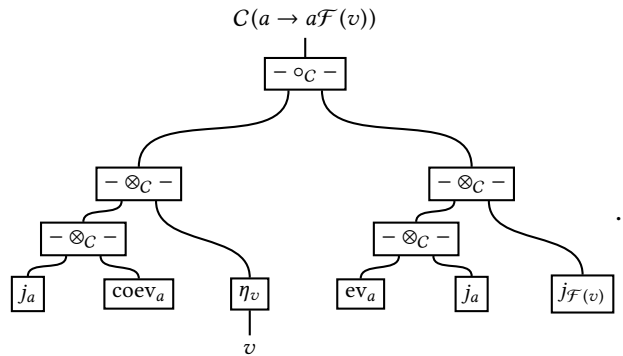
\begin{tikzpicture}[baseline=50,smallstring]
\node (1-1) at (2,0) {$\iV$};
\node (1-2) at (4,0) {$\iV$};
\node (v) at (6,0) {$v$};
\node (1-3) at (8,0) {$\iV$};
\node (1-4) at (10,0) {$\iV$};
\node (1-5) at (12,0) {$\iV$};
\node[draw,rectangle] (j-a-1) at (2,1) {$j_a$};
\node[draw,rectangle] (coev-a) at (4,1) {$\coev_{a}$};
\node[draw,rectangle] (eta-v) at (6,1) {$\eta_{v}$};
\node[draw,rectangle] (ev-a) at (8,1) {$\ev_a$};
\node[draw,rectangle] (j-a-2) at (10,1) {$j_{a}$};
\node[draw,rectangle] (j-Fv) at (12,1) {$j_{\cF(v)}$};
\node[draw,rectangle] (otimes-1) at (3,2) {$-\otimes_\cC -$};
\node[draw,rectangle] (otimes-2) at (9,2) {$-\otimes_\cC -$};
\node[draw,rectangle] (otimes-3) at (10,3) {$-\otimes_\cC -$};
\node[draw,rectangle] (otimes-4) at (4,3) {$-\otimes_\cC -$};
\node[draw,rectangle] (circ) at (7,5) {$-\circ_\cC -$};
\node (a-aFv) at (7,6) {$\cC(a\to a\cF(v))$};
\draw (v) to[in=-90,out=90] (eta-v);
\draw (j-a-1) to[in=-90,out=90] (otimes-1.-135);
\draw (eta-v) to[in=-90,out=90] (otimes-4.-45);
\draw (otimes-1) to[in=-90,out=90] (otimes-4.-135);
\draw (coev-a) to[in=-90,out=90] (otimes-1.-45);
\draw (otimes-4) to[in=-90,out=90] (circ.-135);
\draw (otimes-3) to[in=-90,out=90] (circ.-45);
\draw (otimes-2) to[in=-90,out=90] (otimes-3.-135);
\draw (ev-a) to[in=-90,out=90] (otimes-2.-135);
\draw (j-a-2) to[in=-90,out=90] (otimes-2.-45);
\draw (j-Fv) to[in=-90,out=90] (otimes-3.-45);
\draw (circ) to[in=-90,out=90] (a-aFv);
\end{tikzpicture}
.
$$
We are now in the position to use the interchange relation again, obtaining
$$
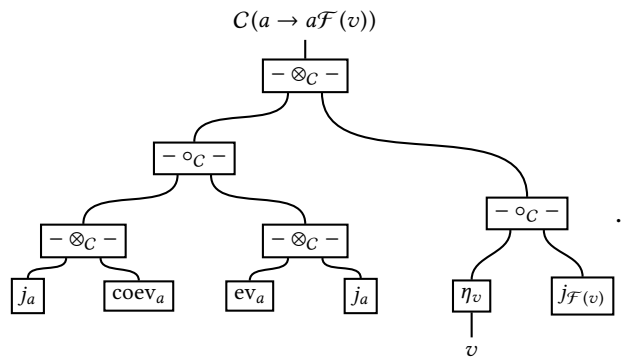
\begin{tikzpicture}[baseline=50,smallstring]
\node (1-1) at (2,0) {$\iV$};
\node (1-2) at (4,0) {$\iV$};
\node (1-3) at (6,0) {$\iV$};
\node (1-4) at (8,0) {$\iV$};
\node (v) at (10,0) {$v$};
\node (1-5) at (12,0) {$\iV$};
\node[draw,rectangle] (j-a-1) at (2,1) {$j_a$};
\node[draw,rectangle] (coev-a) at (4,1) {$\coev_{a}$};
\node[draw,rectangle] (ev-a) at (6,1) {$\ev_a$};
\node[draw,rectangle] (j-a-2) at (8,1) {$j_{a}$};
\node[draw,rectangle] (eta-v) at (10,1) {$\eta_{v}$};
\node[draw,rectangle] (j-Fv) at (12,1) {$j_{\cF(v)}$};
\node[draw,rectangle] (otimes-1) at (3,2) {$-\otimes_\cC -$};
\node[draw,rectangle] (otimes-2) at (7,2) {$-\otimes_\cC -$};
\node[draw,rectangle] (circ-L) at (5,3.5) {$-\circ_\cC -$};
\node[draw,rectangle] (circ-R) at (11,2.5) {$-\circ_\cC -$};
\node[draw,rectangle] (otimes) at (7,5) {$-\otimes_\cC -$};
\node (a-aFv) at (7,6) {$\cC(a\to a\cF(v))$};
\draw (v) to[in=-90,out=90] (eta-v);
\draw (eta-v) to[in=-90,out=90] (circ-R.-135);
\draw (j-Fv) to[in=-90,out=90] (circ-R.-45);
\draw (circ-R) to[in=-90,out=90] (circ.-45);
\draw (j-a-1) to[in=-90,out=90] (otimes-1.-135);
\draw (coev-a) to[in=-90,out=90] (otimes-1.-45);
\draw (ev-a) to[in=-90,out=90] (otimes-2.-135);
\draw (j-a-2) to[in=-90,out=90] (otimes-2.-45);
\draw (otimes-1) to[in=-90,out=90] (circ-L.-135);
\draw (otimes-2) to[in=-90,out=90] (circ-L.-45);
\draw (circ-L) to[in=-90,out=90] (otimes.-135);
\draw (otimes) to[in=-90,out=90] (a-aFv);
\end{tikzpicture}
.
$$
Now using the zig-zag relation and simplifying, we have exactly $(j_a\eta_v)\circ(-\otimes_\cC-)$.
\end{proof}

\noindent
\textbf{Proposition \ref{prop:MateOfComposition}}
\textit{
The mate of the composition map $(-\circ-): \cC(a\to b)\cC(b\to c) \to \cC(a\to c)$ under Adjunction \eqref{eq:La-Adjunction}
with $v= \cC(a\to b)\cC(b\to c)$ and $b=c$ is given by
$$
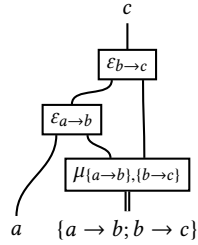
\begin{tikzpicture}[baseline=50,smallstring]
\node (a) at (0,0) {$a$};
\node (F) at (2,0) {$\{a\to b; b\to c\}$};
\node[draw,rectangle] (m) at (2,1) {$\mu_{\{a\to b\}, \{b\to c\}} $};
\draw[double] (F) to[in=-90,out=90] (m.270);
\node[draw,rectangle] (ev-a-b) at (1,2) {$\varepsilon_{a\to b}$};
\draw (a) to[in=-90,out=90] (ev-a-b.-135);
\draw (m.135) to[in=-90,out=90] (ev-a-b.-45);
\node[draw,rectangle] (ev-b-c) at (2,3) {$\varepsilon_{b\to c}$};
\draw (m.45) to[in=-90,out=90] (ev-b-c.-45);
\draw (ev-a-b) to[in=-90,out=90] (ev-b-c.-135);
\node (c) at (2,4) {$c$};
\draw (ev-b-c) to[in=-90,out=90] (c);
\end{tikzpicture}
.
$$
}
\begin{proof}
By Corollary \ref{cor:MateUsingTrace}, the mate of 
the diagram in the statement of Proposition \ref{prop:MateOfComposition} is given by
$$
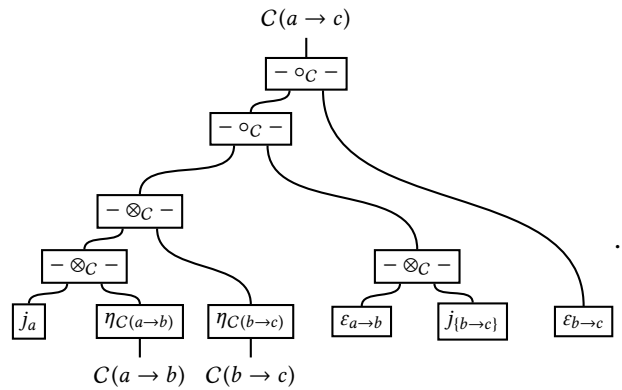
\begin{tikzpicture}[baseline=50,smallstring]
\node (1-1) at (0,0) {$\iV$};
\node (a-b) at (2,0) {$\cC(a\to b)$};
\node (b-c) at (4,0) {$\cC(b\to c)$};
\node (1-2) at (6,0) {$\iV$};
\node (1-3) at (8,0) {$\iV$};
\node (1-4) at (10,0) {$\iV$};
\node[draw,rectangle] (j-a) at (0,1) {$j_a$};
\node[draw,rectangle] (eta-a-b) at (2,1) {$\eta_{\cC(a\to b)}$};
\node[draw,rectangle] (eta-b-c) at (4,1) {$\eta_{\cC(b\to c)}$};
\node[draw,rectangle] (ev-a-b) at (6,1) {$\varepsilon_{a\to b}$};
\node[draw,rectangle] (j-b-c) at (8,1) {$j_{\{b\to c\}}$};
\node[draw,rectangle] (ev-b-c) at (10,1) {$\varepsilon_{b\to c}$};
\node[draw,rectangle] (otimes-1) at (1,2) {$-\otimes_\cC -$};
\node[draw,rectangle] (otimes-2) at (2,3) {$-\otimes_\cC -$};
\node[draw,rectangle] (otimes-3) at (7,2) {$-\otimes_\cC -$};
\node[draw,rectangle] (circ-1) at (4,4.5) {$-\circ_\cC -$};
\node[draw,rectangle] (circ-2) at (5,5.5) {$-\circ_\cC -$};
\node (a-c) at (5,6.5) {$\cC(a\to c)$};
%
\draw (a-b) to[in=-90,out=90] (eta-a-b);
\draw (b-c) to[in=-90,out=90] (eta-b-c);
%
\draw (j-a) to[in=-90,out=90] (otimes-1.-135);
\draw (eta-a-b) to[in=-90,out=90] (otimes-1.-45);
\draw (otimes-1) to[in=-90,out=90] (otimes-2.-135);
\draw (eta-b-c) to[in=-90,out=90] (otimes-2.-45);
\draw (ev-a-b) to[in=-90,out=90] (otimes-3.-135);
\draw (j-b-c) to[in=-90,out=90] (otimes-3.-45);
\draw (otimes-3) to[in=-90,out=90] (circ-1.-45);
\draw (otimes-2) to[in=-90,out=90] (circ-1.-135);
\draw (circ-1) to[in=-90,out=90] (circ-2.-135);
\draw (ev-b-c) to[in=-90,out=90] (circ-2.-45);
\draw (circ-2) to[in=-90,out=90] (a-c);
\end{tikzpicture}
.
$$
We may apply the interchange relation to obtain the map
$$
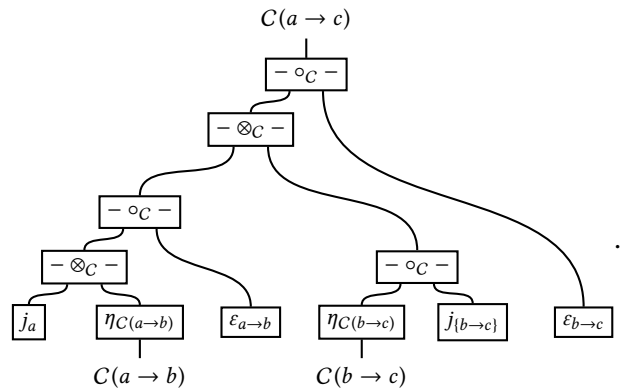
\begin{tikzpicture}[baseline=50,smallstring]
\node (1-1) at (0,0) {$\iV$};
\node (a-b) at (2,0) {$\cC(a\to b)$};
\node (b-c) at (6,0) {$\cC(b\to c)$};
\node (1-2) at (4,0) {$\iV$};
\node (1-3) at (8,0) {$\iV$};
\node (1-4) at (10,0) {$\iV$};
\node[draw,rectangle] (j-a) at (0,1) {$j_a$};
\node[draw,rectangle] (eta-a-b) at (2,1) {$\eta_{\cC(a\to b)}$};
\node[draw,rectangle] (eta-b-c) at (6,1) {$\eta_{\cC(b\to c)}$};
\node[draw,rectangle] (ev-a-b) at (4,1) {$\varepsilon_{a\to b}$};
\node[draw,rectangle] (j-b-c) at (8,1) {$j_{\{b\to c\}}$};
\node[draw,rectangle] (ev-b-c) at (10,1) {$\varepsilon_{b\to c}$};
\node[draw,rectangle] (otimes-1) at (1,2) {$-\otimes_\cC -$};
\node[draw,rectangle] (circ-1) at (2,3) {$-\circ_\cC -$};
\node[draw,rectangle] (circ-2) at (7,2) {$-\circ_\cC -$};
\node[draw,rectangle] (otimes-2) at (4,4.5) {$-\otimes_\cC -$};
\node[draw,rectangle] (circ-3) at (5,5.5) {$-\circ_\cC -$};
\node (a-c) at (5,6.5) {$\cC(a\to c)$};
%
\draw (a-b) to[in=-90,out=90] (eta-a-b);
\draw (b-c) to[in=-90,out=90] (eta-b-c);
%
\draw (j-a) to[in=-90,out=90] (otimes-1.-135);
\draw (eta-a-b) to[in=-90,out=90] (otimes-1.-45);
\draw (otimes-1) to[in=-90,out=90] (circ-1.-135);
\draw (eta-b-c) to[in=-90,out=90] (circ-2.-135);
\draw (ev-a-b) to[in=-90,out=90] (circ-1.-45);
\draw (j-b-c) to[in=-90,out=90] (circ-2.-45);
\draw (otimes-3) to[in=-90,out=90] (otimes-2.-45);
\draw (otimes-2) to[in=-90,out=90] (circ-3.-135);
\draw (circ-1) to[in=-90,out=90] (otimes-2.-135);
\draw (ev-b-c) to[in=-90,out=90] (circ-3.-45);
\draw (circ-3) to[in=-90,out=90] (a-c);
\end{tikzpicture}
.
$$
We may now apply Corollary \ref{cor:Simplify-j-eta-ev} and simplify to obtain
$$
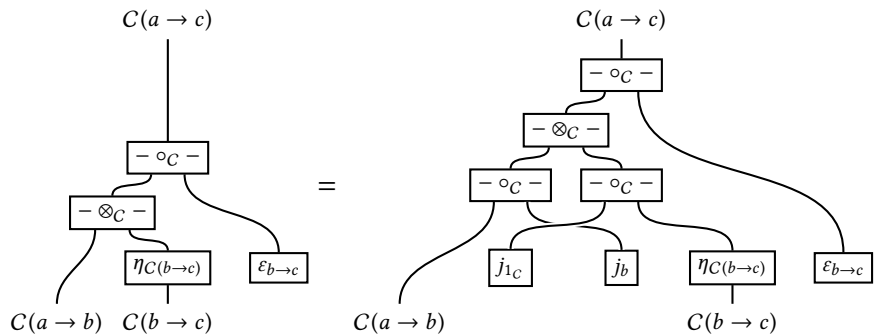
\begin{tikzpicture}[baseline=50,smallstring]
\node (a-b) at (0,0) {$\cC(a\to b)$};
\node (b-c) at (2,0) {$\cC(b\to c)$};
\node (1) at (4,0) {$\iV$};
\node[draw,rectangle] (eta-b-c) at (2,1) {$\eta_{\cC(b\to c)}$};
\node[draw,rectangle] (ev-b-c) at (4,1) {$\varepsilon_{b\to c}$};
\node[draw,rectangle] (otimes) at (1,2) {$-\otimes_\cC -$};
\node[draw,rectangle] (circ) at (2,3) {$-\circ_\cC -$};
\node (a-c) at (2,5.5) {$\cC(a\to c)$};
\draw (b-c) to[in=-90,out=90] (eta-b-c);
\draw (a-b) to[in=-90,out=90] (otimes.-135);
\draw (eta-b-c) to[in=-90,out=90] (otimes.-45);
\draw (otimes-1) to[in=-90,out=90] (circ.-135);
\draw (ev-b-c) to[in=-90,out=90] (circ.-45);
\draw (circ) to[in=-90,out=90] (a-c);
\end{tikzpicture}
=
\begin{tikzpicture}[baseline=50,smallstring]
\node (a-b) at (0,0) {$\cC(a\to b)$};
\node (1-1) at (2,0) {$\iV$};
\node (1-2) at (4,0) {$\iV$};
\node (b-c) at (6,0) {$\cC(b\to c)$};
\node (1-3) at (8,0) {$\iV$};
\node[draw,rectangle] (j-1) at (2,1) {$j_{1_\cC}$};
\node[draw,rectangle] (j-b) at (4,1) {$j_b$};
\node[draw,rectangle] (eta-b-c) at (6,1) {$\eta_{\cC(b\to c)}$};
\node[draw,rectangle] (ev-b-c) at (8,1) {$\varepsilon_{b\to c}$};
\node[draw,rectangle] (circ-L) at (2,2.5) {$-\circ_\cC -$};
\node[draw,rectangle] (circ-R) at (4,2.5) {$-\circ_\cC -$};
\node[draw,rectangle] (otimes) at (3,3.5) {$-\otimes_\cC -$};
\node[draw,rectangle] (circ) at (4,4.5) {$-\circ_\cC -$};
\node (a-c) at (4,5.5) {$\cC(a\to c)$};
\draw (b-c) to[in=-90,out=90] (eta-b-c);
\draw (a-b) to[in=-90,out=90] (circ-L.-135);
\draw[knot] (j-b) to[in=-90,out=90] (circ-L.-45);
\draw[knot] (j-1) to[in=-90,out=90] (circ-R.-135);
\draw (eta-b-c) to[in=-90,out=90] (circ-R.-45);
\draw (circ-L) to[in=-90,out=90] (otimes.-135);
\draw (circ-R) to[in=-90,out=90] (otimes.-45);
\draw (otimes) to[in=-90,out=90] (circ.-135);
\draw (ev-b-c) to[in=-90,out=90] (circ.-45);
\draw (circ) to[in=-90,out=90] (a-c);
\end{tikzpicture}$$
since we may add identity elements to the diagram at no cost.
Now applying the interchange relation and using associativity of composition, we obtain
$$
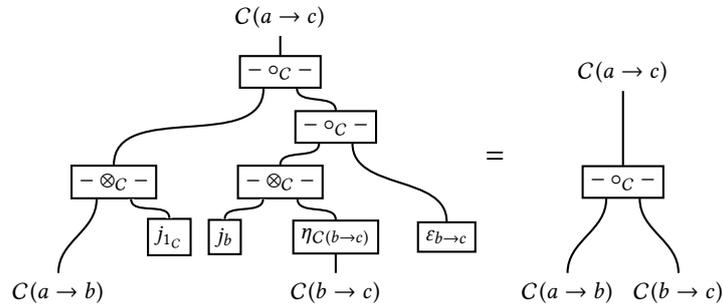
\begin{tikzpicture}[baseline=50,smallstring]
\node (a-b) at (0,0) {$\cC(a\to b)$};
\node (1-1) at (2,0) {$\iV$};
\node (1-2) at (3,0) {$\iV$};
\node (b-c) at (5,0) {$\cC(b\to c)$};
\node (1-3) at (7,0) {$\iV$};
\node[draw,rectangle] (j-1) at (2,1) {$j_{1_\cC}$};
\node[draw,rectangle] (j-b) at (3,1) {$j_b$};
\node[draw,rectangle] (eta-b-c) at (5,1) {$\eta_{\cC(b\to c)}$};
\node[draw,rectangle] (ev-b-c) at (7,1) {$\varepsilon_{b\to c}$};
\node[draw,rectangle] (otimes-L) at (1,2) {$-\otimes_\cC -$};
\node[draw,rectangle] (otimes-R) at (4,2) {$-\otimes_\cC -$};
\node[draw,rectangle] (circ-1) at (5,3) {$-\circ_\cC -$};
\node[draw,rectangle] (circ) at (4,4) {$-\circ_\cC -$};
\node (a-c) at (4,5) {$\cC(a\to c)$};
\draw (b-c) to[in=-90,out=90] (eta-b-c);
\draw (a-b) to[in=-90,out=90] (otimes-L.-135);
\draw[knot] (j-b) to[in=-90,out=90] (otimes-R.-135);
\draw[knot] (j-1) to[in=-90,out=90] (otimes-L.-45);
\draw (eta-b-c) to[in=-90,out=90] (otimes-R.-45);
\draw (otimes-L) to[in=-90,out=90] (circ.-135);
\draw (otimes-R) to[in=-90,out=90] (circ-1.-135);
\draw (circ-1) to[in=-90,out=90] (circ.-45);
\draw (ev-b-c) to[in=-90,out=90] (circ-1.-45);
\draw (circ) to[in=-90,out=90] (a-c);
\end{tikzpicture}
=
\begin{tikzpicture}[baseline=50,smallstring]
\node (a-b) at (0,0) {$\cC(a\to b)$};
\node (b-c) at (2,0) {$\cC(b\to c)$};
\node[draw,rectangle] (circ) at (1,2) {$-\circ_\cC -$};
\node (a-c) at (1,4) {$\cC(a\to c)$};
\draw (a-b) to[in=-90,out=90] (circ.-135);
\draw (b-c) to[in=-90,out=90] (circ.-45);
\draw (circ) to[in=-90,out=90] (a-c);
\end{tikzpicture}
$$
by applying Corollary \ref{cor:Simplify-j-eta-ev} again.
\end{proof}

\subsection*{Acknowledgements}
This project began at the 2016 Trimester on von Neumann algebras at the Hausdorff Institute for Mathematics (HIM) and
the 2016 Noncommutative Geometry and Operator Algebras Spring Institute.
The authors would like to thank HIM and the organizers for their support during this time. 
This project continued at the 2016 AIM SQuaRE on Classifying fusion categories.
The authors would like to thank AIM for their support and hospitality.

The authors would like to thank 
Marcel Bischoff,
Andr\'{e} Henriques, 
Corey Jones,
Steve Lack,
Zhengwei Liu,
Julia Plavnik,
Noah Snyder,
Ross Street,
James Tener, and
Kevin Walker
for helpful conversations.
Scott Morrison was partially supported by Australian Research Council grants
`Subfactors and symmetries' DP140100732 and `Low dimensional categories' DP160103479.
David Penneys was partially supported by NSF DMS grant 1500387/1655912.


\setlength{\bibitemsep}{0pt}
\raggedright
\printbibliography

\end{document}